\documentclass[12pt]{article}

\usepackage[a4paper,includeheadfoot,margin=2.54cm]{geometry} % replace a4wide
\usepackage[utf8]{inputenc}
\usepackage{cite} % citations in right order
\usepackage{authblk}
\usepackage{amsmath,amsfonts,amssymb}
\usepackage{bm}
\usepackage{color}
\usepackage[english]{babel}
\usepackage{graphicx}%
\usepackage{todonotes}
\usepackage{subcaption}
\usepackage{booktabs}
\usepackage{color,soul}
\usepackage[margin=10pt,font=small,labelfont=bf]{caption}
\usepackage{stmaryrd}
\usepackage{mathtools}

\usepackage[plainpages=false]{hyperref}
\hypersetup{bookmarks=true,
		bookmarksopen=false,
		pdfborder={0 0 0},
		pdftitle={Cut Finite Element Methods for Linear Elasticity Problems},
		pdfauthor={P. Hansbo, M.G. Larson, K. Larsson},
		pdfcreator={P. Hansbo, M.G. Larson, K. Larsson},
		colorlinks=true,
		linkcolor=blue,
		citecolor=red	
	}

\numberwithin{equation}{section}

\newcommand{\bfzero}{\boldsymbol{0}}
\newcommand{\bfa}{\boldsymbol{a}}
\newcommand{\bfb}{\boldsymbol{b}}
\newcommand{\bfx}{\boldsymbol{x}}

\newcommand{\bfu}{\boldsymbol{u}}
\newcommand{\bfv}{\boldsymbol{v}}
\newcommand{\bfw}{\boldsymbol{w}}
\newcommand{\bff}{\boldsymbol{f}}
\newcommand{\bfg}{\boldsymbol{g}}
\newcommand{\bfn}{\boldsymbol{n}}

\newcommand{\bfV}{\boldsymbol{V}}

\newcommand{\bfH}{\boldsymbol{H}}

\newcommand{\bfRM}{\boldsymbol{RM}}
\newcommand{\bfsig}{\boldsymbol{\sigma}}
\newcommand{\bfeps}{\boldsymbol{\epsilon}}

\newcommand{\bfpi}{\boldsymbol{\pi}}
\newcommand{\bfvarphi}{\boldsymbol{\varphi}}

\newcommand{\tn}{\vert\kern-0.25ex\vert\kern-0.25ex\vert}

\newcommand{\bndD}{{\partial \Omega_D}}
\newcommand{\bndN}{{\partial \Omega_N}}

\newcommand{\mcK}{\mathcal{K}}

\newcommand{\mcF}{\mathcal{F}}

\newcommand{\bft}{\boldsymbol{t}}
\newcommand{\bfP}{\boldsymbol{P}}

\newcommand{\bfp}{\boldsymbol{p}}
\newcommand{\bfy}{\boldsymbol{y}}

\newcommand{\mcP}{\mathcal{P}}

\newcommand{\IR}{\mathbb{R}}
\newtheorem{lem}{Lemma}[section]

\newtheorem{thm}{Theorem}[section]
\newtheorem{rem}{Remark}[section]

\newenvironment{proof}{\noindent \newline {\bf Proof.}}
{\hfill \mbox{\fbox{} } \newline}

\newcommand{\bbA}{{\bf A}}
\newcommand{\bbM}{{\bf M}}
\newcommand{\hatv}{\widehat{\bfv}}
\newcommand{\hatw}{\widehat{\bfw}}

\makeindex             % used for the subject index
\begin{document}

\title{\bf Cut Finite Element Methods\\ for Linear Elasticity Problems\thanks{This research was supported in part by 
the Swedish Foundation for Strategic Research Grant No.\ AM13-0029, 
the Swedish Research Council Grant No. 2013-4708, and 
the Swedish strategic research programme eSSENCE.}}

\date{}
\author[1]{Peter Hansbo\footnote{peter.hansbo@ju.se}}
\author[2]{Mats G. Larson\footnote{mats.larson@umu.se}}
\author[2]{Karl Larsson\footnote{karl.larsson@umu.se}}
\affil[1]{\it Department of Mechanical Engineering, J\"onk\"oping University, SE-551\,11 J\"onk\"oping, Sweden}
\affil[2]{\it Department of Mathematics and Mathematical Statistics, Ume{\aa}~University, SE-901\,87 Ume{\aa}, Sweden}
\maketitle

% starred version for online content
\abstract{We formulate a cut finite element method for linear elasticity 
based on higher order elements on a fixed background mesh. Key to the method 
is a stabilization term which provides control of the jumps in the derivatives 
of the finite element functions across faces in the vicinity of the boundary. 
We then develop the basic theoretical results including error estimates and 
estimates of the condition number of the mass and stiffness matrices. We apply 
the method to the standard displacement problem, the frequency response problem, 
and the eigenvalue problem. We present several numerical examples including 
studies of thin bending dominated structures relevant for engineering applications. 
Finally, we develop a cut finite element method for fibre reinforced materials 
where the fibres are modeled as a superposition of a truss and a Euler-Bernoulli
beam. The beam model leads to a fourth order problem which we discretize using the restriction of the bulk finite element space to the fibre together with a 
continuous/discontinuous finite element formulation. Here the bulk material 
stabilizes the problem and it is not necessary to add additional stabilization 
terms.
}

\pagebreak

\tableofcontents

\pagebreak

\section{Introduction}
\paragraph{The CutFEM Paradigm}
In this contribution we develop a cut finite element method (CutFEM) 
for linear elasticity. The CutFEM paradigm is based on the following 
main concepts
\begin{itemize}
\item The computational domain is considered as a subset of polygonal 
domain which is equipped with a so called background mesh. The computational 
domain is represented on the background mesh and its boundary and interior 
interfaces are allowed to cut through the elements in an arbitrary fashion. 
On the background mesh there is a finite element space.
\item The active mesh is the set of elements that intersects the computational domain. The finite element space on the active mesh is taken to be the restriction 
of the finite element space on the background mesh.
\item The finite element method is based on a variational formulation on the computational domain, which includes weak enforcement of all interface and boundary conditions. 
\item Stabilization terms are added to handle the loss of control emanating 
from the presence of cut elements. The stabilization terms essentially restores 
the properties of standard finite element formulations.
\end{itemize}
The CutFEM is designed in such a way that we can establish the following theoretical results
\begin{itemize}
\item Optimal order a priori error estimates.
\item Optimal order bounds on the condition number of the stiffness and mass matrices.
\end{itemize}
These results hold also for higher order polynomials provided 
the geometric representation of the domain and quadrature computations 
are sufficiently accurate. See the overview article \cite{BurClaHan15} and the 
references therein for further details.

\paragraph{Contributions}

In this contribution we focus on applications of 
CutFEM in solid mechanics. In particular we consider:
\begin{itemize}
\item Formulation of CutFEM for linear elasticity including an easy to access 
account of the  theoretical results. We also provide extensive numerical results illustrating the performance of the method for test problems of practical interest. 
We study the displacement problem, the frequency response problem, and the 
eigenvalue problem. We also study the behavior of the method for thin structures 
which are common in engineering applications. 
\item  Employing the CutFEM paradigm to the study of reinforced bulk 
materials, where the reinforcement consists of fibres, which are modeled 
as a superposition of beams and trusses. The beam model leads to a fourth 
order problem which we discretize using techniques from continuous/discontinuous Galerkin methods and the finite element space is the restriction of the bulk 
finite element space to the fibre. This approach requires a continuous finite element 
space consisting of at least second order piecewise polynomials. In its simplest form, we simply superimpose the additional stiffness, from the fibres to the bulk stiffness, by summing the bilinear forms for the fibres to the bulk material 
bilinear form and using the same finite element space. Note that here the the bulk material provides sufficient stability and no further stabilization is necessary. 
We provide illustrating numerical examples.
\end{itemize} 

\paragraph{Previous Work} 

The possibility of discretizing partial differential equations on the cut part of elements was first proposed by Olshanskii, Reusken, and Grande 
\cite{OlReGr09}, and has been further developed by this group in, e.g., \cite{OlReXu12,GrOlRe15,GrRe16}.
The work presented here builds on the CutFEM developments by the authors and their coworkers concerning stabilization of cut elements in \cite{BeBuHa09,BuHa10,BuHa12,MasLarLog14,multipatch-cutfem} and concerning beam, plate, and membrane formulations in \cite{HaLaLa14,CeHaLa16,HaLa17}.
Related work has been presented by Fries et al. \cite{FrOmScSt17,FrOm16} and by Rank and coworkers \cite{PaDuRa07,BaVaRa13}, see also \cite{ScRu15,PrVeZwBr17}.

% \todo{Add in!}

%
%\subparagraph{Contributions} We develop a theoretical framework where we 
%first show that the added stabilization yields forms that have the 
%desired continuity and coercivity properties with respect to relevant norms. With these basic results at hand we may derive error estimates 
%for the model problems using standard abstract approaches that rely 
%on Galerkin orthogonality, coercivity and continuity of the forms, 
%and an approximation property. We also present extensive numerical 
%investigations of the accuracy and robustness of CutFEM. We consider applications that are motivated by engineering applications 
%including standard load cases, the eigenvalue problem, and the 
%frequency response problem. We study the accuracy of the method and
%the dependence of the solution on the position of the domain in the background grid. We find that for relevant problem in 2D there is no 
%loss of accuracy compared to a standard method. As in the standard 
%case quadratic elements are required to obtain good solutions in bending dominated cases. 

\paragraph{Outline} In Section 2 we formulate the 
governing equations, the different model problems, and the finite 
element method, in Section 3 we present the theoretical results 
including error estimates and bounds for the condition number 
of the mass and stiffness matrix, in Section 4 we present extensive 
numerical results, in Section 5 we develop a cut finite element method 
for fibre reinforced materials.

\section{Linear Elasticity and the Cut Finite Element Method}

\subsection{Linear Elasticity and Model Problems}\label{sec:linear-elasticity}

Let $\Omega$ be a domain in $\mathbb{R}^d$, $d=2$ or $3$, with 
boundary $\partial \Omega = \bndD \cup \bndN$, $\bndD \cap \bndN=\emptyset$, 
and exterior unit normal $\bfn$. We consider the following problems: 
\begin{itemize}
\item (\emph{The Displacement Problem})\, Find the displacement
$\bfu:\Omega \rightarrow \IR^d$ such that
\begin{subequations}\label{eq:standard}
\begin{alignat}{2}\label{eq:standard-a}
-\bfsig(\bfu)\cdot \nabla &= \bff \qquad 
&& \text{in $\Omega$}
\\ \label{eq:standard-b}
\bfsig(\bfu) \cdot \bfn &= \bfg_N \qquad && 
\text{on $\bndN$}
\\ \label{eq:standard-c}
\bfu &= \bfg_D \qquad && 
\text{on $\bndD$}
\end{alignat}
\end{subequations}
where the stress and strain tensors are defined by 
\begin{equation}
\bfsig(\bfu) = 2\mu \bfeps(\bfu) + \lambda\text{tr}(\bfeps(\bfu)),
\qquad 
\bfeps(\bfu) = \frac{1}{2}\Big( \bfu\otimes\nabla + \nabla \otimes \bfu \Big)
\end{equation}
with Lam\'e parameters $\lambda$ and $\mu$,
$\bff$, $\bfg_N$, $\bfg_D$ are given data, and $\bfa \otimes \bfb$ 
is the tensor product of vectors $\bfa$ and $\bfb$ with elements
$(\bfa \otimes \bfb)_{ij} = a_i b_j$.

\item (\emph{The Frequency Response Problem})\, Given a frequency $\omega$ find the 
displacement 
$\bfu:\Omega \rightarrow \IR^d$, 
such that
\begin{subequations}\label{eq:frequency}
\begin{alignat}{2}\label{eq:frequency-a}
-\bfsig(\bfu)\cdot \nabla - \omega^2 \bfu &= \bff \qquad 
&& \text{in $\Omega$}
\\ \label{eq:frequency-b}
\bfsig(\bfu) \cdot \bfn &= \bfg_N \qquad && 
\text{on $\bndN$}
\\ \label{eq:frequency-c}
\bfu &= \bfzero \qquad && 
\text{on $\bndD$}
\end{alignat}
\end{subequations}
Here we often compute the energy 
\begin{equation}
E(\omega) = \int_\Omega \bfsig(\bfu(\omega)):\bfeps(\bfu(\omega))\label{eq:energyE}
\end{equation}
for a range of values of the parameter $\omega$.

\item (\emph{The Eigenvalue Problem})\, Find the displacement 
eigenvector 
$\bfu:\Omega \rightarrow \IR^d$ and eigenvalue 
$\lambda \in \IR$, such that
\begin{subequations}\label{eq:evp}
\begin{alignat}{2}\label{eq:evp-a}
-\bfsig(\bfu)\cdot \nabla - \lambda \bfu &= \bfzero \qquad 
&& \text{in $\Omega$}
\\ \label{eq:evp-b}
\bfsig(\bfu) \cdot \bfn &= \bfzero \qquad && 
\text{on $\bndN$}
\\ \label{eq:evp-c}
\bfu &= \bfzero \qquad && 
\text{on $\bndD$}
\end{alignat}
\end{subequations}
This problem has an infinite sequence of solutions $(\bfu_i,\lambda_i)$ 
where $0\leq \lambda_1 <\lambda_2 <\dots$, such that 
$\lambda_i\rightarrow \infty$ and there are no accumulation points. There 
may be multiple eigenvalues, which is common in applications due 
to symmetries of the domain.
\end{itemize}

\subsection{Weak Formulations}

Let $\bfV_{\bfg}=\{\bfv \in \bfH^1(\Omega) : \text{$\bfv=\bfg$ on $\partial\Omega_{D}$}\}$, 
and define the forms
\begin{align}
m(\bfv,\bfw) &= (\bfv,\bfw)_\Omega
\\
a(\bfv,\bfw) &=  2\mu (\bfeps(\bfv),\bfeps(\bfw))_\Omega 
+ \lambda (\text{tr}(\bfeps(\bfv)),\text{tr}(\bfeps(\bfw)))_\Omega
\end{align}
Here and below we use the standard notation $H^s(\omega)$ 
for the $\IR^d$ valued Sobolev space of order $s$ on the set 
$\omega$, with norm $\|\cdot\|_{H^s(\omega)}$, for $s=0$ 
we use the notation $L^2(\omega)$ and denote the scalar 
product by $(\cdot, \cdot)_\omega$, and the norm by 
$\|\cdot \|_\omega$.

The weak formulations of problems 
\eqref{eq:standard}, \eqref{eq:frequency}, and 
\eqref{eq:evp} take the form:
\begin{itemize}
\item (\emph{The Displacement Problem})\, Find $\bfu \in \bfV_{\bfg_D}$ such that
\begin{equation}\label{eq:standard-weak}
a(\bfu,\bfv) = l(\bfv),  
\qquad \forall \bfv \in \bfV_0 
\end{equation}
where the linear form on the right hand side is defined by
\begin{align}
l(\bfv) &= (\bff,\bfv)_\Omega + (\bfg_N,\bfv)_{\partial \Omega_N} 
\end{align}
\item (\emph{The Frequency Response Problem})\, Given $\omega\in\IR$ find $\bfu \in \bfV_0$ such that 
\begin{equation}\label{eq:frequency-weak}
a(\bfu,\bfv) - \omega^2 m(\bfu,\bfv) = l(\bfv),  
\qquad \forall \bfv \in \bfV_0 
\end{equation}

\item (\emph{The Eigenvalue Problem})\, Find $(\bfu,\lambda) \in \bfV_0 \times \IR$ such that 
\begin{equation}\label{eq:evp-weak}
a(\bfu,\bfv) - \lambda m(\bfu,\bfv) = 0,  
\qquad \forall \bfv \in \bfV_0 
\end{equation}
\end{itemize}

Let $\bfRM = \text{ker}(\bfeps)$ be the space of linearized 
rigid body motions, which has dimension $3$ in $\IR^2$ 
and $6$ in $\IR^3$, see \cite{BreSco08} Section 11. We assume 
for simplicity that the Dirichlet boundary condition completely 
determines the linearized rigid body part of the solution, i.e., 
if $v \in \bfRM$ and $v=0$ on $\partial \Omega_D$ then $v=0$, 
then problem (\ref{eq:standard-weak}) is well posed. More general 
situations may be considered by working in a quotient space. For 
instance, an important case in practice is a free elastic body, 
$\partial\Omega_D = \emptyset$. We then seek the solution in the 
quotient space $H^1(\Omega)/\bfRM$ and problem 
(\ref{eq:standard-weak}) is well posed if the right hand side satisfies
\begin{equation}
l_h(\bfv) = 0, \qquad \bfv \in \bfRM
\end{equation}
The frequency response problem (\ref{eq:frequency-weak}) is well 
posed if $\omega$ is not an eigenvalue.

\subsection{The Mesh and Finite Element Spaces}

\begin{figure}
\centering
\includegraphics[width=0.8\linewidth]{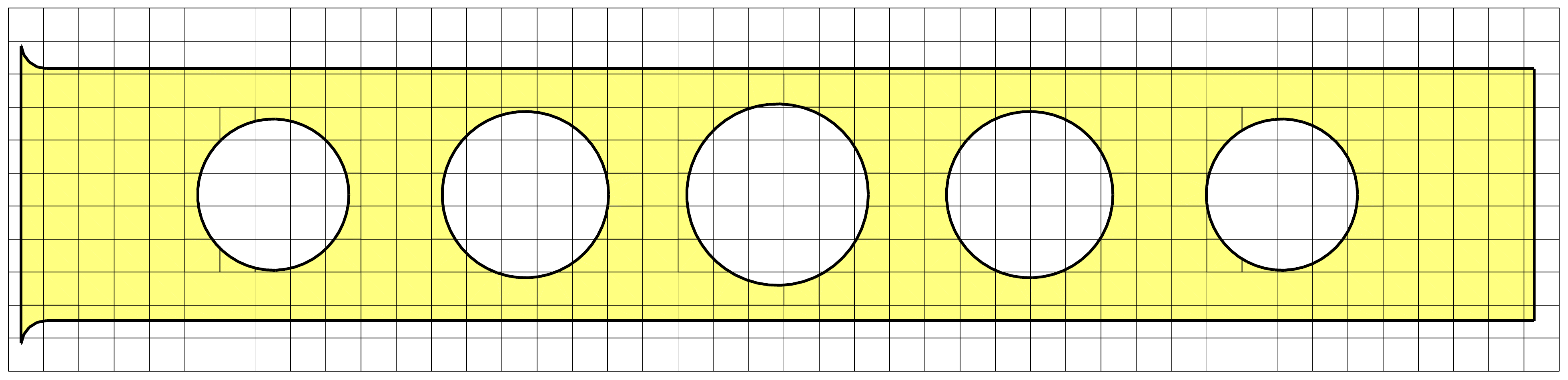}
\caption{The domain $\Omega$ and the background mesh $\mcK_{h,0}$.}\label{fig:model-mesh}
\end{figure}

\begin{figure}
\centering
\includegraphics[width=0.8\linewidth]{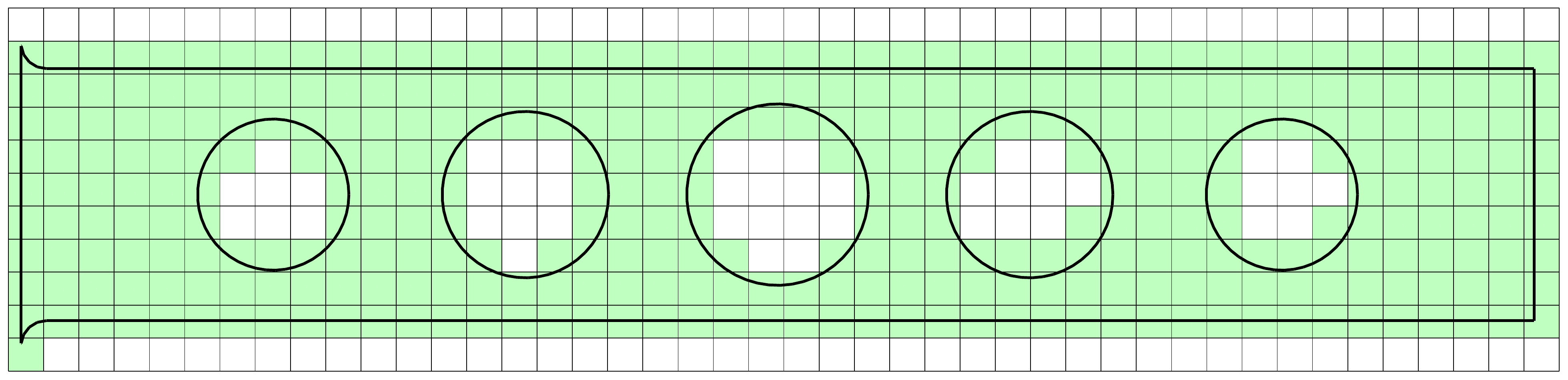}
\caption{The domain $N_h(\Omega)$ of the active mesh $\mcK_{h}$.}\label{fig:model-active}
\end{figure}

\begin{figure}
\centering
\includegraphics[width=0.8\linewidth]{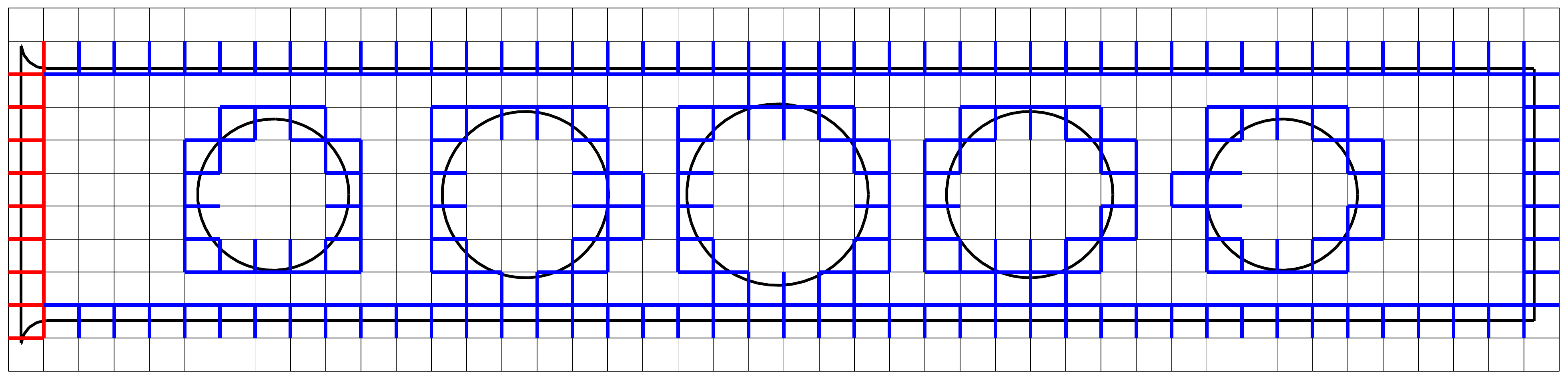}
\caption{The sets of faces used in stabilization of the method. 
Here the set $\mcF_h(\partial\Omega_D)$ associated with the 
Dirichlet boundary is shown in red and the set 
$\mcF_h(\partial\Omega)\setminus\mcF_h(\partial\Omega_D)$ 
associated with the Neumann boundary is shown in blue.}\label{fig:faces-stab}
\end{figure}

\begin{itemize}
\item Let $\Omega_0$ be polygonal domain such that 
$\Omega \subset\Omega_0 \subset \mathbb{R}^d$ and let 
$\{\mcK_{0,h}, h \in (0,h_0]\}$ be a family of quasiuniform 
partitions, with mesh parameter $h$, of $\Omega_0$ into 
shape regular elements $K$. We denote the set of faces 
in $\mcK_{h,0}$ by $\mcF_{h,0}$. We refer to $\mcK_{h,0}$ 
as the background mesh. 

\item Let
\begin{equation}
\mcK_{h} = \{K \in \mcK_{0,h} \, : \, K \cap 
\Omega
\neq \emptyset \} 
\end{equation}
be the submesh of $\mcK_{h,0}$ consisting of elements 
that intersect $\Omega$. 
We denote the set of faces 
in $\mcK_{h}$ by $\mcF_{h}$ and we refer to $\mcK_{h}$ 
as the active mesh.
The domain of the active mesh is
\begin{equation}
N_h(\Omega) = \cup_{K \in \mcK_{h}}
\end{equation}
and thus $N_h(\Omega)$ covers $\Omega$ and we obviously have 
$\Omega \subset N_h(\Omega)$.

\item Given some subset of the boundary $\Gamma\subset\partial\Omega$ let
\begin{align}
\mcK_h(\Gamma) &=
\{ K \in \mcK_h \, : \, \overline{K} \cap \Gamma \neq \emptyset \}
\\
\mcF_h(\Gamma) &=
\{ F \in \mcF_h \, : \, F \cap \partial K \neq \emptyset \, , \ K \in \mcK_h(\Gamma) \}
\end{align}
Thus $\mcF_h(\Gamma)$ is the set of interior faces belonging to elements in $\mcK_h$ that intersects~$\Gamma$.

\item Let $\bfV_{0,h}$ be the space of piecewise continuous $\IR^d$ 
valued polynomials of order $p$ defined on $\mcK_{0,h}$. Let the 
finite element space on $N_h(\Omega)$ be defined by 
\begin{equation}
\bfV_{h} = \bfV_{0,h}|_{N_{h}(\Omega)}
\end{equation}

\end{itemize}

We illustrate the computational domain $\Omega$ and the background mesh $\mcK_{h,0}$ in Figure~\ref{fig:model-mesh}, the domain of active elements $N_h(\Omega)$ in Figure~\ref{fig:model-active}, and the sets of faces $\mcK_h(\Gamma_N)$ and $\mcK_h(\Gamma_D)$ in Figure~\ref{fig:faces-stab}.

\subsection{Bilinear Forms}
\begin{itemize}

\item Let $\mcF_h(\partial \Omega)$ in be the set of interior faces that belongs to an element $K$ 
such that $K\cap \partial \Omega \neq \emptyset$, and define the stabilization form
\begin{align}\label{def:jh}
j_h(\bfv,\bfw) &= \sum_{F \in \mcF_h(\partial \Omega)} 
\sum_{l=1}^p h^{2l+1} ([D_{\bfn_F}^l \bfv],[D_{\bfn_F}^l \bfw])_F
\end{align}
where $D_{\bfn_F}^l$ is the $l$:th partial derivative in the direction 
of the normal $\bfn_F$ to the face $F\in \mcF_h(\partial \Omega)$ and $[\cdot]$ 
denotes the jump in a discontinuous function at $F$.

\item Define the stabilized forms
\begin{align}\label{def:mh}
m_{h}(\bfv,\bfw) &= m(\bfv,\bfw) + \gamma_m j_{h}(\bfv,\bfw)
\\
\label{def:ah}
a_{h}(\bfv,\bfw) &= a(\bfv,\bfw) +  \gamma_a h^{-2}  j_{h}(\bfv,\bfw)
\end{align}
where $\gamma_m,\gamma_a > 0$ are parameters.
We will see that the stabilized forms have the same general 
properties as the corresponding forms on standard (not cut) 
elements.

\item Define the stabilized Nitsche form  
\begin{align}
\label{def:Ah}
A_h(\bfv,\bfw) &= a_{h}(\bfv,\bfw) 
-(\bfsig(\bfv)\cdot \bfn ,\bfw)_{\partial \Omega_{D}}
-(\bfv, \bfsig(\bfw)\cdot \bfn)_{\partial \Omega_{D}}
+ \beta h^{-1}b_h(\bfv,\bfw)
\end{align}
where $\beta>0$ is a parameter and 
\begin{equation}
b_h(\bfv,\bfw) = 2\mu (\bfv,\bfw)_{\partial \Omega_{D}}
+ \lambda (\bfv\cdot \bfn,\bfw\cdot \bfn)_{\partial \Omega_{D}}
\end{equation}
Note that the stabilization is included in the form $a_h$, see (\ref{def:ah}).
\end{itemize}

\begin{rem}
We will see in the analysis that we may use an alternative form 
of the stabilization term in $a_h$ defined as follows. Define 
\begin{align}
j_{h,D} (\bfv,\bfw) &= \sum_{F \in \mcF_h(\partial \Omega_D)}
\sum_{l=1}^p h^{2l+1} ([D_{\bfn_F}^l \bfv],[D_{\bfn_F}^l \bfw])_F
\\
j_{h,N} (\bfv,\bfw) &= \sum_{F \in \mcF_h(\partial \Omega) \setminus \mcF_h(\partial \Omega_D)} 
\sum_{l=1}^p h^{2l+1} ([D_{\bfn_F}^l \bfv],[D_{\bfn_F}^l \bfw])_F
\end{align}
with the obvious notation, and let 
\begin{equation}\label{def:jh-alt}
a_h(v,w) = a(v,w) + \gamma_a \left( j_{h,N}(v,w) + h^{-2} j_{h,D}(v,w) \right)
\end{equation}
Thus we use a stronger stabilization only on the Dirichlet boundary, which is 
needed in the proof of the coercivity of $A_h$. The weaker control on the 
Neumann part of the boundary is sufficient to establish the bounds on the 
condition number. See Figure \ref{fig:faces-stab}. The stabilization of the form 
$m_h$ is always done using $j_h$ as defined in (\ref{def:jh}). As a rule of thumb 
weaker stabilization yields more accurate numerical results and therefore (\ref{def:jh-alt}) 
may be prefered in practice.
\end{rem}

\subsection{The Finite Element Method} 

With $m_h$ and $A_h$ defined in (\ref{def:mh}) and 
(\ref{def:Ah}) we have the finite element methods 
\begin{itemize}
\item (\emph{Static Load})
Find $\bfu_h \in \bfV_h$ such that
\begin{equation}\label{eq:standard-fem}
A_h(\bfu_h ,\bfv) = L_h(\bfv), \qquad 
\forall \bfv \in \bfV_h
\end{equation}
where the right hand side is given 
by
\begin{align}
L_h(\bfv) &= (\bff,\bfv)_{\Omega} 
+ (\bfg_N,\bfv)_{\Gamma_N}
\\&\qquad\nonumber
- (\bfg_D,\bfsig(\bfv)\cdot \bfn )_{\partial \Omega_D}
+ \beta h^{-1} b_h(\bfg_D,\bfv)_{\partial \Omega_D} 
\end{align}

\item (\emph{Frequency Response})
Given $\omega\in\IR$ find $\bfu_h \in \bfV_h$ such that 
\begin{equation}\label{eq:frequency-fem}
A_h(\bfu_h,\bfv) - \omega^2 m_h(\bfu_h,\bfv) = L_h(\bfv),  
\qquad \forall \bfv \in \bfV_h
\end{equation}

\item (\emph{Eigenvalue Problem})
Find $(\bfu_h,\lambda_h) \in \bfV_h\times \mathbb{R}$ 
such that
\begin{equation}\label{eq:evp-fem}
A_h(\bfu_h,\bfv) - \lambda_h m_h(\bfu_h,\bfv) = 0, \qquad 
\forall \bfv \in \bfV_h
\end{equation}
\end{itemize}

\section{Implementational Aspects}

\begin{figure}
\centering
\includegraphics[width=0.49\linewidth]{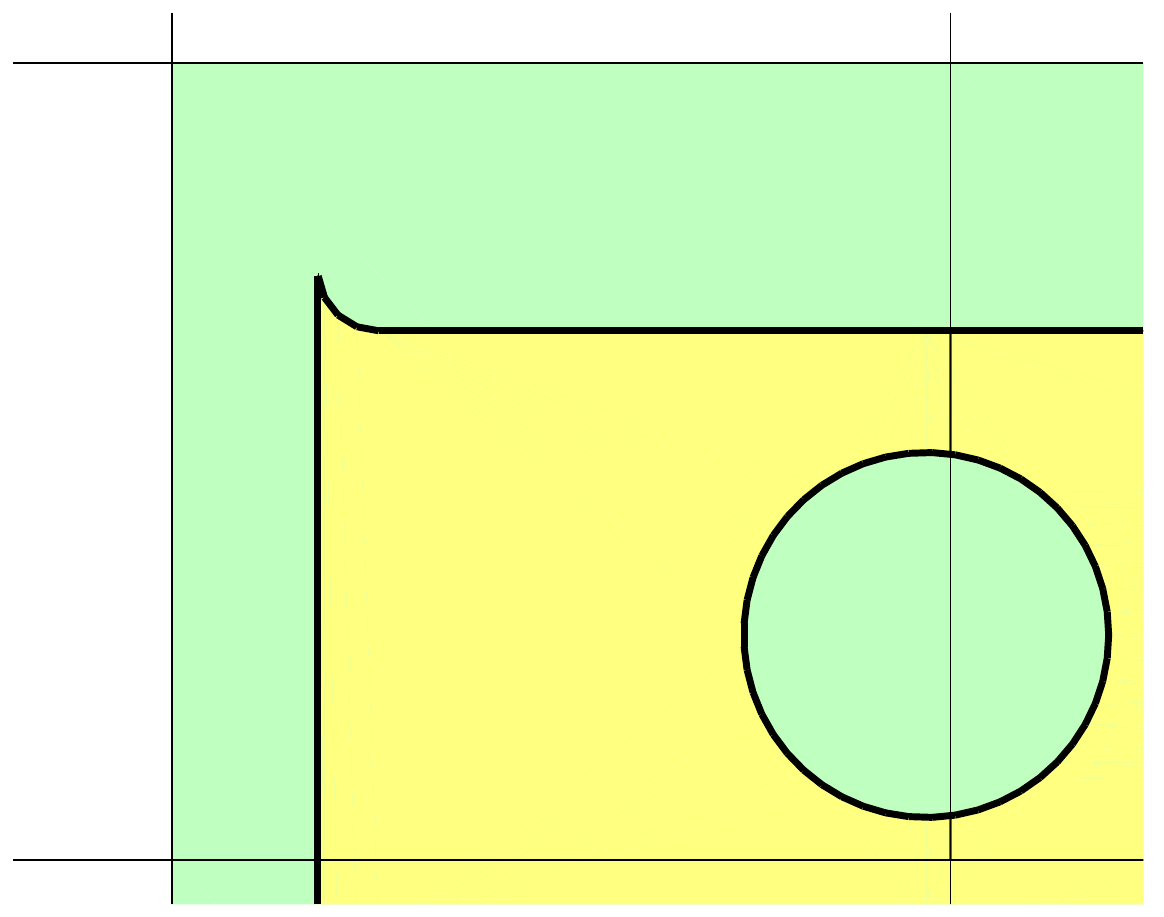}
\includegraphics[width=0.49\linewidth]{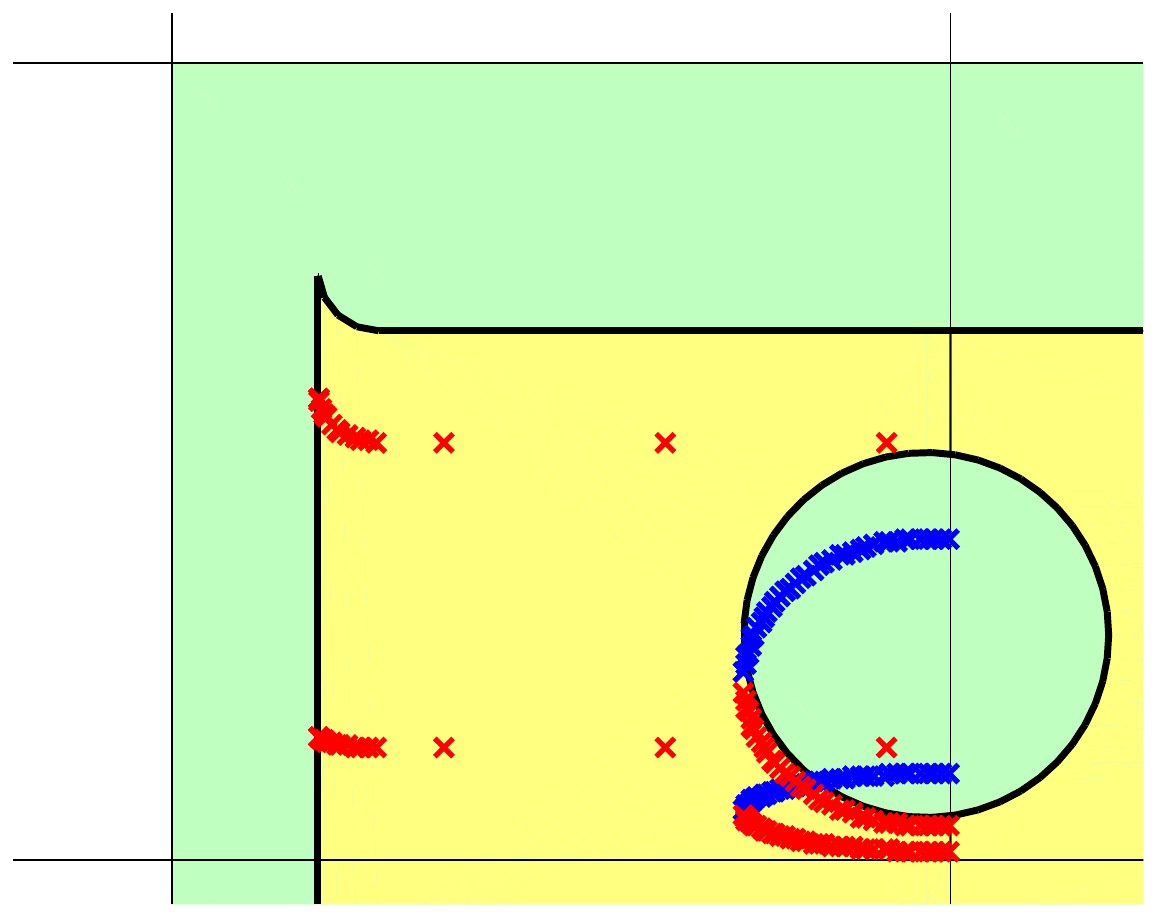}
\caption{Left: Geometry described by piecewise linear boundary representation where the round off in the top left corner consists of 5 line segments and the (complete) circle consists of 50 segments. Right: Quadrature points constructed such that tensor product polynomials of order $2$ are exactly integrated over the intersection between the element and the domain. Red points have positive weights and blue points have negative weights.}
\label{eq:cut-elm-implement}
\end{figure}

\subsection{Boundary Representation}

While it is common for fictitious domain methods to be based on implicit representations of the geometry, such as level-sets or signed distance functions, mechanical components typically are parametrically described using CAD. We therefore in the present work focus on parametric boundary representations.

For simplicity we in our implementation use a piecewise linear boundary representation albeit with a resolution which is independent of the finite element mesh. Thus, within a single element in the mesh we can represent geometric features of arbitrary complexity, see for example Figure~\ref{eq:cut-elm-implement}. 

\subsection{Quadrature}

To assemble our matrices we need to correctly integrate over the intersection between the computational domain $\Omega$ and each active element $K$. As the geometry of $\Omega\cap K$ is allowed to be arbitrarily complex we use the 2D quadrature rule described in \cite{multipatch-cutfem} for accurate integration of both full and tensor product polynomials of higher order over this intersection. This rule assumes the domain of integration is described as a number of closed loops of piecewise polynomial curves and by using the divergence theorem the integral is first posed as a boundary integral. Using the fundamental theorem of calculus it is then reformulated as a sequence of one dimensional integrals, one in each spatial dimension, which are evaluated using appropriate Gauss quadrature rules chosen based on the polynomial structure of the integrand and the polynomial order of the boundary description. Example quadrature points for integration of a tensor product polynomial of order $2$ are displayed in Figure~\ref{eq:cut-elm-implement}. Note that each boundary segment produces a number of quadrature points.

\section{Theoretical Results}

In this section we develop the basic theoretical results 
concerning stability, conditioning of the stiffness and mass 
matrices, and error estimates for the displacement problem 
(\ref{eq:standard-a}-\ref{eq:standard-c}). The main results are:
\begin{itemize} 
\item The stabilized forms $m_h$ and $a_h$ enjoy the same 
coercivity and continuity properties with respect to the proper 
norms on $N_h(\Omega)$ as corresponding standard forms 
on $\Omega$ equipped with a fitted mesh. 
\item Optimal order approximation holds in the relevant 
norms since there is a stable interpolation operator obtained 
by composing a standard Scott-Zhang interpolation operator 
with an extension operator that in a $H^s$ stable way extends a
function from $\Omega$ to a neighborhood of $\Omega$ containing 
$N_h(\Omega)$.
\item Optimal order error estimates follows from the stability 
and continuity properties combined with the approximation properties.
\item The mass and stiffness matrices have condition numbers that scale 
with the mesh size in the same way as for standard fitted meshes.
\end{itemize}

\subsection{Norms}
\begin{itemize}
\item Define the norms
\begin{equation}
\|\bfv\|^2_{m_h} = m_h(\bfv,\bfv), 
\qquad
\| \bfv \|^2_{a_h} = a_h(\bfv,\bfv),
\qquad 
\|\bfv \|^2_{j_h} = j_h(\bfv,\bfv) 
\end{equation}
We recall that $m_h$ and $a_h$ are the stabilized forms, see (\ref{def:mh}) 
and (\ref{def:ah}), so that
\begin{equation}
\|\bfv\|^2_{m_h} = \|\bfv \|^2_\Omega + \| \bfv \|^2_{j_h}, 
\qquad
\|\bfv\|^2_{a_h} = \|\nabla \bfv \|^2_\Omega + h^{-2} \| \bfv \|^2_{j_h}
\end{equation}
\item Define the discontinuous Galerkin norm associated with 
the form $A_h$,
\begin{align}
\tn \bfv \tn_h^2 &= 
\| \bfv \|^2_{a_h} 
+ 
h \|\bfsig(\bfv)\|^2_{\partial\Omega_D} 
+ 
h^{-1}\tn \bfv \tn^2_{b_h}
\end{align}
\end{itemize}

\subsection{Inverse Estimates}

The form $j_h$ provides additional control of the variation 
of functions in $\bfV_h$ close to the boundary. This additional 
control implies that certain inverse inequalities that hold for 
standard elements also hold for the stabilized forms. We have the 
following results.

\begin{lem} The following inverse estimates hold 
\begin{alignat}{2}
\label{eq:jhL2control}
\| \bfv \|_{N_h(\Omega)} 
&\lesssim \| \bfv \|_{m_{h}}, 
\qquad &&\forall \bfv \in \bfV_h
\\
\label{eq:jhgradcontrol}
\| \nabla \bfv  \|_{N_h(\Omega)} 
&\lesssim \| \bfv \|_{a_h},
\qquad &&\forall \bfv \in \bfV_h
\\
\label{eq:jhenergy-control}
2\mu \| \bfeps(\bfv)  \|^2_{N_h(\Omega)} 
+ \lambda \| \mathrm{tr} \epsilon(\bfv) \|^2_{N_h(\Omega)}
&\lesssim \tn \bfv \tn_h^2,
\qquad &&\forall \bfv \in \bfV_h
\end{alignat}
\end{lem}
\begin{proof} {\bf (\ref{eq:jhL2control}).} We first 
consider two neighboring elements
$K_1$ and $K_2$ sharing a face $F$. We shall prove that 
\begin{equation}\label{eq:jh-twoelements-l2}
\| \bfv \|^2_{K_1} 
\lesssim 
\| \bfv \|^2_{K_2}
+ 
\sum_{j=1}^p \|[D^j_n \bfv]\|^2_F 
h^{2j+1}
\end{equation}
Iterating (\ref{eq:jh-twoelements-l2}) a uniformly bounded number 
of times we reach an element in the interior of $\Omega$ and thus 
we can control the elements on the boundary in terms of elements 
in the interior and the stabilization term. 
To prove (\ref{eq:jh-twoelements-l2}) we consider $\bfv\in V_h$ and 
let $\bfv_i$ be the polynomial on $K_1 \cup K_2$ such that 
$\bfv_i = v|_{K_i}$, $i=1,2$. We then note that 
\begin{align}
\| \bfv_1 \|_{K_1} &\leq  \| \bfv_1 - \bfv_2 \|_{K_1} + \| \bfv_2 \|_{K_1} 
\lesssim 
 \| \bfv_1 - \bfv_2 \|_{K_1} + \| \bfv_2 \|_{K_2}
\end{align}
where we used the inverse inequality
\begin{equation}\label{eq:inverse-elements}
\| \bfw \|_{K_1} \lesssim  \| \bfw \|_{K_2} 
\end{equation}
which holds for all $\bfw \in \mcP_p (K_1 \cup K_2 )$ since $K_1$ and $K_2$ are 
shape regular. 

Now for each $\bfx \in K_1$ we have the identity
\begin{equation}\label{proog:jhcontrola}
\bfv_1(\bfx) - \bfv_2(\bfx) 
= \sum_{j=1}^p [D^j_n \bfv(\bfx_F)] x_n^j 
\end{equation}
where $\bfx = \bfx_F + x_n \bfn_F$, with $\bfn_F$ the unit normal to $F$ exterior to $K_2$, 
and $x_n \in \IR$ the distance from $\bfx$ to the hyperplane containing $F$. 
Let $P_F (K_1) = \{ \bfx_F = \bfx - x_n \bfn_F : \bfx\in K_1 \}$ be the projection of $K_1$ 
onto the hyperplane containing $F$ and $h_n\lesssim h$ be the 
maximal value of $x_n$ over $K_1$. Then we have the estimate
\begin{align}
\| \bfv_1 - \bfv_2 \|^2_{K_1} 
&\leq
\int_0^{h_n} \|[D^j_n \bfv]\|^2_{P_F (K_1)} x_n^{2j} \, d x_n 
\\
&\lesssim \sum_{j=1}^p \|[D^j_n \bfv]\|^2_{P_F (K_1)}
h_n^{2j+1}
\\
&\lesssim \sum_{j=1}^p \|[D^j_n \bfv]\|^2_F 
h_n^{2j+1}
\end{align}
where we used the inverse estimate 
\begin{equation}\label{eq:inverse-face}
\|[D^j_n \bfv]\|_{P_F (K_1)} \lesssim \|[D^j_n \bfv]\|_{F}
\end{equation} 
which holds with a uniform constant due to shaperegularity  
and the fact that $[D^j_n \bfv]$ is a polynomial on $P_F(K_1)$.
{\hfill \mbox{\fbox{} } \newline}
\\

\noindent{\bf{(\ref{eq:jhgradcontrol}).}} Here we have 
\begin{align}
\| \nabla \bfv_1 \|_{K_1} 
&\leq  
\|\nabla (\bfv_1 - \bfv_2) \|_{K_1} + \| \nabla \bfv_2 \|_{K_1} 
\lesssim 
 \|\nabla (\bfv_1 - \bfv_2) \|_{K_1} + \| \nabla \bfv_2 \|_{K_2}
\end{align}
where we used the inverse inequality
\begin{equation}\label{eq:inverse-elements-grad}
\| \nabla \bfw \|_{K_1} \lesssim  \| \nabla \bfw \|_{K_2} 
\end{equation}
for all $w \in \mcP_p(K_1 \cup K_2 )$, which holds since 
$\nabla w \in \mcP_{p-1}(K_1 \cup K_2 )$ is also a polynomial. 
Computing the gradient of (\ref{proog:jhcontrola}) we have the 
identity  
\begin{align}
\nabla (\bfv_1(x)- \bfv_2(x)) 
%&= \sum_{j=1}^p (\nabla_F g_j(x_F)) x_n^j + j g_j(x_F)) x_n^{j-1}
%\\
&= \sum_{j=1}^p (\nabla_F [D^j_n \bfv(\bfx_F)]) x_n^j  
+ j[D^j_n \bfv(\bfx_F)] x_n^{j-1}
\end{align}
where $\nabla_F = (1 - \bfn_F \otimes \bfn_F) \nabla$ is the gradient 
tangent 
to $F$. We then have the estimate
\begin{align}
\nonumber
&\|\nabla (\bfv_1(x)- \bfv_2(x))\|_{K_1}^2 
\\
&\qquad \lesssim \sum_{j=1}^p \int_0^{h_n} \Big( 
\|\nabla_F [D^j_n \bfv(\bfx_F)]\|^2_{P_F(K_1)} x_n^{2j}  
+ \|[D^j_n \bfv(\bfx_F)]\|^2_{P_F(K_1)} x_n^{2(j-1)} \Big) \, d x_n
\\
&\qquad \lesssim \sum_{j=1}^p \int_0^{h_n} \Big( 
h^{-2}\|[D^j_n \bfv(\bfx_F)]\|^2_{P_F(K_1)} x_n^{2j}  
+ \|[D^j_n \bfv(\bfx_F)]\|^2_{P_F(K_1)} x_n^{2(j-1)} \Big) \, d x_n
\\
&\qquad \lesssim \sum_{j=1}^p \|[D^j_n \bfv(\bfx_F)]\|^2_{P_F(K_1)} h_n^{2j-1} 
\\
&\qquad \lesssim \sum_{j=1}^p \|[D^j_n \bfv(\bfx_F)]\|^2_{F} h_n^{2j-1} 
\end{align}
where we used an inverse estimate to remove the tangent gradient 
in the first term on the right hand side and finally we used 
the inverse estimate (\ref{eq:inverse-face}).
{\hfill \mbox{\fbox{} } \newline}
\\

\noindent{\bf{(\ref{eq:jhenergy-control}).}} Proceeding as above we 
have the estimate
\begin{align}
2 \mu \| \bfeps (\bfv_1) \|^2_{K_1} 
+ \lambda \| \text{tr} \bfeps (\bfv_1) \|^2_{K_1}
%&\lesssim   
%2 \mu \| \bfeps (v_1 - v_2) \|^2_{K_1} 
%+ \lambda \| \text{tr} \bfeps (v_1 - v_2) \|^2_{K_1}
%\\ \nonumber
%&\qquad 
%+ 2 \mu \| \bfeps (v_2) \|^2_{K_1} 
%+ \lambda \| \text{tr} \bfeps (v_2) \|^2_{K_1}
%\\
&\lesssim 
2 \mu \| \bfeps (\bfv_1 - \bfv_2) \|^2_{K_1} + \lambda \| \text{tr} \bfeps (\bfv_1 - \bfv_2) \|^2_{K_1}
\\ \nonumber
&\qquad 
+ 2 \mu \| \bfeps (\bfv_2) \|^2_{K_2} 
+ \lambda \| \text{tr} \bfeps (\bfv_2) \|^2_{K_2}
\end{align}
where we used the inverse inequalities
\begin{equation}\label{eq:inverse-elements-eps}
\| \bfeps(\bfw) \|_{K_1} \lesssim  \| \bfeps(\bfw) \|_{K_2}, 
\qquad 
\| \text{tr} \bfeps(\bfw) \|_{K_1} \lesssim  \| \text{tr} \bfeps(\bfw) \|_{K_2}
\end{equation}
for all $w \in \mcP_k(K_1 \cup K_2 )$, which hold 
since $\bfeps(w) \in \mcP_{p-1}(K_1 \cup K_2 )$ and 
$\text{tr} \bfeps(w) \in \mcP_{p-1}(K_1 \cup K_2 )$ are both polynomials. 
Next, in view of the identity
\begin{align}
2\bfeps(\bfv) &= \bfv \otimes \nabla + \nabla \otimes \bfv
\\
&=  \bfv \otimes \nabla_F + \nabla_F \otimes \bfv + (D_n \bfv) \otimes \bfn_F 
+ \bfn_F \otimes (D_n \bfv)
\end{align}
we may prove (\ref{eq:jhenergy-control}) using the same approach as for 
the gradient estimate (\ref{eq:jhgradcontrol}).
\end{proof}

%\paragraph{The Form $\boldsymbol{m_h}$.}
%Introducing the norms
%\begin{align}
%\| v \|_{h,m}^2 
%&= \sum_{i=1}^n \|v\|^2_{H^m(\Omega_i)}
%+ h^{-2m} \tn v \tn^2_{j_{h,i}},
%\qquad m=-1,0,1
%\end{align}
%We have the obvious estimates
%\begin{equation}\label{eq:mhcoercont}
%\| v \|_{h,0}^2 \lesssim m_h(v,v), 
%\qquad
%m_h(v,w) \lesssim \| v \|_{h,0} \| w \|_{h,0},  
%\qquad
%\forall v,w \in V + V_h
%\end{equation}
%and we also have 
%\begin{equation}\label{eq:mhdual}
%m_h(v,w) \lesssim \| v \|_{h,1} \| w \|_{h,-1}, 
%\qquad \forall v,w \in V + V_h
%\end{equation}
%\paragraph{Verification of (\ref{eq:mhdual}).} We have
%\begin{align}
%m_h(v,w) &= \sum_{i=1}^n (v,w)_{\Omega_i} + j_{h,i}(v,w)
%\\
%&\leq \sum_{i=1}^n \|v\|_{H^1(\Omega_i)} \|w\|_{H^{-1}(\Omega_i)} 
%+ h^{-1}\tn v \tn_{j_{h,i}}  h \tn w \tn_{j_{h,i}}
%\\
%&\leq \left( \sum_{i=1}^n \|v\|^2_{H^1(\Omega_i)} 
%+ h^{-2}\tn v \tn_{j_{h,i}}^2 \right)^{1/2}
%\left( \sum_{i=1}^n \|w\|^2_{H^{-1}(\Omega_i)} 
%+ h^{2}\tn w \tn_{j_{h,i}}^2 \right)^{1/2}
%\\
%&= \| v \|_{h,1} \|w \|_{h,-1}
%\end{align}
%{\hfill \mbox{\fbox{} } \newline}

\subsection{Properties of the Forms}

\paragraph{Properties of $m_h$} 

\begin{lem}\label{lem:mh} The bilinear form $m_h$, defined by (\ref{def:mh}), is continuous and coercive 
\begin{equation}
m_h(\bfv,\bfw) \lesssim \| \bfv \|_{N_h(\Omega)} 
\| \bfw \|_{N_h(\Omega)},
\qquad 
\| \bfv \|^2_{N_h(\Omega)}
\lesssim m_h(\bfv,\bfv),\qquad \bfv,\bfw\in \bfV_h 
\end{equation}
\end{lem} 
\begin{proof} Coercivity follows 
directly from the inverse estimate (\ref{eq:jhL2control}) 
and to show continuity we use an inverse estimate to conclude 
that 
\begin{equation}
\|\bfv\|_{j_h} \lesssim \|\bfv\|_{N_h(\Omega)}, \qquad \bfv \in \bfV_h
\end{equation}

\end{proof}

\paragraph{Properties of $A_h$}
First we note that we have the following Poincar\'e inequality 
which shows that $\tn \cdot \tn_h$ is indeed a norm on $\bfV_h$ 
and plays an important role in the estimate of the condition number 
of the stiffness matrix associated with $A_h$.

\begin{lem}\label{lem:Poincare} The following Poincar\'e estimate holds
\begin{equation}\label{eq:Poincare}
\| \bfv \|_{N_h(\Omega)} 
\lesssim \tn \bfv \tn_h,
\qquad
\bfv \in \bfV_h
\end{equation}
\end{lem}

Next we turn to continuity and coercivity of $A_h$.

\begin{lem}\label{lem:Ah} The bilinear form $A_h$, defined by 
(\ref{def:ah}), is continuous
\begin{equation}\label{eq:ahcont}
A_h(\bfv,\bfw) \lesssim \tn \bfv \tn_{h} \tn \bfw \tn_{h},
\qquad 
\forall \bfv,\bfw \in \bfV + \bfV_h  
\end{equation}
and, if the stabilization parameter $\beta$ is large 
enough, coercive 
\begin{equation}\label{eq:ahcoer}
 \tn \bfv \tn_h^2 \lesssim a_h(\bfv,\bfv), 
\qquad
\forall
\bfv \in \bfV_h  
\end{equation}
\end{lem}
\begin{proof}
We show continuity of $A_h$ on $\bfV + \bfV_h$, directly 
using the Cauchy-Schwarz inequality. To show coercivity we first 
note that using the Cauchy-Schwarz inequality 
\begin{align}
\Big|(\bfsig(\bfv)\cdot n,\bfv)_{\Gamma_D}\Big| 
&\leq
2\mu \|\bfeps(\bfv)\|_{\partial \Omega_D} \|\bfv\|_{\partial \Omega_D}
+ 
\lambda \|\text{tr}\bfeps(\bfv)\|_{\partial \Omega_D} \|\bfv\cdot \bfn\|_{\partial \Omega_D}
\\
&\lesssim
\Big( 2\mu \|\bfeps(\bfv)\|_{\partial \Omega_D}^2
+ 
\lambda \|\text{tr}\bfeps(\bfv)\|_{\partial \Omega_D}^2\Big)^{1/2} 
\tn \bfv \tn_{b_h}
\\
&\lesssim
\delta h \Big( 2\mu \|\bfeps(\bfv)\|_{\partial \Omega_D}^2
+ 
\lambda \|\text{tr}\bfeps(\bfv)\|_{\partial \Omega_D}^2\Big)
+  
\delta^{-1}h^{-1} \tn \bfv \tn^2_{b_h}
\end{align}
for each $\delta>0$. Using the elementwise inverse 
inequality
\begin{equation}
h \Big( 2\mu \|\bfeps(\bfv)\|_{\partial \Omega_D\cap K}^2
+ 
\lambda \|\text{tr}\bfeps(\bfv)\|_{\partial \Omega_D\cap K}^2\Big)
\lesssim 
2\mu \|\bfeps(\bfv)\|_{K}^2
+ 
\lambda \|\text{tr}\bfeps(\bfv)\|_{K}^2
\end{equation}
where the hidden constant is independent of the position 
of the boundary $\partial \Omega_D$, see \cite{HaHaLa03}, followed 
by the inverse estimate (\ref{eq:jhgradcontrol}) we conclude 
that 
\begin{equation}
h \Big( 2\mu \|\bfeps(\bfv)\|_{\partial \Omega_D\cap K}^2
+ 
\lambda \|\text{tr}\bfeps(\bfv)\|_{\partial \Omega_D\cap K}^2\Big)
\lesssim \tn \bfv \tn^2_{a_h}
\end{equation}
This is the crucial inverse inequality required in the 
standard proof, see \cite{LarBen13} Section 14.2, of 
coercivity of $A_h$ on $\bfV_h$, which holds if $\beta$ 
is large enough. Note that the size of the penalty parameter 
is completely independent of the actual position of the 
domain $\Omega$ in the background mesh.
\end{proof}

\paragraph{Properties of  $L_h$} The form $L_h$ is continuous 
on $\bfV_h$, 
\begin{equation}\label{eq:Lh}
L_h(\bfv) \lesssim h^{-1/2} \tn \bfv \tn_h, \qquad \forall \bfv \in \bfV_h
\end{equation}

\begin{rem}
For fixed $h \in (0,h_0]$ we may then use the coercivity and continuity of $A_h$, 
see Lemma \ref{lem:Ah}, the continuity (\ref{eq:Lh}), and apply the Lax-Milgram 
lemma to conclude that there exists a unique solution to the finite element 
problem (\ref{eq:standard-fem}).
\end{rem}

\subsection{Conditioning of the Mass and Stiffness Matrices}

Let $\{\bfvarphi_{j}\}_{j=1}^{N}$ be the standard Lagrange basis 
in $\bfV_{h}$ and $N$ the dimension of 
$\bfV_{h}$. Let  $\bfv = \sum_{j=1}^{N} \widehat{v}_j \bfvarphi_{j}$ 
be the expansion of $\bfv$ in the Lagrange basis. We 
define the mass and stiffness matrices, ${\bf M}$ and 
${\bf A}$ with elements
\begin{equation}\label{def:mass-stiffness-matrix}
m_{ij} = m_h(\bfvarphi_i,\bfvarphi_j),\qquad 
a_{ij} = A_h(\bfvarphi_i,\bfvarphi_j), 
\end{equation}
which is equivalent to the identities 
\begin{equation}
m_h(\bfv,\bfw) = (\bbM \hatv,\hatw)_{\IR^N},\qquad
A_h(\bfv,\bfw) = (\bbA \hatv,\hatw)_{\IR^N}, \qquad \forall \bfv,\bfw \in \bfV_h
\end{equation}
where $(\cdot,\cdot)_{\IR^N}$ is the inner product in $\IR^N$. 

Recall that the condition number of an $N\times N$ matrix 
$\bf B$ is defined by $\kappa({\bf B}) 
= \| {\bf B} \|_{\IR^N} \|{\bf B}^{-1}\|_{\IR^N}$, 
where $\|\cdot\|_{\IR^N}$ is the Euclidian norm in $\IR^N$.
We then have the following result.

\begin{thm} The condition numbers of mass and stiffness 
matrices, ${\bf M}$ and ${\bf A}$, defined by 
(\ref{def:mass-stiffness-matrix}), satisfy the estimates
\begin{equation}\label{eq:condestimates}
\kappa({\bf M}) \sim 1, \qquad \kappa({\bf A}) \lesssim h^{-2}
\end{equation}
\end{thm}
The proof uses the approach in \cite{ErnGue06} and builds 
on the standard equivalence of norms
\begin{equation}\label{eq:RNL2equiv}
h^{d} \|\widehat{\bfv}\|^2_{\IR^N} 
\sim 
\| \bfv \|^2_{N_h(\Omega)} 
\end{equation}
where $\widehat{\bfv} \in \IR^N$ is the coefficient 
vector with elements $\widehat{v}_i$. 
%together with 
%the following estimates
%\begin{itemize}
%\item For ${\bf M}$ we use continuity and coercivity of $m_h$.
%\item For ${\bf A}$ we use continuity and coercivity of $a_h$, 
%and in order to pass between the norms $\|\cdot\|_{N_h(\Omega)}$ 
%and $\tn \cdot \tn_h$ we employ the Poincar\'e inequality (\ref{eq:Poincare}), and the inverse inequality 
%\begin{equation}
%\tn \bfv \tn_h \lesssim h^{-1} \| \bfv \|_{N_h(\Omega)}
%\end{equation}
%\end{itemize}
%We note that the coercivity results and the Poincar\'e estimate 
%crucially depend on stabilization.

\begin{proof} We present the proof for the bound on the condition 
number of the stiffness matrix. The bound for the condition number of the 
mass matrix follows in the same way but is slightly simpler since only the equivalence (\ref{eq:RNL2equiv}) and the coercivity and continuity properties 
in Lemma \ref{lem:mh} are used, while in the the case of the stiffness matrix 
we also employ the Poincar\'e inequality and the inverse bound
\begin{equation}\label{eq:cond-inverse}
\tn \bfv \tn_h \lesssim h^{-1} \| \bfv \|_{N_h(\Omega)}, \qquad \bfv \in \bfV_h
\end{equation}

We have  
\begin{equation}
\kappa(\bbA) = \| {\bbA}\|_{\IR^N}\| {\bbA}^{-1}\|_{\IR^N}
\end{equation}
Starting with the estimate of $\| {\bf A}\|_{\IR^N}$ we obtain, 
using the definition of the stiffness matrix and 
the inverse bound (\ref{eq:cond-inverse}), 
\begin{align}
({\bf A} \hatv, \hatw)_{\IR^N} &= A_h(\bfv,\bfw) 
\lesssim \tn \bfv \tn_h \tn \bfw \tn_h 
\\ 
&\qquad 
\lesssim h^{-2} \| \bfv \|_{N_h(\Omega)}  \| \bfw \|_{N_h(\Omega)}
\lesssim h^{d-2} \| \hatv \|_{\IR^N} \| \hatw \|_{\IR^N}
\end{align}
We conclude that 
$\| \bbA \hatv \|_{\IR^N} \lesssim h^{d-2} \| \hatv \|_{\IR^N}$, 
and thus
\begin{equation}\label{eq:cond-A}
\| \bbA \|_{\IR^N} \lesssim h^{d-2}
\end{equation}

Next to estimate $\| {\bf A}^{-1}\|_{\IR^N}$ we note that 
\begin{align}
h^{d/2}\| \hatv \|_{\IR^N} 
&\lesssim \| \bfv \|_{N_h(\Omega)} 
\lesssim \tn \bfv \tn_h 
\lesssim \sup_{\bfw \in \bfV_h \setminus \{0\}}  
\frac{A_h(\bfv,\bfw)}{\tn \bfw \tn_h}
\\ 
&
= \sup_{\bfw \in \bfV_h \setminus \{0\}}
\frac{(\bbA \hatv, \hatw)_{\IR^N}}{\|\hatw\|_{\IR^N}} 
\frac{\|\hatw\|_{\IR^N}}{\tn \bfw \tn_h}
\lesssim 
\sup_{w \in V_h \setminus \{0\}}
h^{-d/2} \frac{(\bbA \hatv, \hatw)_{\IR^N}}{\|\hatw\|_{\IR^N}} 
\end{align} 
where in the last step we used th eequivalence (\ref{eq:RNL2equiv}) and 
the Poincar\'e estimate (\ref{eq:Poincare}) to conclude that
\begin{equation}
\|\hatw\|_{\IR^N} 
\lesssim h^{-d/2} \| \bfw \|_{N_h(\Omega)} 
\lesssim h^{-d/2} \tn \bfw \tn_h
\end{equation}
Thus we arrive at 
\begin{equation}
h^d \| \hatv \|_{\IR^N} \lesssim \| \bbA \hatv \|_{\IR^N}
\end{equation}
and setting $\hatv = \bbA^{-1} \hatw$ we find that 
\begin{equation}
\| \bbA^{-1} \hatw \|_{\IR^N} \lesssim h^{-d} \| \hatw \|_{\IR^N} 
\end{equation}
and thus 
\begin{equation}\label{eq:cond-Ainv}
\| \bbA^{-1} \|_{\IR^N} \lesssim h^{-d} 
\end{equation}

Combining (\ref{eq:cond-A}) and (\ref{eq:cond-Ainv}) the 
estimate of the condition number for the stiffness matrix follows.
\end{proof}

\subsection{Interpolation}
To define the interpolation operator $\bfpi_h$ we use a bounded 
extension operator $E:H^s(\Omega) \rightarrow H^s(U(\Omega))$, 
where $U(\Omega)$ is a neighborhood of $\Omega$ such that 
$N_h(\Omega)\subset U(\Omega)$, $h\in (0,h_0]$, and then using 
the Scott-Zhang interpolation operator 
$\bfpi_{h,SZ}:\bfH^1(N_h(\Omega))\rightarrow \bfV_h$, we may 
define 
\begin{equation}
\bfpi_h : H^1(\Omega) \ni \bfu \mapsto \bfpi_{h,SZ} E \bfu
\in \bfV_h 
\end{equation}
Using the simplified notation $\bfu = E \bfu$ on $U(\Omega)$ we have
have the interpolation error estimate 
\begin{equation}\label{eq:interpol-K}
\| \bfu - \bfpi_h \bfu \|_{H^m(K)} \lesssim h^{s - m} \| \bfu \|_{H^s(N(K))}, \qquad
0\leq m \leq s \leq p+1  
\end{equation}
Summing over the elements $K \in \mcK_h$ and using the stability 
of the extension operator we obtain
\begin{equation}\label{eq:interpol-Nh}
\| \bfu - \bfpi_h \bfu \|_{H^m(N_h(K))} 
\lesssim h^{s - m} \| \bfu \|_{H^s(\Omega)}, \qquad
0\leq m \leq s \leq p+1  
\end{equation}
Furthermore, we have the following 
estimate for the interpolation error in the energy norm
\begin{equation}\label{eq:interpol-energy}
\tn \bfu - \bfpi_h \bfu \tn_h 
\lesssim h^k \| \bfu \|_{H^{k+1}(\Omega)} 
\end{equation}
To verify (\ref{eq:interpol-Nh}) we note that using the trace 
inequality 
\begin{equation}
\| \bfv \|^2_{\partial \Omega_D \cap K } \lesssim h^{-1} \| \bfv \|^2_{K } 
+ h \| \nabla \bfv \|^2_{K } 
\end{equation}
where the hidden constant is independent of the position of the 
boundary $\partial \Omega_D$ in element $K$ see \cite{HaHaLa03} for further details, 
we may estimate the boundary terms as follows
\begin{align}
h^{-1} \| \bfv \|^2_{b_h}
&\lesssim 
h^{-2} \| \bfv \|^2_{\mcK_h(\partial \Omega_D)} + \| \nabla \bfv \|^2_{\mcK_h(\partial \Omega_D)}
\end{align}
\begin{align}
h \| \bfsig(\bfv) \cdot \bfn \|^2_{\partial \Omega_D} 
\lesssim 
\| \nabla \bfv \|^2_{\mcK_h(\partial \Omega_D)} + h^2 \| \nabla^2 \bfv  \|^2_{\mcK_h(\partial \Omega_D)}
\end{align}
where $\mcK_h(\partial \Omega_D) = \{ K \in \mcK_h : K \cap \partial \Omega_D \neq \emptyset\}$.
Using the standard trace inequality 
\begin{equation}
\| \bfv \|^2_{F } \lesssim h^{-1} \| \bfv \|^2_{K } 
+ h \| \nabla \bfv \|^2_{K } 
\end{equation}
where $F$ is a face associated with element $K$ we obtain
\begin{equation}
h^{-2} \| \bfv \|^2_{j_h} 
\lesssim 
h^{-2} \| \bfv \|^2_{\mcK_h(\mcF_h(\partial \Omega))}
\end{equation}
where $\mcK_h(\mcF_h(\partial \Omega))$ is the set of elements that have a face in 
$\mcF_h(\partial \Omega)$, i.e. an element with a stabilized face. Thus we conclude that 
we have the estimate 
\begin{equation}
\tn \bfv \tn^2_h \lesssim h^{-2} \| \bfv \|_{N_h(\Omega)} + \|\nabla \bfv \|^2_{N_h(\Omega)} 
+ h^2 \| \nabla^2 \bfv \|_{N_h(\Omega)}  
\end{equation}
Finally, setting 
$\bfv = \bfu - \bfpi_h \bfu$ and using (\ref{eq:interpol-K}) and the stability of the extension 
operator (\ref{eq:interpol-energy}) follows.

\subsection{A Priori Error Estimates}

\begin{thm} Let $\bfu$ be the solution to problem 
(\ref{eq:standard-a}-\ref{eq:standard-c}) and $\bfu_h$ 
the corresponding finite element approximation defined by 
(\ref{eq:standard-fem}), then the following a priori 
error estimates hold
\begin{equation}
\tn \bfu - \bfu_h\tn \lesssim h^k\| \bfu \|_{H^{k+1}(\Omega)}, 
\qquad 
\| \bfu - \bfu_h\|_{\Omega} \lesssim h^{k+1}\| \bfu \|_{H^{k+1}(\Omega)},
\end{equation}
\end{thm}
\begin{proof} Using the coercivity and continuity properties of $A_h$, 
Galerkin orthogonality, and the interpolation estimate (\ref{eq:interpol-energy}) 
we prove the error estimates using the standard approach, see \cite{LarBen13}.

Adding and subtracting an interpolant and using the triangle inequality we have 
\begin{align}
\tn \bfu - \bfu_h \tn_h 
&\lesssim 
\tn \bfu - \bfpi_h \bfu \tn_h + \tn \bfpi_h \bfu - \bfu_h \tn_h 
\\
&\lesssim 
h^k \| \bfu \|_{H^{k+1}(\Omega)} + \tn \bfpi_h \bfu - \bfu_h \tn_h
\end{align}
For the second term we use coercivity of $A_h$, 
\begin{align}
\tn \bfpi_h \bfu - \bfu_h \tn_h^2 &\lesssim A(\bfpi_h \bfu - \bfu_h, \bfpi_h \bfu - \bfu_h) 
\\
&\lesssim 
A(\bfpi_h \bfu - \bfu, \bfpi_h \bfu - \bfu_h)
\\
&\lesssim 
\tn \bfpi_h \bfu - \bfu\tn_h \tn \bfpi_h \bfu - \bfu_h \tn_h 
\end{align}
and thus 
\begin{equation}
\tn \bfpi_h \bfu - \bfu_h \tn_h 
\lesssim 
\tn \bfpi_h \bfu - \bfu\tn_h
\lesssim
h^k \| \bfu \|_{H^{k+1}(\Omega)}
\end{equation}

The $L^2$ estimate is proved using a standard duality argument, see \cite{LarBen13}. 
\end{proof}

\section{Examples}

The numerical results presented below include benchmarks of CutFEM for standard problems in linear elasticity, a study of the condition numbers of mass and stiffness matrices, and also various examples which demonstrate different modeling possibilities which are naturally combined with CutFEM.

Unless otherwise stated we in the numerical results below assume a linear elastic isotropic material with the material constants of steel, i.e., Young's modulus $E=200\cdot 10^9$ Pa, Poisson's ratio $\nu=0.3$, and density $\rho=7850$ kg/m$^3$. Default parameter values used in the CutFEM method are $\gamma_D = 1000\cdot p^2$ for the Nitche penalty parameter, $\gamma_{m} = \rho \cdot 10^{-4}$ for the mass matrix stabilization and $\gamma_a = (2\mu + \lambda)\cdot 10^{-4}$ for the stiffness matrix stabilization.
The finite elements used in are Lagrange elements with evenly distributed nodes; either full polynomials of order $p$ on triangles or tensor product polynomials of order $p$ on quadrilaterals.

\subsection{Benchmarks}
\label{sect:benchmarks}
We begin our numerical examples by studying the performance of CutFEM in the three standard problems of linear elasticity described in Section~\ref{sec:linear-elasticity}.

\begin{itemize}
\item
(\emph{The Displacement Problem})\,
In the stationary load problem we benchmark CutFEM using a manufactured problem on the unit square. The geometry and the solution is given by
\begin{subequations}
\label{eq:manufactured}
\begin{align}
&\Omega = [0,1]^2
 , \
\partial\Omega_D = \{x\in[0,1],y=0\} , \
\partial\Omega_N = \partial\Omega \backslash \partial\Omega_D
\\
&\bfu(x,y) = [ -\cos(\pi x)\sin(\pi y), \, \sin(\pi x/7)\sin(\pi y/3) ]/10
\end{align}
\end{subequations}
and from this we deduce expressions for the input data $f$, $g_N$ and $g_D$.
Visually inspecting numerical solutions using $p=\{1,2\}$ elements and different rotations of the background grid we in Figure~\ref{fig:load-sol} note that $p=1$ elements are sensitive to the grid rotation with regards to the quality of the stresses while $p=2$ elements give results which are visually invariant to the grid rotation.
In Figure~\ref{fig:load-conv} we present convergence results in $L^2(\Omega)$ norm and we see that the performance of CutFEM using $p=\{1,2,3,4\}$ elements in this problem is equivalent to that of conforming FEM using the same elements.

\begin{figure}
\centering
\begin{minipage}[b]{9.5cm}
\includegraphics[width=0.49\linewidth]{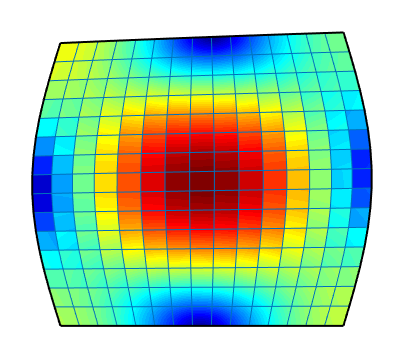}
\includegraphics[width=0.49\linewidth]{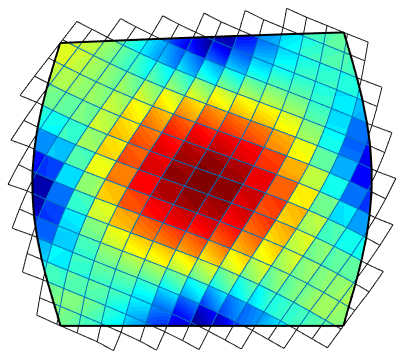}
\\
\includegraphics[width=0.49\linewidth]{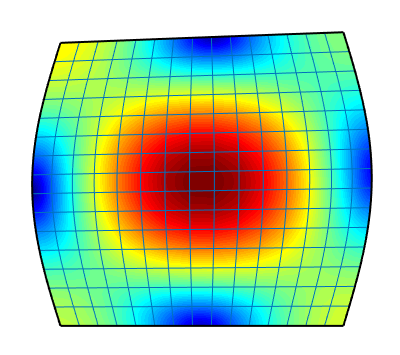}
\includegraphics[width=0.49\linewidth]{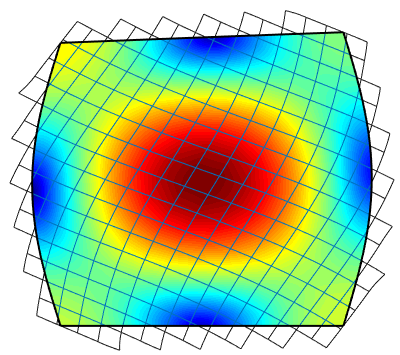}
\end{minipage}
\caption{CutFEM solutions to the static load model problem using a mesh size $h=0.1$ visualized with displacements and von-Mises stresses. In the two top subfigures $p=1$ elements are used while in the bottom two subfigures $p=2$ elements are used.}
\label{fig:load-sol}
\end{figure}

\begin{figure}
\centering
\includegraphics[width=0.49\linewidth]{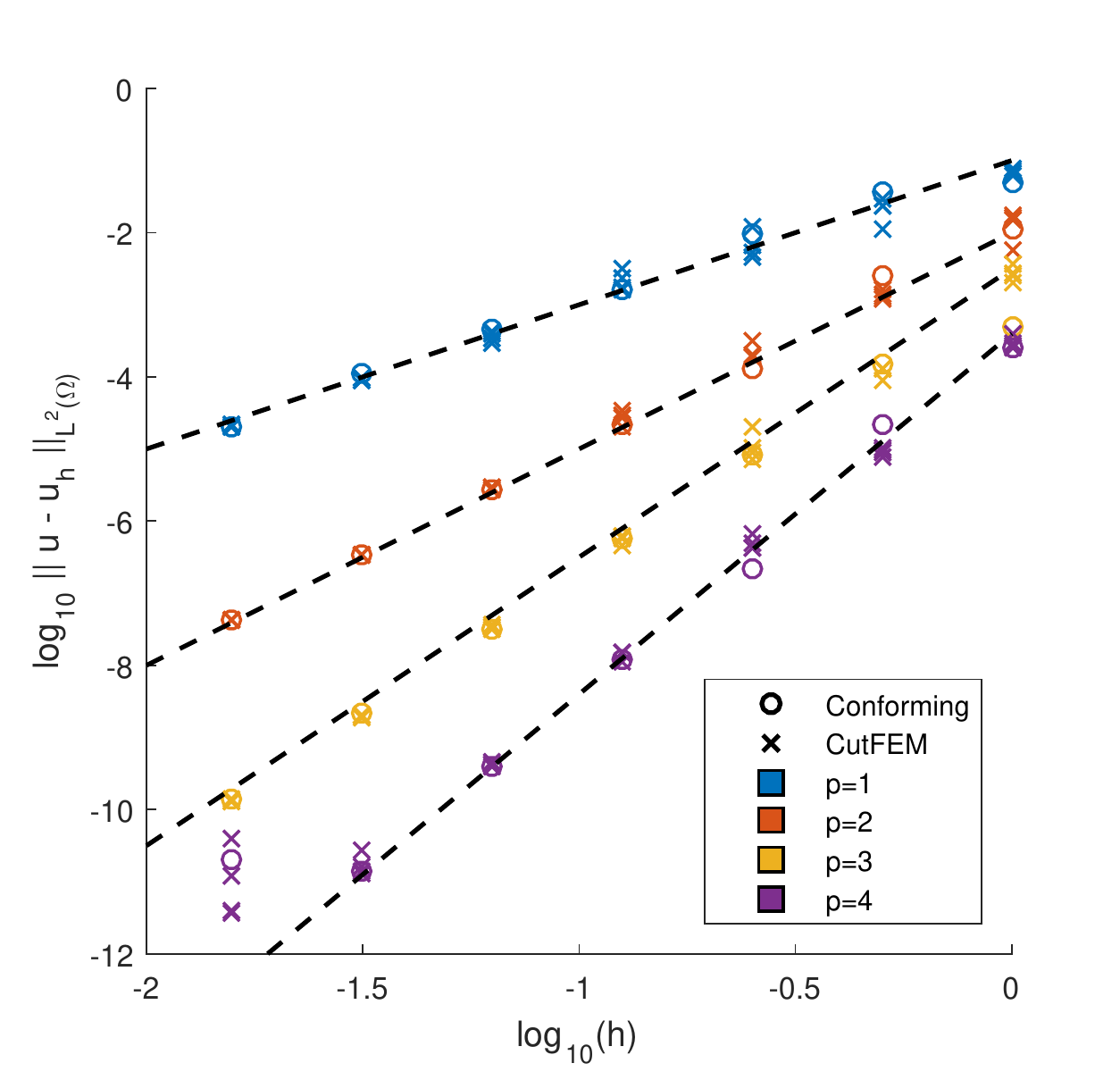}
\includegraphics[width=0.49\linewidth]{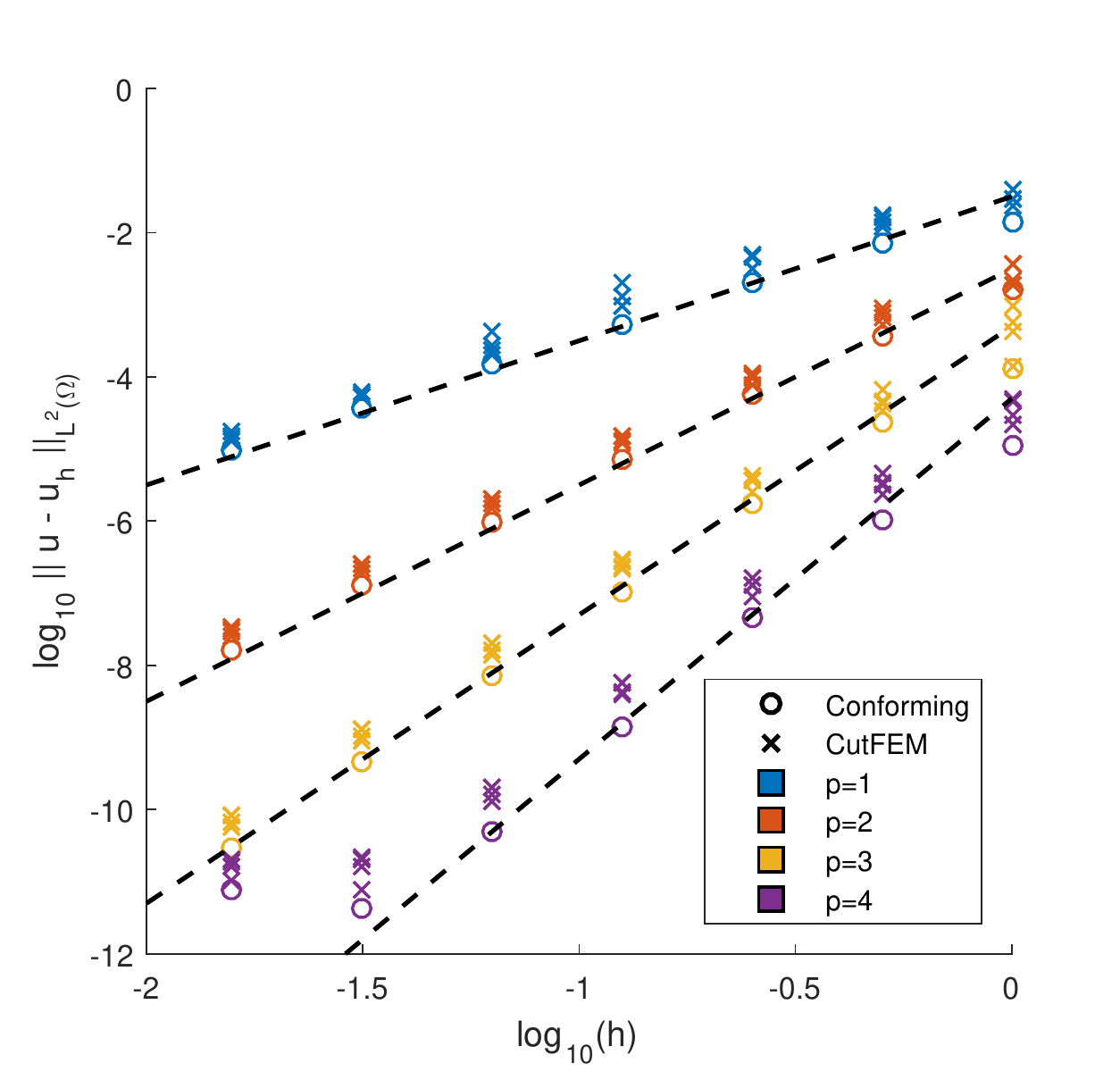}
\caption{Convergence in $L^2(\Omega)$ norm for the stationary load model problem using $p=1,2,3,4$ elements. In the left subfigure triangle elements are used and in the right subfigure quadrilateral elements are used. The dashed reference lines indicate theoretical convergence rates proportional to $h^{p+1}$.}
\label{fig:load-conv}
\end{figure}

\item
(\emph{The Frequency Response Problem})\,
To evaluate the performance of CutFEM in the frequency response problem we consider the cantilever beam with holes used to illustrate the cut meshes in Figures~\ref{fig:model-mesh}--\ref{fig:faces-stab}. This steel beam is fixated along left side and is under the influence of an oscillating gravitational load with an oscillatory frequency $\omega$. As reference we use $p=4$ elements on a conforming triangle mesh and we compare this to CutFEM calculations using $p=2$ elements on two structured meshes, see Figure~\ref{eq:freq-rep-meshes}. We are interested in accurately estimating the energy $E(\omega)$ defined in \eqref{eq:energyE} and in Figure~\ref{eq:freq-rep} we present the results for CutFEM compared to the higher order conforming method. It seems the coarse grid CutFEM estimation fails to capture some of the higher frequency details while the higher resolution CutFEM estimation succeeds. This is quite natural as the oscillatory modes corresponding to higher frequency loads most likely have more local details.

\begin{figure}
\centering
\includegraphics[width=0.9\linewidth]{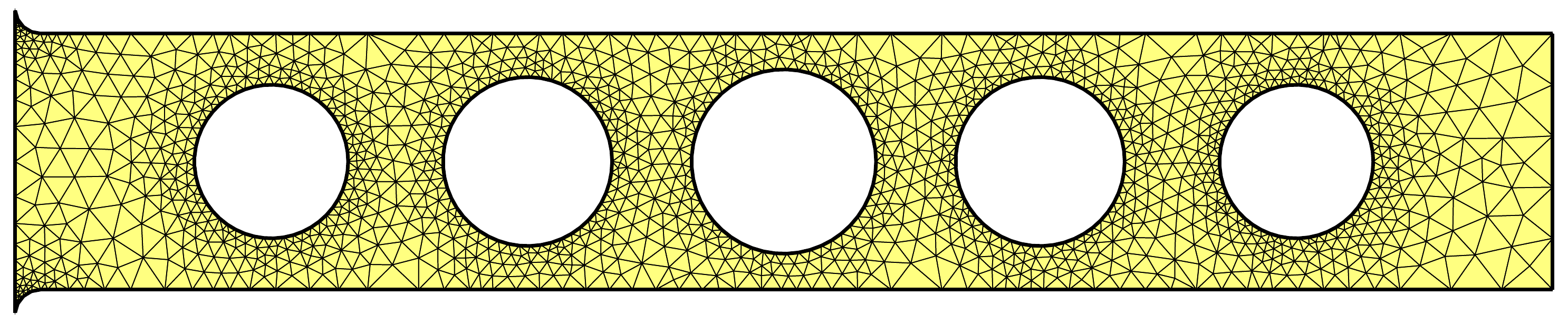}
\includegraphics[width=0.9\linewidth]{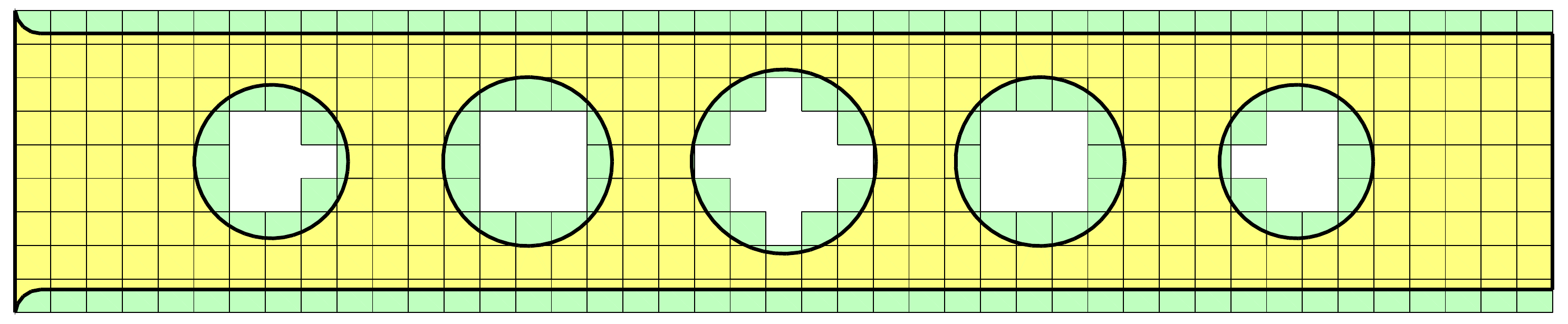}
\includegraphics[width=0.9\linewidth]{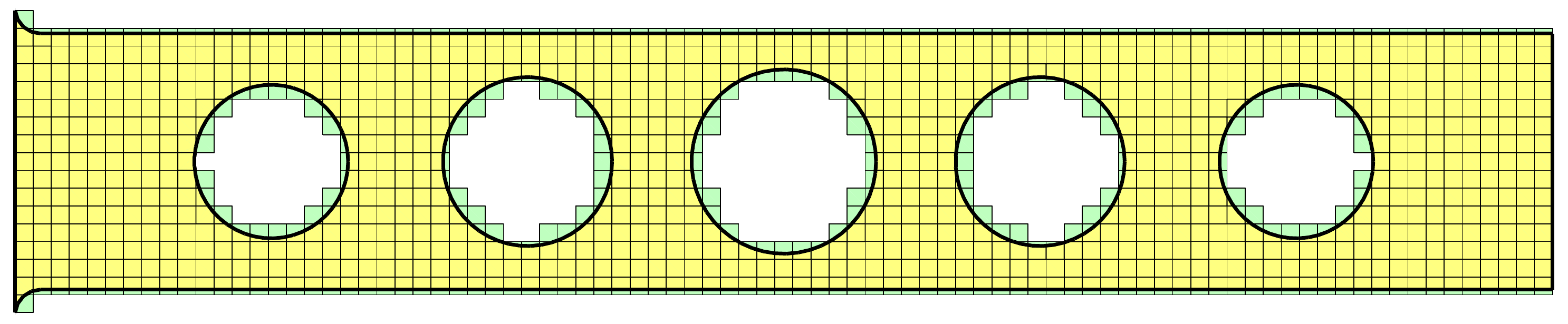}
\caption{Meshes used in the frequency response calculations.}
\label{eq:freq-rep-meshes}
\end{figure}

\begin{figure}
\centering
\includegraphics[width=0.49\linewidth]{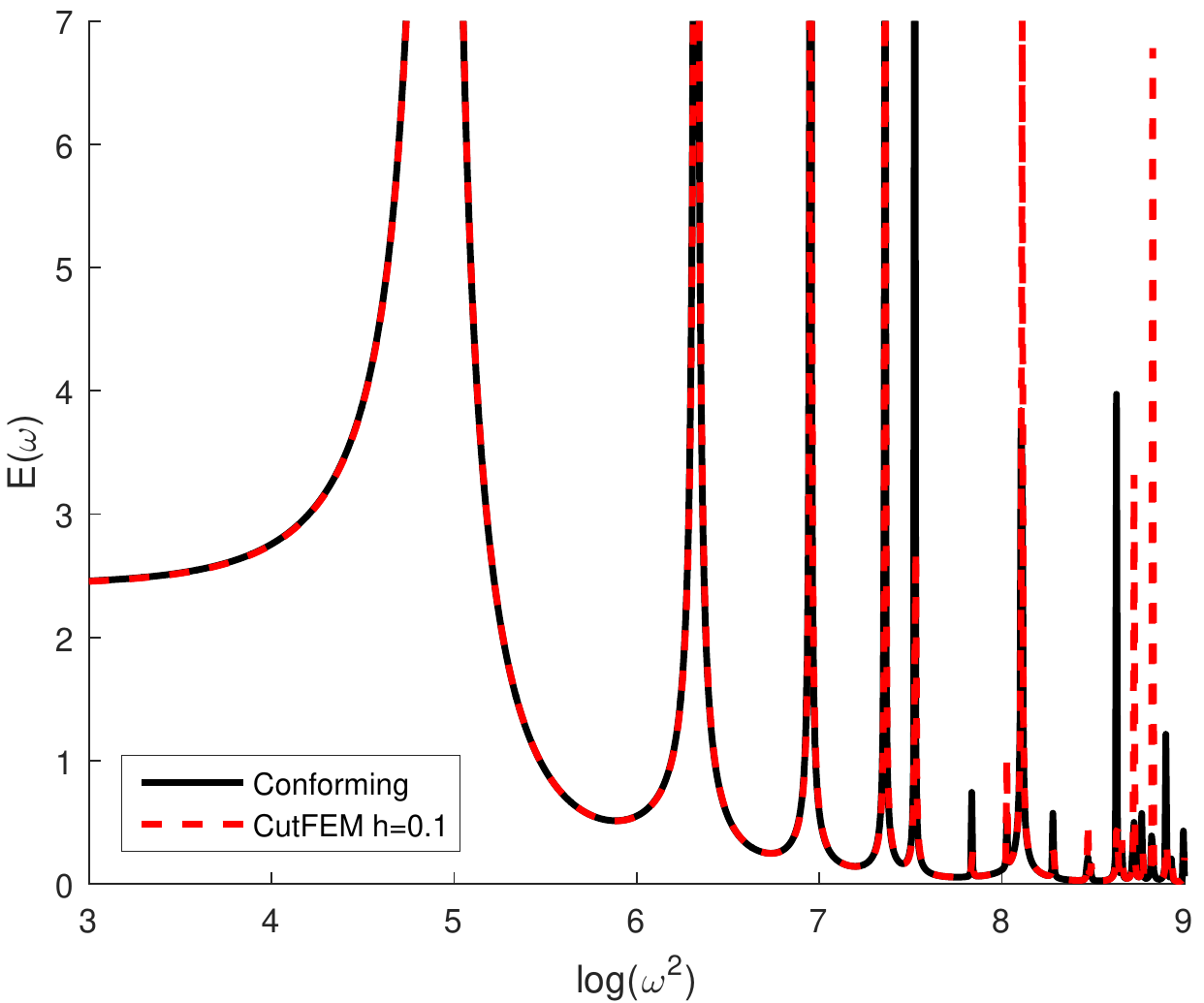}
\includegraphics[width=0.49\linewidth]{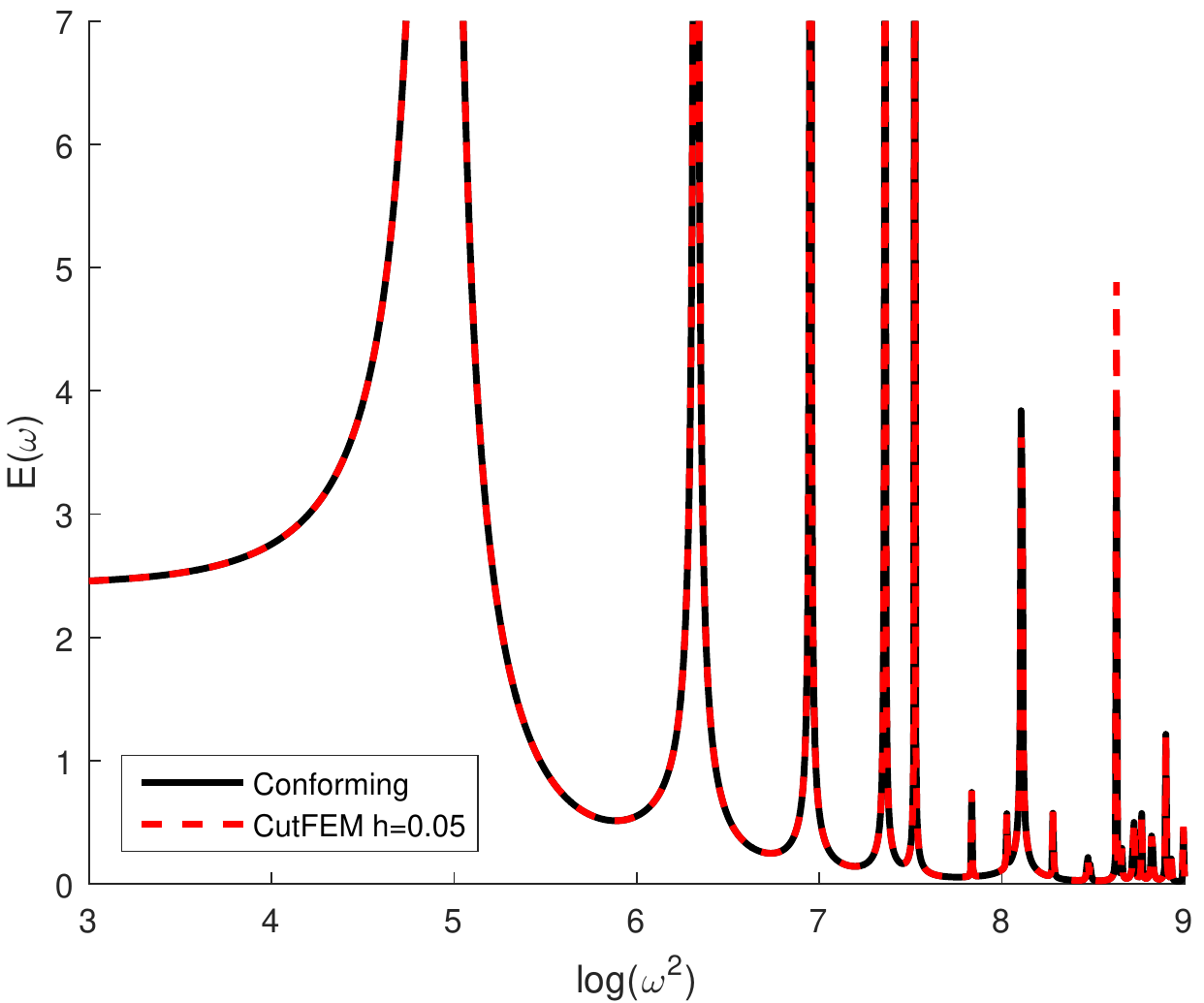}
\caption{Energy in frequency response calculations. $p=2$}
\label{eq:freq-rep}
\end{figure}

\item
(\emph{The Eigenvalue Problem})\,
We benchmark CutFEM for the eigenvalue problem on a free steel beam of length 3\,m and height 0.3\,m, i.e.
\begin{align}
\Omega=[0,3]\times[0,0.3]\, , \ \partial\Omega_N =\partial\Omega \, , \ \partial\Omega_D =\emptyset
\end{align}
In Figure~\ref{fig:evp-modes} we give a visual comparison of an eigenmode computed using $p=2$ elements with conforming FEM and CutFEM in two different cut situations and we note no noticeable differences.
To investigate the convergence in eigenvalues we as reference use
conforming FEM with $p=4$ elements on a fine grid.
The convergence results for an eigenvalue is presented in Figure~\ref{fig:evp-conv} and we see that CutFEM in this problem performs equivalently to conforming FEM.

\begin{figure}
\centering
\includegraphics[width=0.8\linewidth]{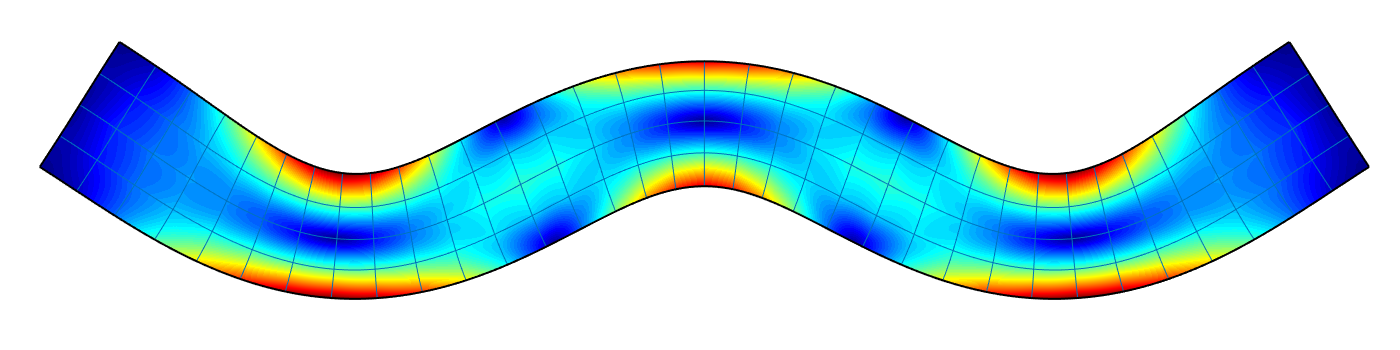}
\includegraphics[width=0.8\linewidth]{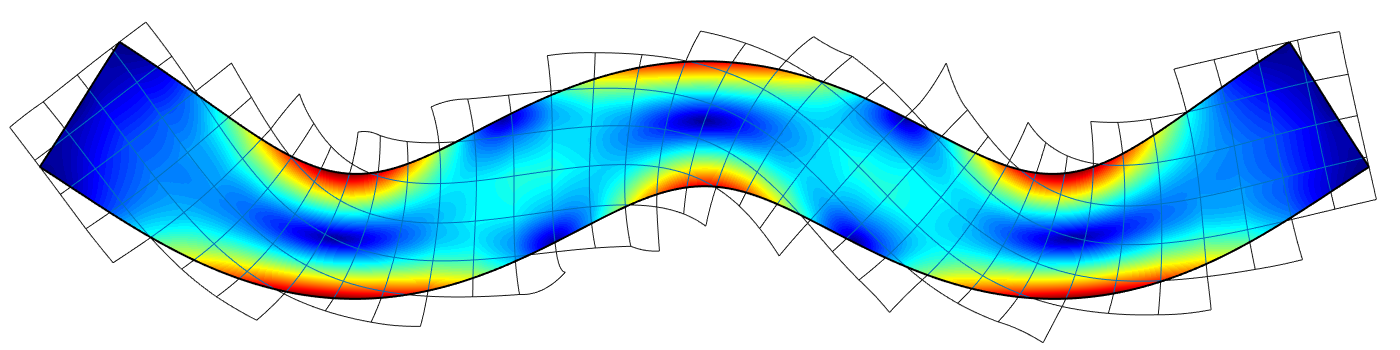}
\includegraphics[width=0.8\linewidth]{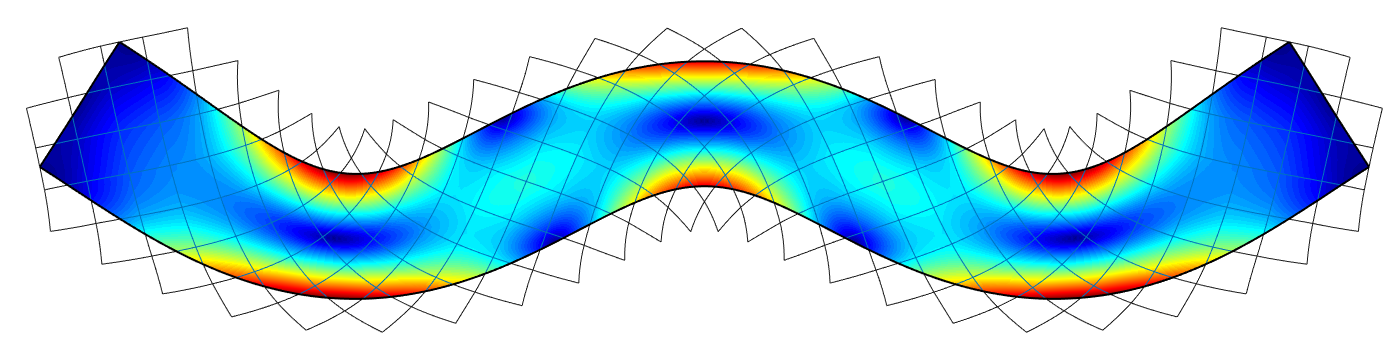}
\caption{Shape and von-Mises stress for 6:th eigenmode using $p=2$ elements. From top to bottom: Conforming FEM, CutFEM with the background mesh rotated $\theta=\pi/9$ respectively $\theta=\pi/4$. There are no visible differences in the shape and only minute differences in the stress.}
\label{fig:evp-modes}
\end{figure}

\begin{figure}
\centering
\includegraphics[width=0.49\linewidth]{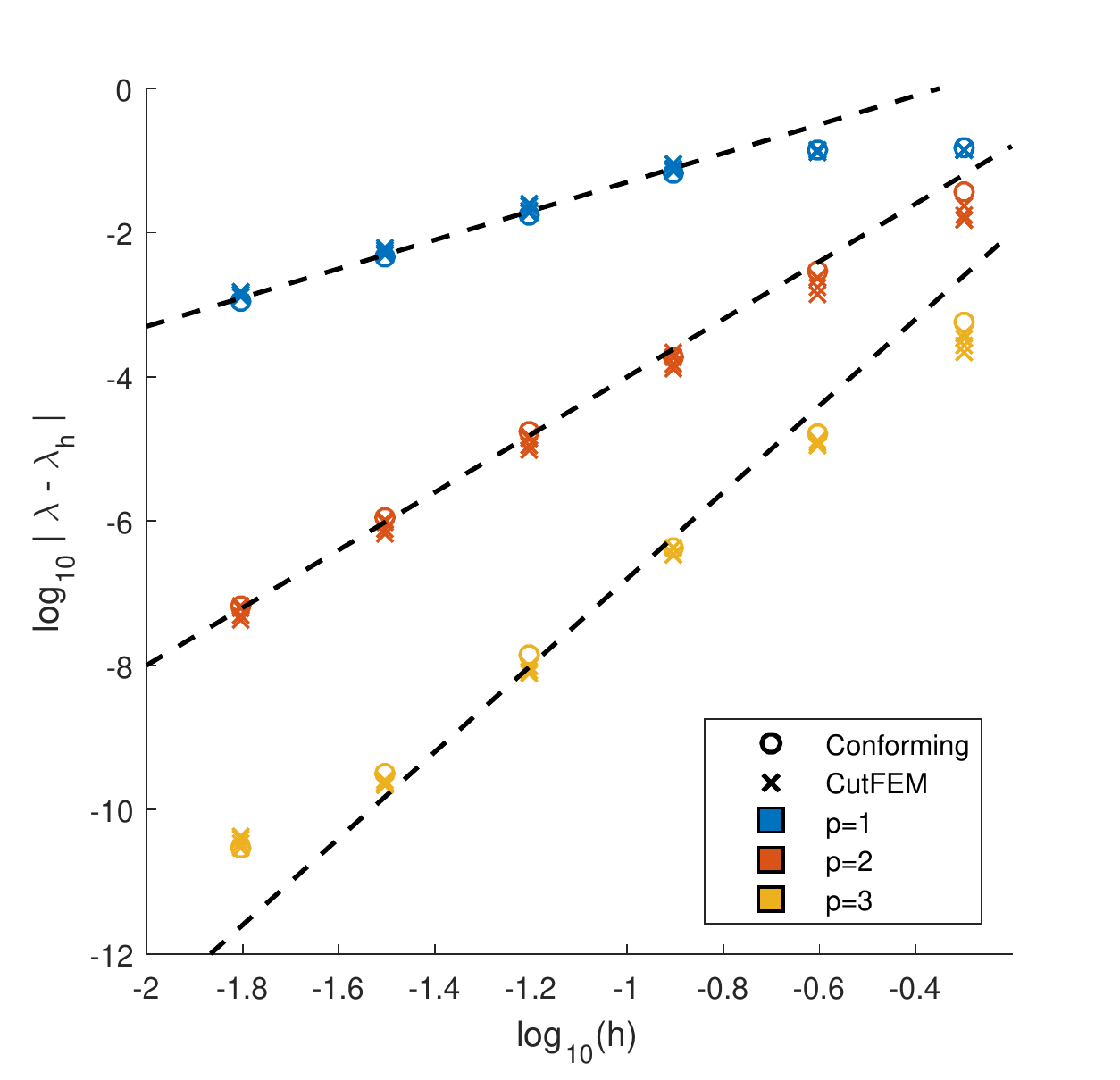}
\caption{Convergence of 6:th eigenvalue in the free eigenvalue problem. Computations are performed on quadrilaterals and a reference value $\lambda_{\text{ref}} = 2.7063377630 \cdot 10^{7}$ is computed using $p=4$ elements on the finest matching grid. Dashed reference lines are proportional to $h^{2p}$. }
\label{fig:evp-conv}
\end{figure}

\end{itemize}

\subsection{Condition Numbers}

As the domain is allowed to cut the elements in the background grid in an arbitrary fashion we may end up with cut situations in which the method becomes ill-conditioned unless properly stabilized or preconditioned.
While we focus our work on the stabilization side an alternative, or complement, to stabilization is preconditioning techniques which can yield good results, see for example \cite{PrVeZwBr17}.
In the numerical results below we include preconditioning in the form of simple diagonal scaling, demonstrated for the stationary load problem as follows.
The stationary load problem in matrix form reads
\begin{align}
\mathbf{A}\mathbf{U} = \mathbf{L}
\end{align}
where $\mathbf{A}$ is the stiffness matrix, $\mathbf{U}$ is the vector of unknowns, and $\mathbf{L}$ is the load vector. By the substitution $\mathbf{U} = \mathbf{D}\mathbf{Y}$ where $\mathbf{D}$ is the diagonal matrix given by $\mathbf{D}_{ii} = \mathbf{A}_{ii}^{-1/2}$ we can instead consider the equivalent matrix problem
\begin{align}
(\mathbf{D}^T\mathbf{A}\mathbf{D})\mathbf{Y} = \mathbf{D}^T\mathbf{L}
\end{align}
where the (proconditioned) matrix $\mathbf{D}^T\mathbf{A}\mathbf{D}$ have diagonal elements which are all one and in which the symmetry of $\mathbf{A}$ is preserved.

We investigate the effect of preconditioning and stabilization using the model static load problem from Section~\ref{sect:benchmarks}. First we construct two extreme cut situations as illustrated in Figure~\ref{fig:cond-meshes}; a perfectly fitted (conforming) mesh and a mesh where all boundary elements have a very small intersection with the domain.
Comparing the condition number results for a perfectly fitted mesh in Table~\ref{tab:cond-perfect-cut} to the `worst case' mesh results in Table~\ref{tab:cond-bad-cut} we note that a combination of stabilization and preconditioning by simple diagonal scaling gives the best results. However, we also see that even when applying both stabilization and preconditioning condition numbers when using $p\geq 3$ elements are orders of magnitude greater in the worst case scenario.
We see the same behavior when investigating how the condition numbers scale with the mesh size $h$. By rotating the background grid $\theta=\{0,\frac{\pi}{9},\frac{\pi}{7},\frac{\pi}{5}\}$ radians we construct both conforming meshes and various cut situations for which we compare condition numbers from preconditioning by diagonal scaling alone to condition numbers when we also add stabilization, i.e. CutFEM. In Figure~\ref{fig:cond-conv-p1p2} we for $p=\{1,2\}$ elements see good results, both regarding scaling and size of condition numbers. Note that for $p=1$ elements preconditioning by diagonal scaling alone is sufficient. For $p=\{3,4,5\}$ elements we in  Figure~\ref{fig:cond-conv-p3p4p5} see that the condition numbers scale in the correct manner but their size in cut situations are orders of magnitude higher compared to those of conforming FEM. Reviewing the analytical results above we attribute this effect to the constant in the inverse estimates \eqref{eq:inverse-elements}, \eqref{eq:inverse-elements-grad} and \eqref{eq:inverse-elements-eps} which can become quite large for higher order polynomials.

\begin{figure}
\centering
\includegraphics[width=0.35\linewidth]{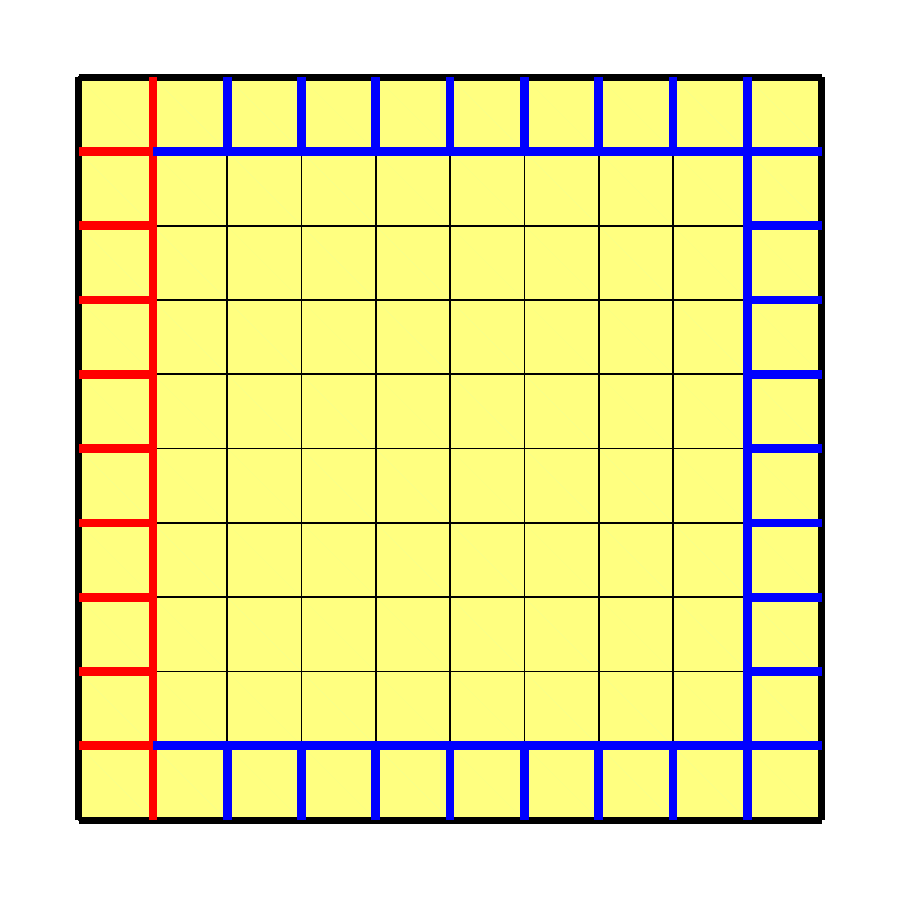}
\includegraphics[width=0.35\linewidth]{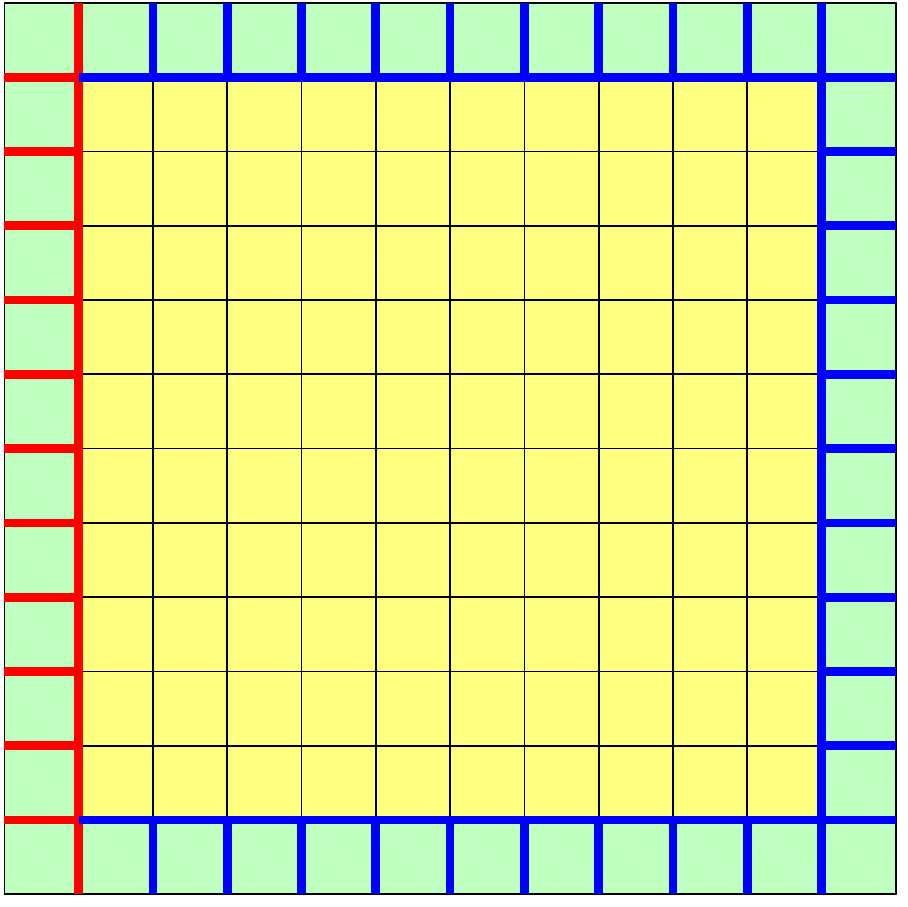}
\caption{Meshes used in estimation of condition numbers. Perfectly fitted mesh (left) and a `worst case' mesh (right) where only 1/1000 of the boundary elements are inside the domain.}
\label{fig:cond-meshes}
\end{figure}

\begin{table}
\caption{Numerical condition number estimates for stiffness matrix $\mathbf{A}$ and mass matrix $\mathbf{M}$ in the perfectly fitted situation in Figure~\ref{fig:cond-meshes}. The preconditioning here is a simple diagonal scaling of the matrix.}
\label{tab:cond-perfect-cut}	
\begin{tabular}{lllllllll}
\toprule
$\mathbf{A}\quad$ & $p\quad$ &  Plain & Precond. & \quad & $\mathbf{M}\quad$ &  $ p\quad$ &  Plain & Precond. \\
\midrule
& 1 & $5.2187\cdot 10^{5}\quad$  &  $3.4096\cdot 10^{3}\quad$ & &
& 1 & $2.2392\cdot 10^{1}\quad$  &  $1.1646\cdot 10^{1}$ \\
& 2 & $5.0805\cdot 10^{6}$  &  $2.7717\cdot 10^{4}$ & &
& 2 & $6.1285\cdot 10^{1}$  &  $1.3824\cdot 10^{1}$ \\
& 3 & $1.9148\cdot 10^{7}$  &  $1.1609\cdot 10^{5}$ & &
& 3 & $1.2169\cdot 10^{2}$  &  $2.0935\cdot 10^{1}$ \\
& 4 & $6.0930\cdot 10^{7}$  &  $4.6446\cdot 10^{5}$ & &
& 4 & $3.1309\cdot 10^{2}$  &  $3.4586\cdot 10^{1}$ \\
& 5 & $1.7828\cdot 10^{8}$  &  $1.7832\cdot 10^{6}$ & &
& 5 & $9.6345\cdot 10^{2}$  &  $6.2514\cdot 10^{1}$ \\
\bottomrule
\end{tabular}
\end{table}

\begin{table}
\caption{Numerical condition number estimates for the stiffness matrix $\mathbf{A}$ and the mass matrix $\mathbf{M}$ for the `worst case' cut situation in Figure~\ref{fig:cond-meshes} where only 1/1000 of boundary elements are inside the domain. The preconditioning here is a simple diagonal scaling of the matrix.}
\label{tab:cond-bad-cut}	
\begin{tabular}{llllll}
\toprule
$\mathbf{A}\quad$ & $p\quad$ &  Plain & Precond. & Stabilized & Stab. + Precond. \\
\midrule
& 1 &   $1.6675\cdot 10^{16}\quad$  &  $2.8270\cdot 10^{3}\quad$  &  $4.6725\cdot 10^{6}\quad$  &  $2.9021\cdot 10^{3}\quad$ \\
& 2 &   $6.0692\cdot 10^{30}$  &  $1.4264\cdot 10^{16}$  &  $1.0239\cdot 10^{9}$  &  $2.3952\cdot 10^{4}$ \\
& 3 &   $1.6325\cdot 10^{32}$  &  $4.0341\cdot 10^{19}$  &  $1.3580\cdot 10^{11}$  &  $4.7715\cdot 10^{5}$ \\
& 4 &   $3.0219\cdot 10^{32}$  &  $1.0862\cdot 10^{20}$  &  $2.2118\cdot 10^{13}$  &  $1.1224\cdot 10^{8}$ \\
& 5 &   $4.1334\cdot 10^{33}$  &  $2.0851\cdot 10^{21}$  &  $4.4914\cdot 10^{15}$  &  $3.5444\cdot 10^{10}$ \\
\midrule
$\mathbf{M}\quad$ & $p\quad$ &  Plain & Precond. & Stabilized & Stab. + Precond. \\
\midrule
& 1 & $2.6963\cdot 10^{19}$  &  $1.1848\cdot 10^{1}$  &  $3.4553\cdot 10^{3}$  &  $1.5677\cdot 10^{1}$ \\
& 2 & $3.8825\cdot 10^{34}$  &  $5.1576\cdot 10^{16}$  &  $1.8307\cdot 10^{5}$  &  $1.3623\cdot 10^{3}$ \\
& 3 & $1.2091\cdot 10^{35}$  &  $5.7858\cdot 10^{18}$  &  $1.0203\cdot 10^{7}$  &  $2.6414\cdot 10^{5}$ \\
& 4 & $3.6643\cdot 10^{35}$  &  $6.0559\cdot 10^{19}$  &  $1.4240\cdot 10^{9}$  &  $6.7336\cdot 10^{7}$ \\
& 5 & $4.7194\cdot 10^{35}$  &  $5.1345\cdot 10^{19}$  &  $5.3499\cdot 10^{11}$  &  $2.4344\cdot 10^{10}$ \\
\bottomrule
\end{tabular}
\end{table}

\begin{figure}
\centering
\includegraphics[width=0.49\linewidth]{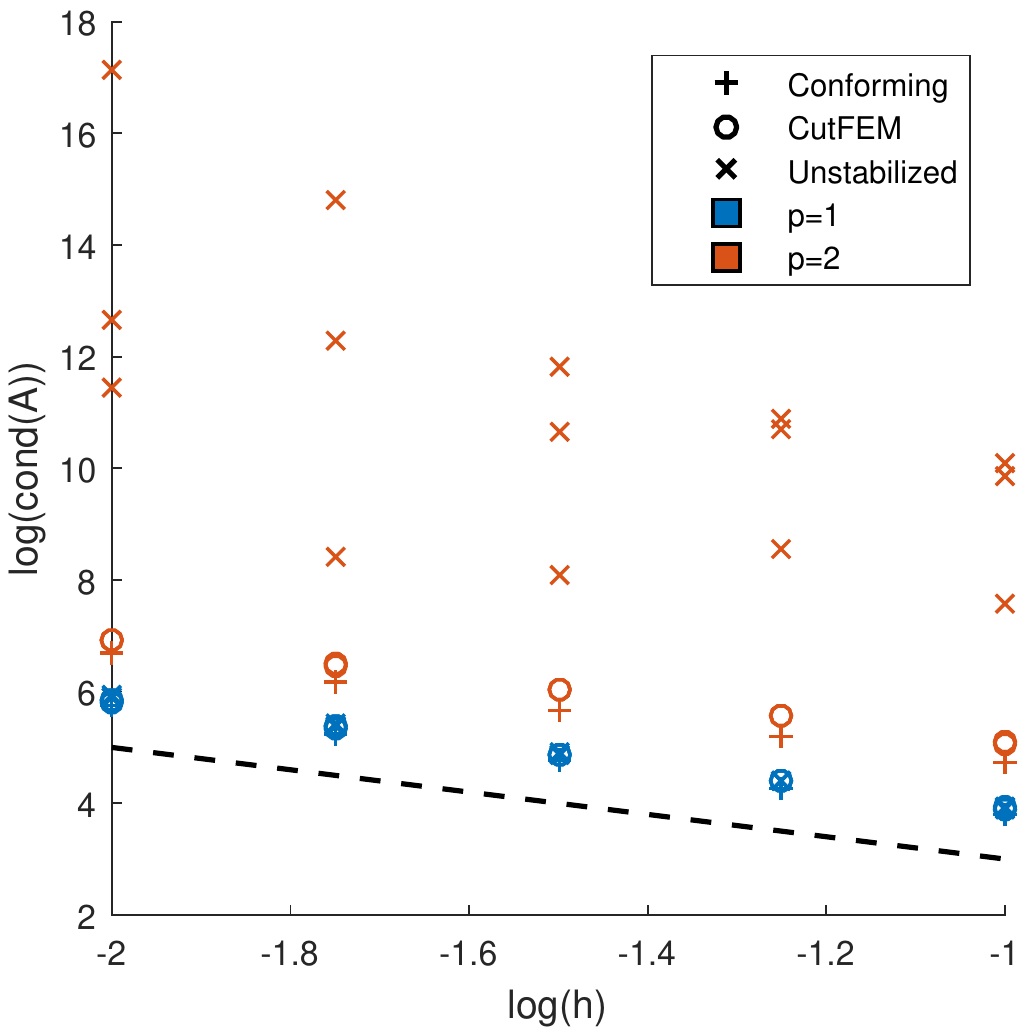}
\includegraphics[width=0.49\linewidth]{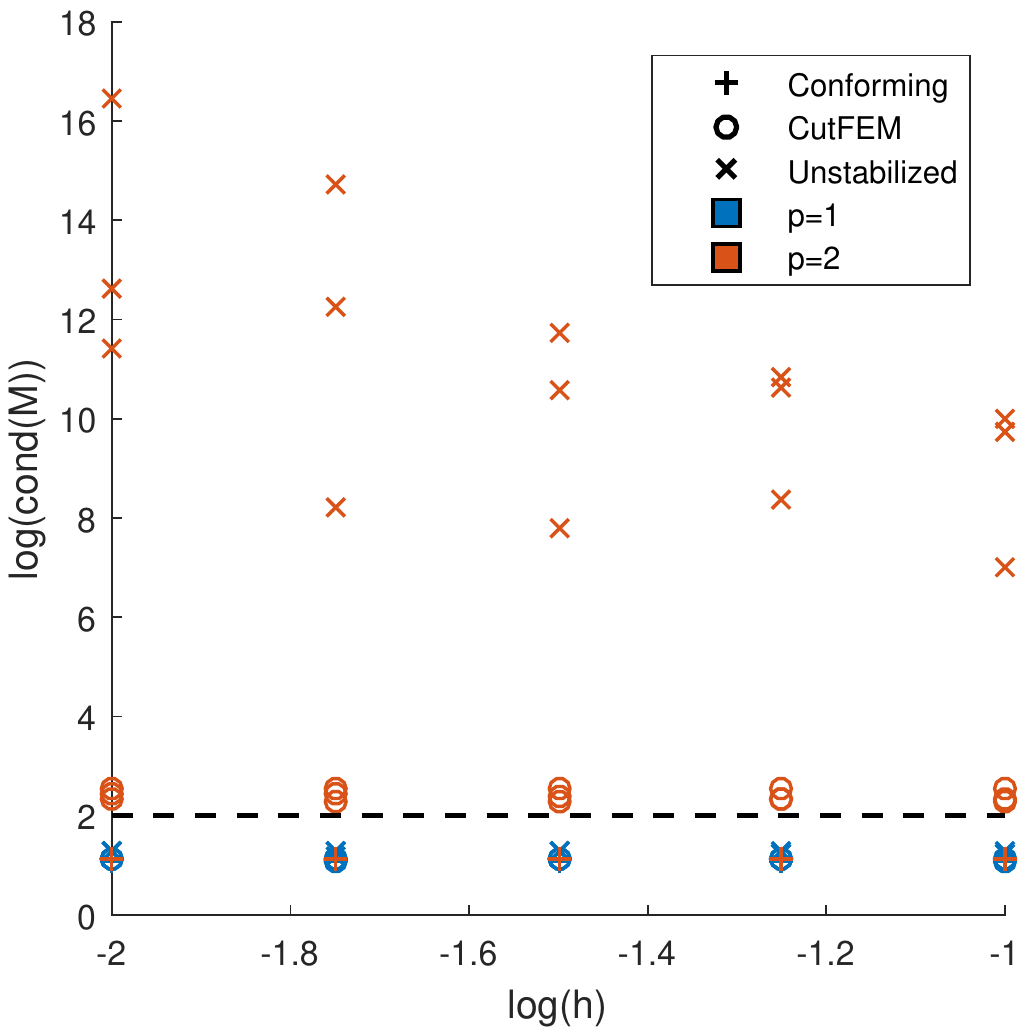}
\caption{Numerical estimation of condition numbers for the stiffness matrix (left) and the mass matrix (right) using $p=\{1,2\}$ elements. All results includes preconditioning by simple diagonal scaling.
The dashed reference lines indicate the theoretical condition number scaling of $h^{-2}$ for the stiffness matrix and constant scaling for the mass matrix.}
\label{fig:cond-conv-p1p2}
\end{figure}

\begin{figure}
\centering
\includegraphics[width=0.49\linewidth]{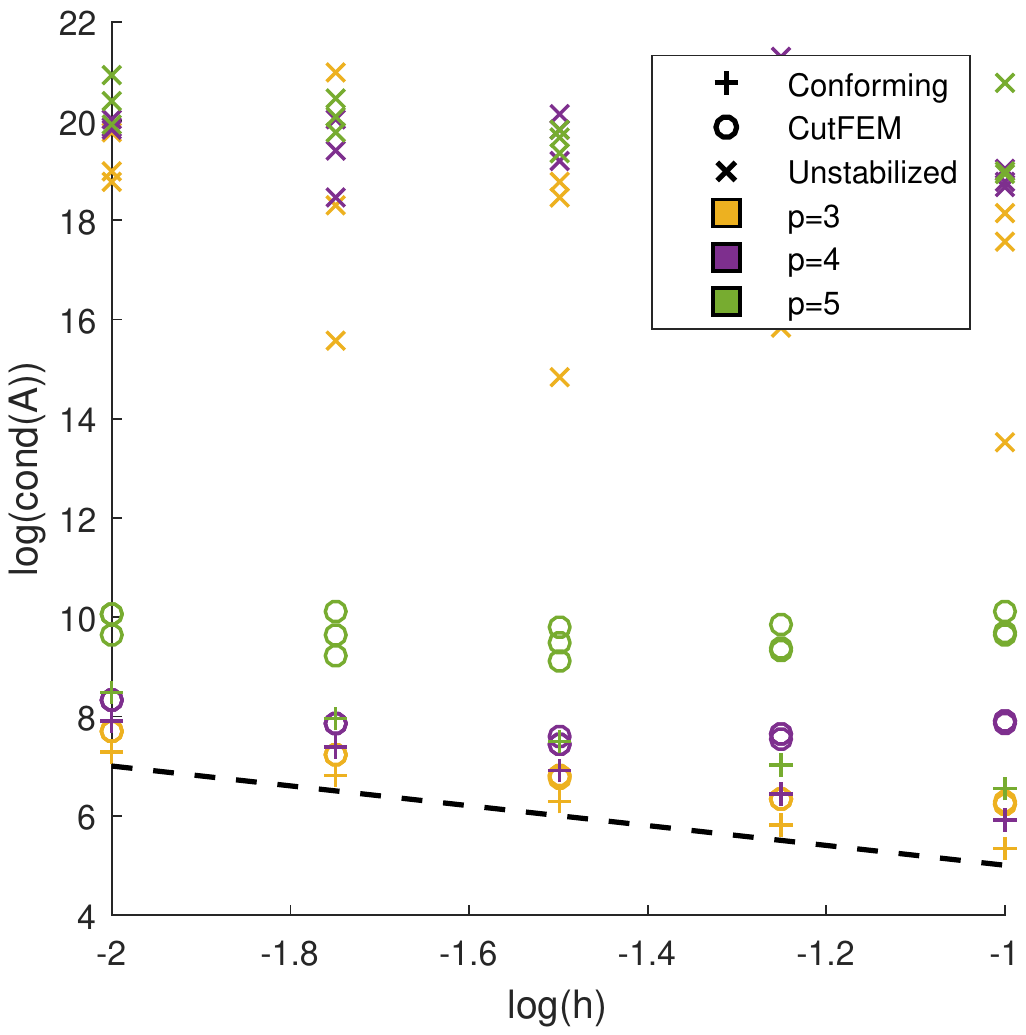}
\includegraphics[width=0.49\linewidth]{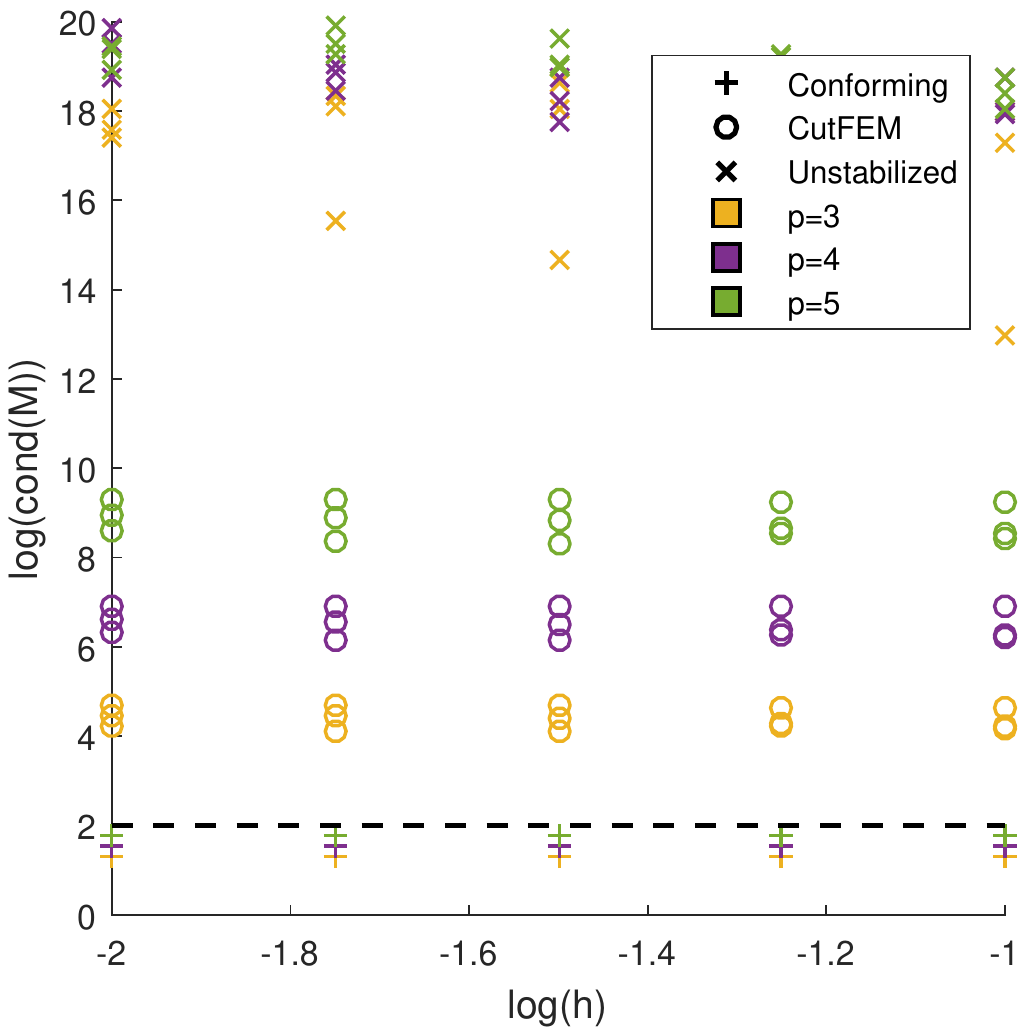}
\caption{Numerical estimation of condition numbers for the stiffness matrix (left) and the mass matrix (right) using $p=\{3,4,5\}$ elements. All results includes preconditioning by simple diagonal scaling.
The dashed reference lines indicate the theoretical condition number scaling of $h^{-2}$ for the stiffness matrix and constant scaling for the mass matrix.}
\label{fig:cond-conv-p3p4p5}
\end{figure}

\subsection{Thin Geometries}
\label{sec:thin-geom}

It is well known that low order elements ($p=1$) in elasticity problems suffer from locking on thin geometries when the number of elements in the thickness direction approaches one. This effect stems from the boundary condition $\bfsig(\bfu)\cdot\bfn = 0$ which in such situations constrains the non-zero components of the stress to the tangential plane. In CutFEM the situation may become even more extreme as the number of elements in the thickness direction can be a small fraction of an element. This effect is illustrated in Figure~\ref{fig:thin-beam} where a thin cantilever beam under gravitational load using large elements exhibits extreme locking when using $p=1$ elements while there is no such effects when using higher order elements.

Furthermore, in CutFEM it is possible for the geometry to be curved inside an element, yielding a tangential plane with non-zero curvature. To avoid locking in such situations when using coarse meshes we must use even higher order elements or increase resolution until the curvature is small compared to the element size. As an illustration of this we here consider a free ring under a centrifugal load. In Figure~\ref{fig:thin-ring} we note that while $p=2$ elements yield fairly good results for this mesh resolution we need $p=3$ elements for visually perfect results with regards to rotational symmetry for the stress distribution.

\begin{figure}
\centering
\includegraphics[width=0.6\linewidth]{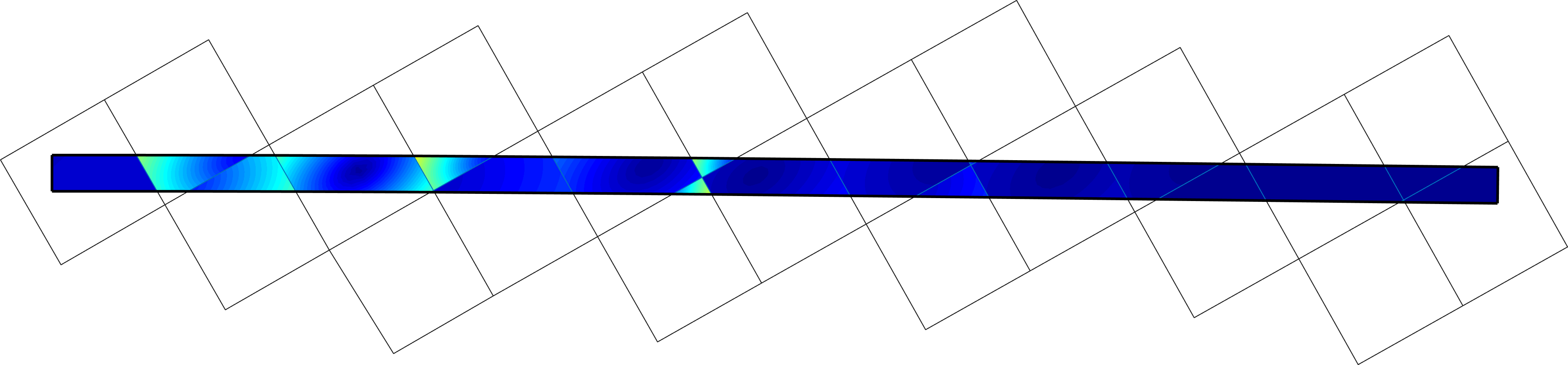} \quad
\includegraphics[width=0.30\linewidth]{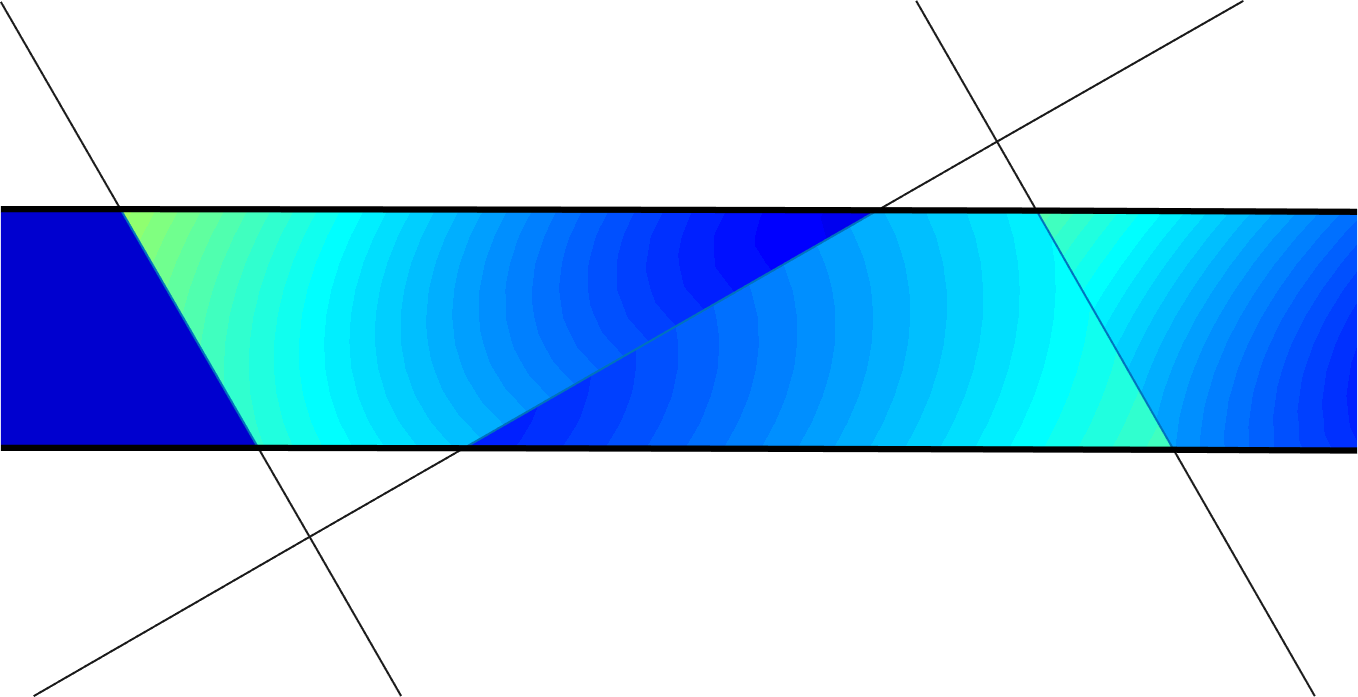}

\includegraphics[width=0.6\linewidth]{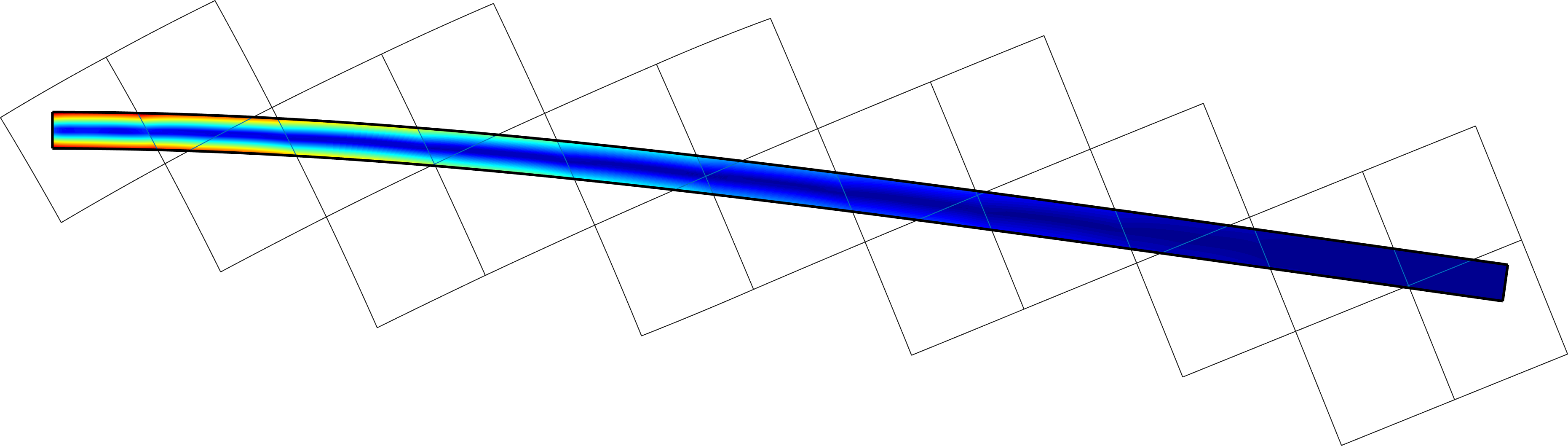} \quad
\includegraphics[width=0.30\linewidth]{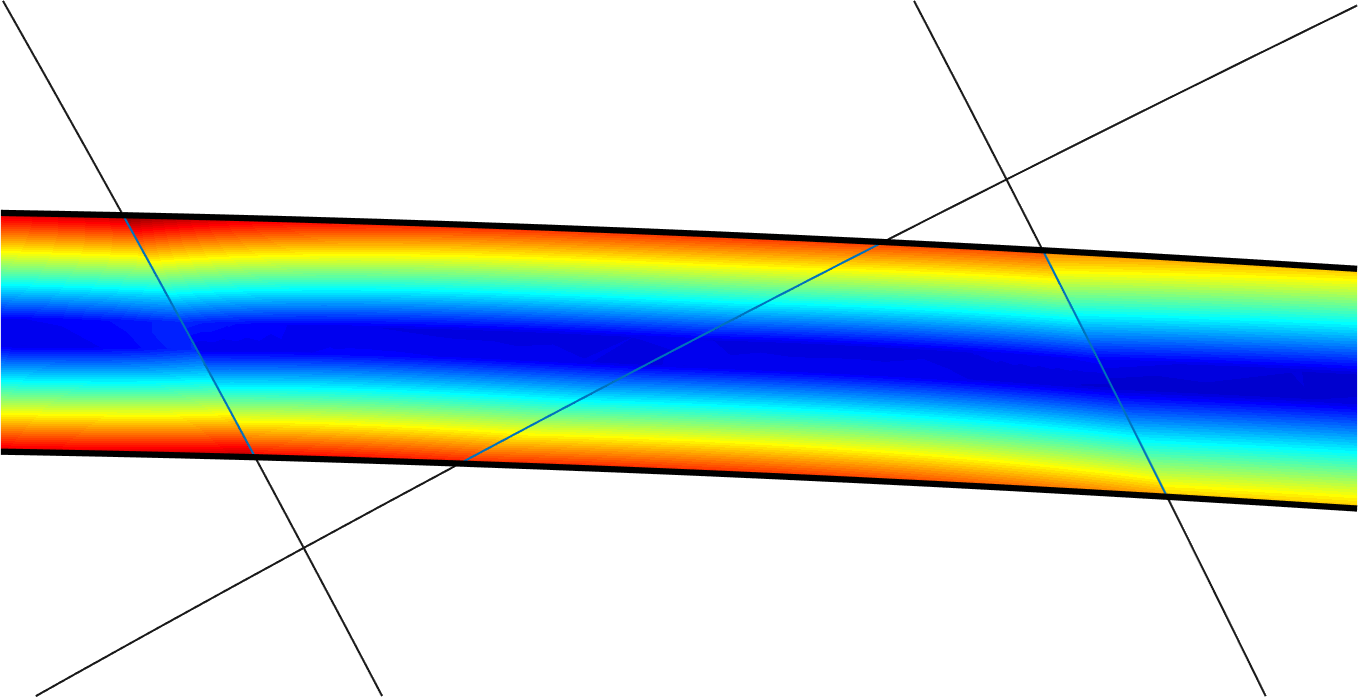}

\includegraphics[width=0.6\linewidth]{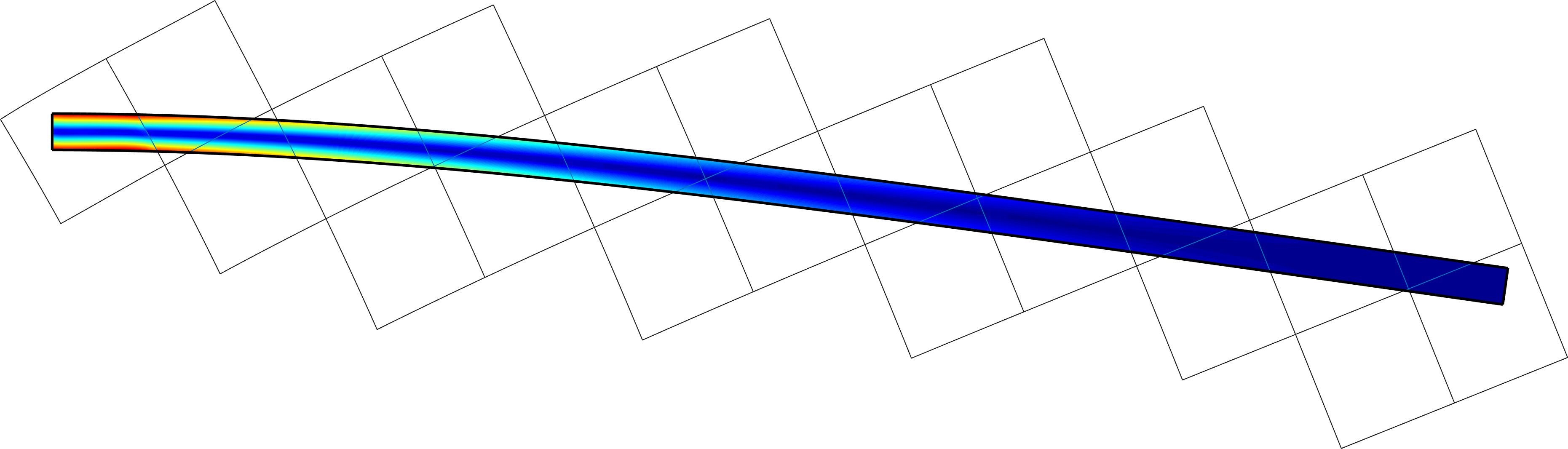} \quad
\includegraphics[width=0.30\linewidth]{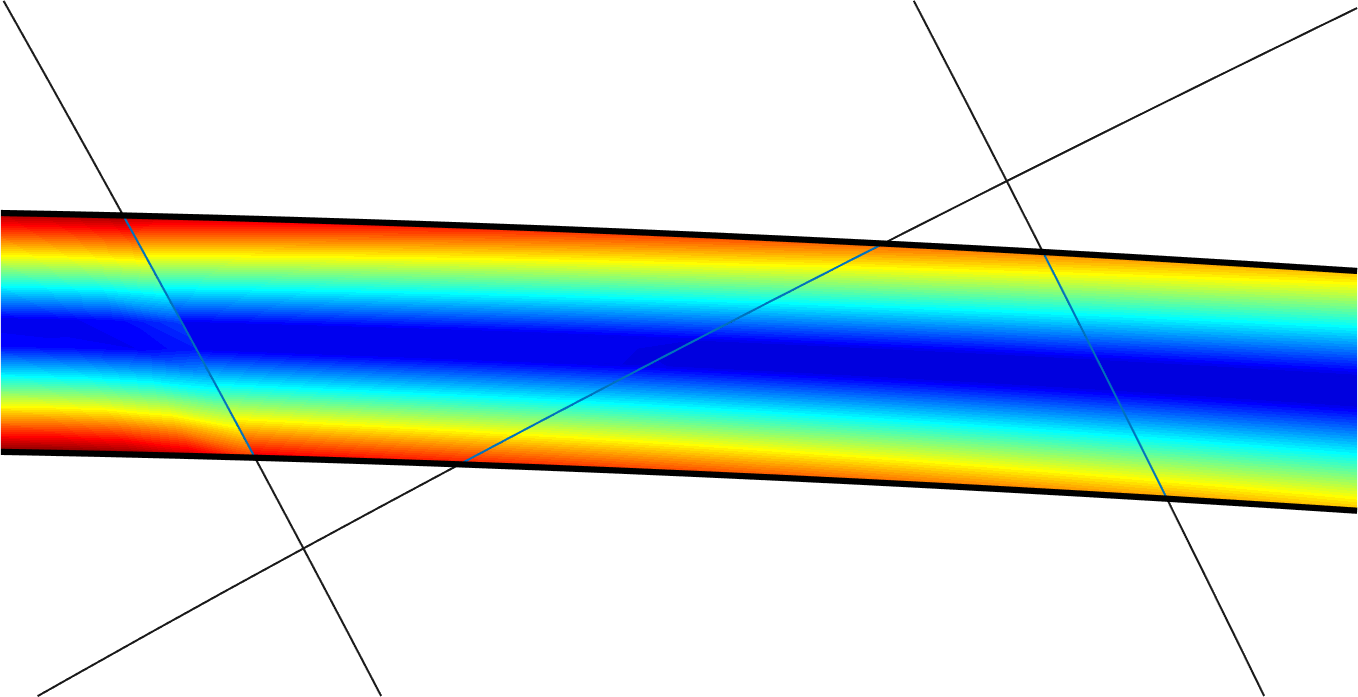}
\caption{Thin cantilever beam under gravitational loading. Von Mises stresses in a thin cantilever beam under gravitational load using $p=1,2,3$ elements. Note the extreme locking occuring when using $p=1$.}
\label{fig:thin-beam}
\end{figure}

\begin{figure}
\centering
\includegraphics[width=0.31\linewidth]{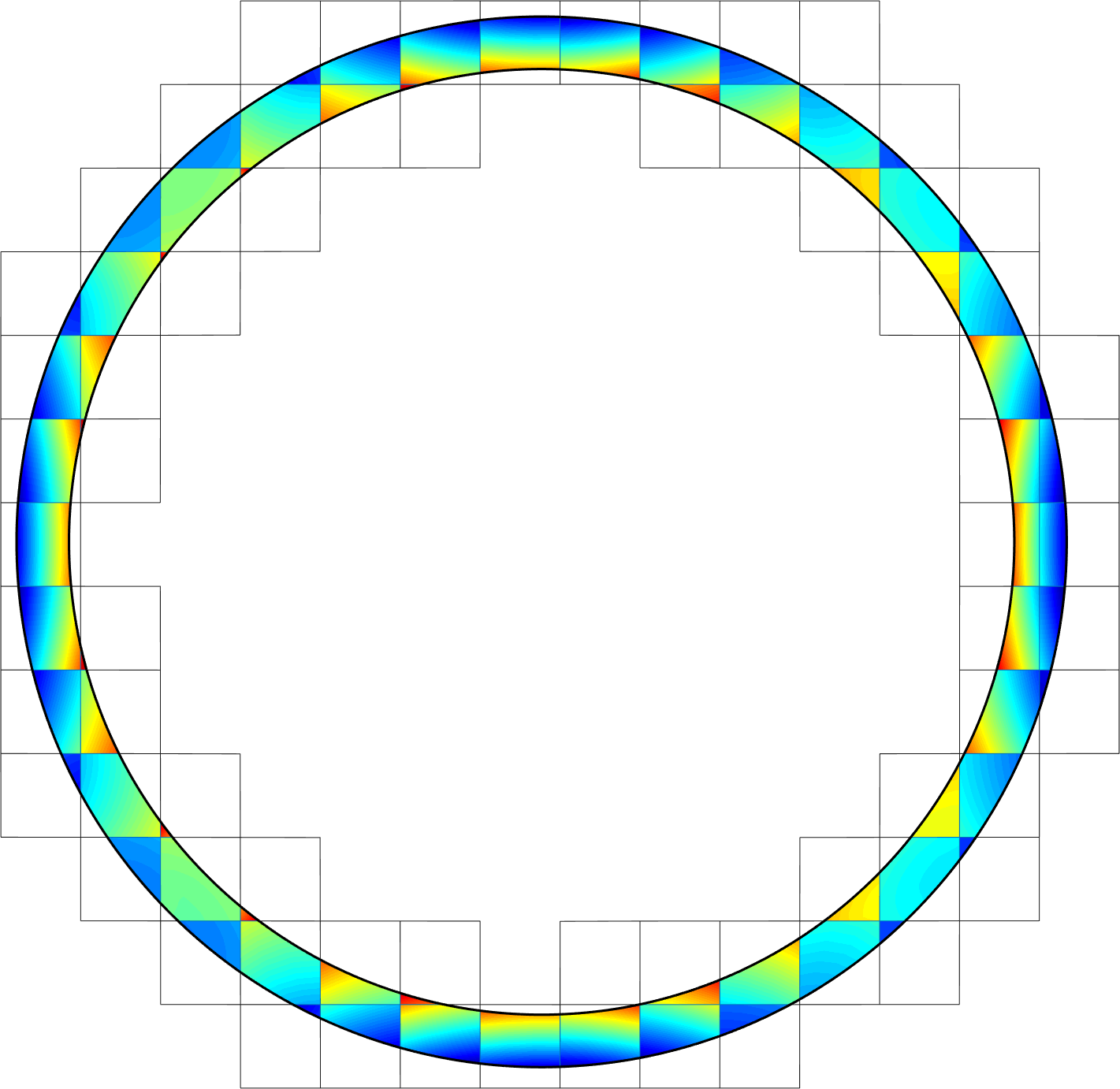} \quad
\includegraphics[width=0.31\linewidth]{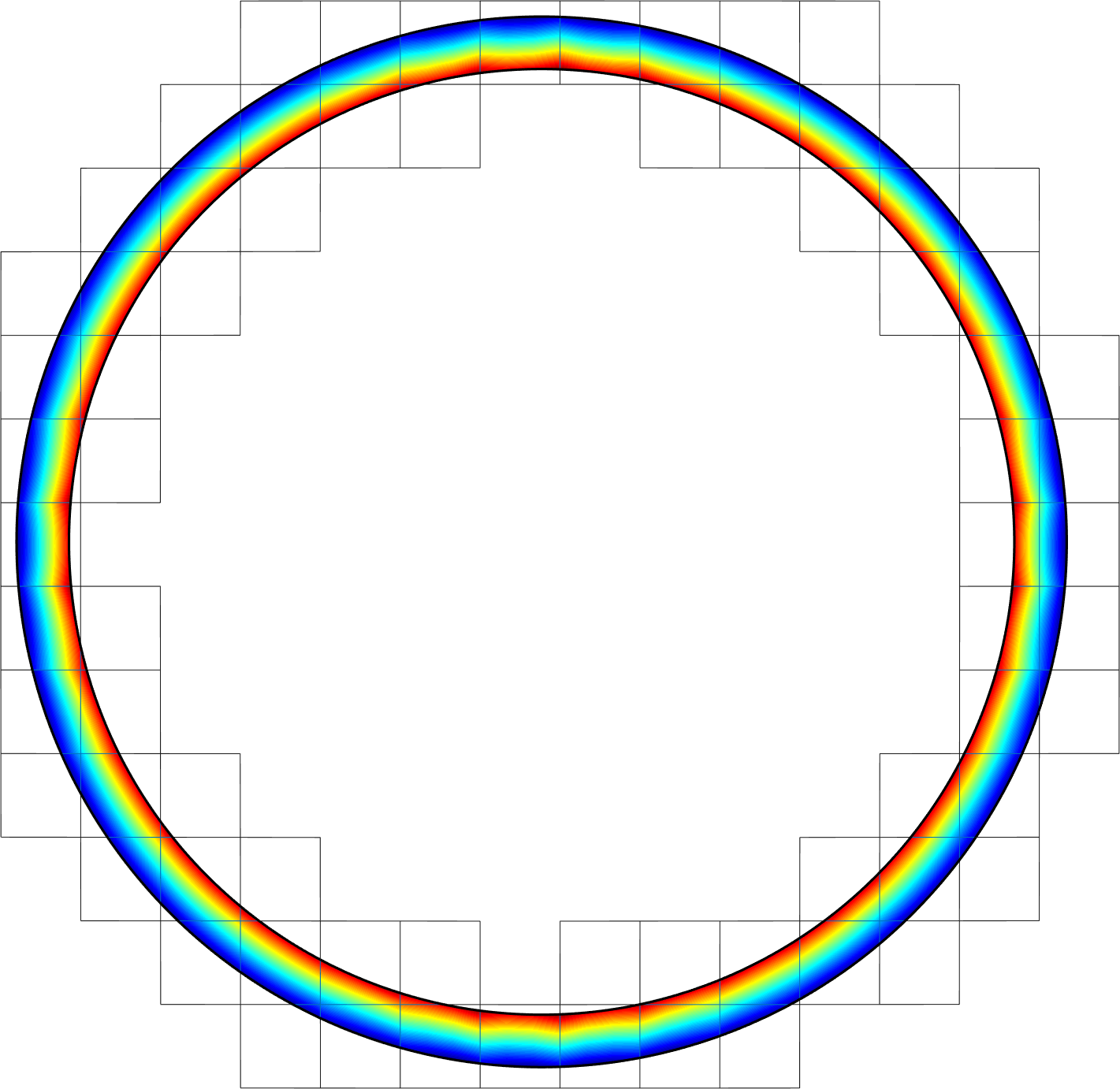} \quad
\includegraphics[width=0.31\linewidth]{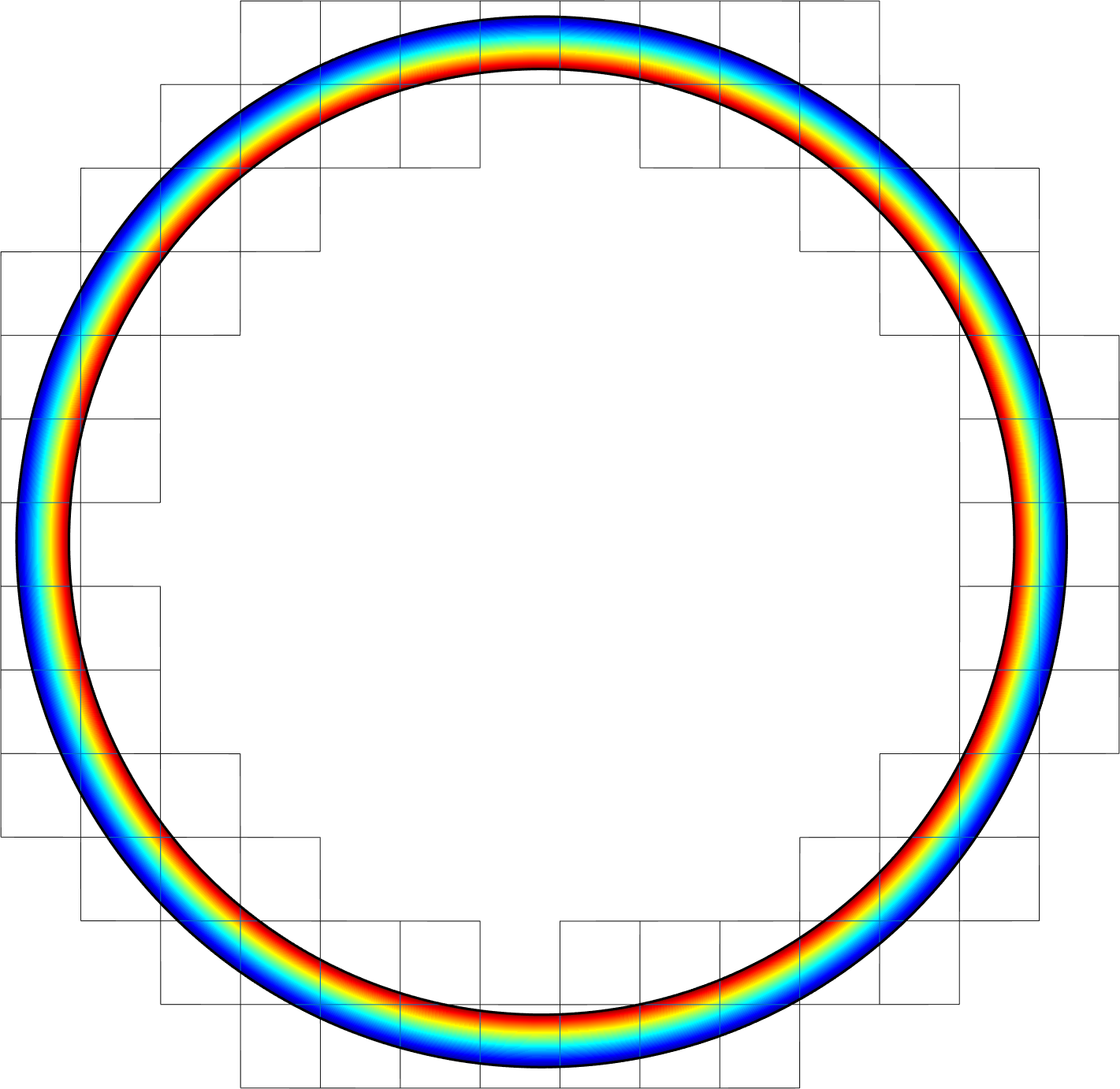}

\includegraphics[width=0.32\linewidth]{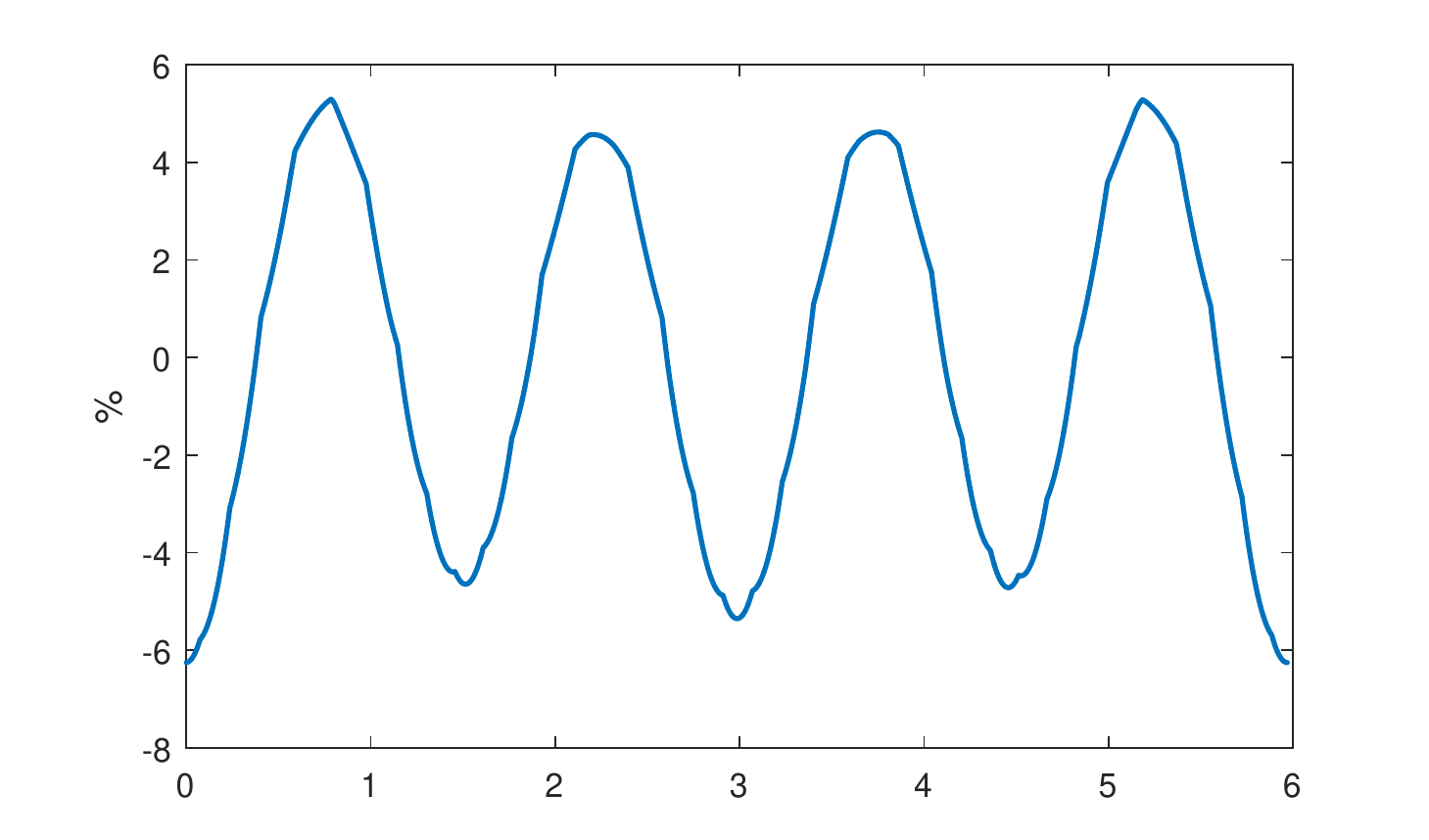} 
\includegraphics[width=0.32\linewidth]{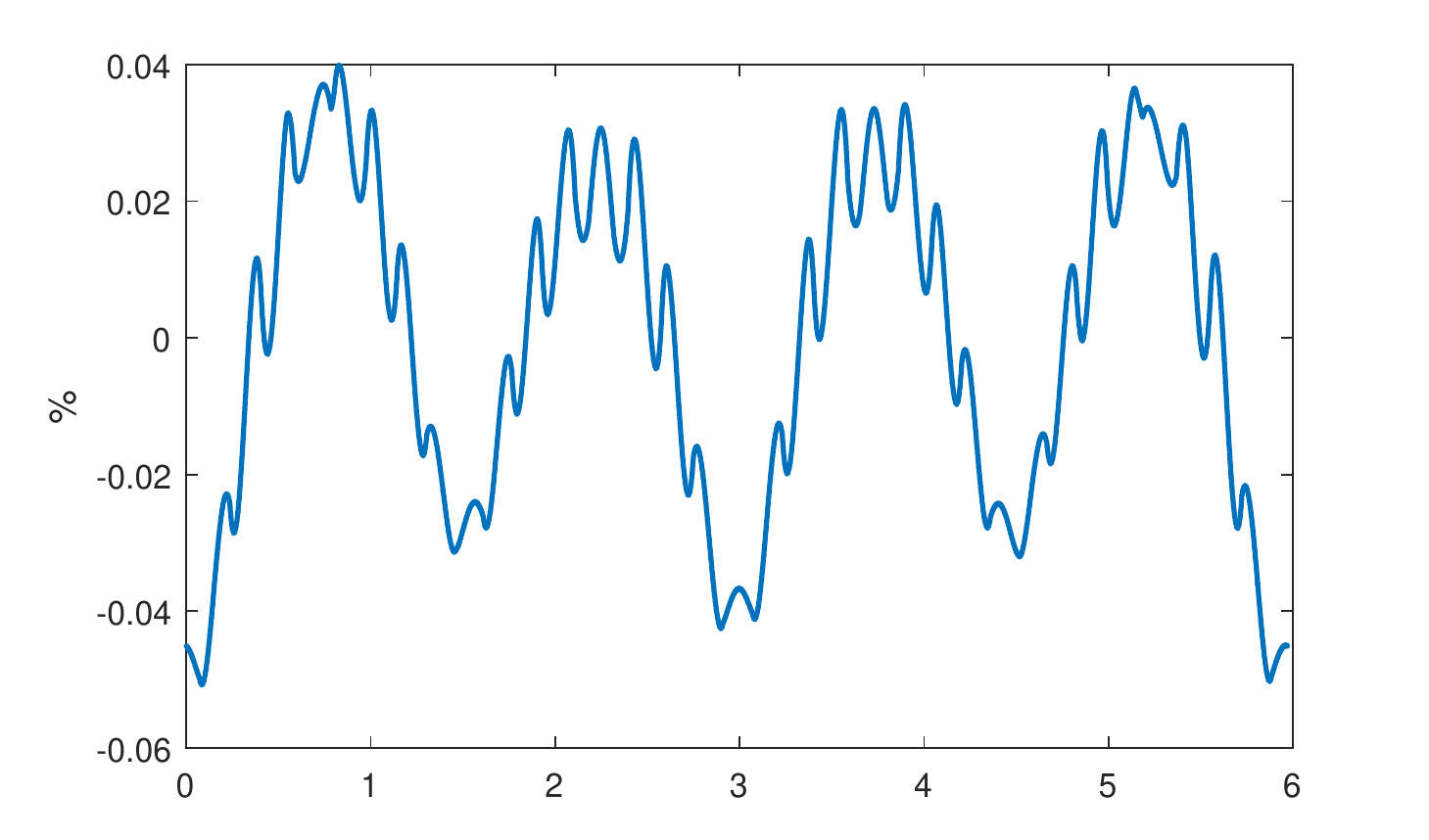} 
\includegraphics[width=0.32\linewidth]{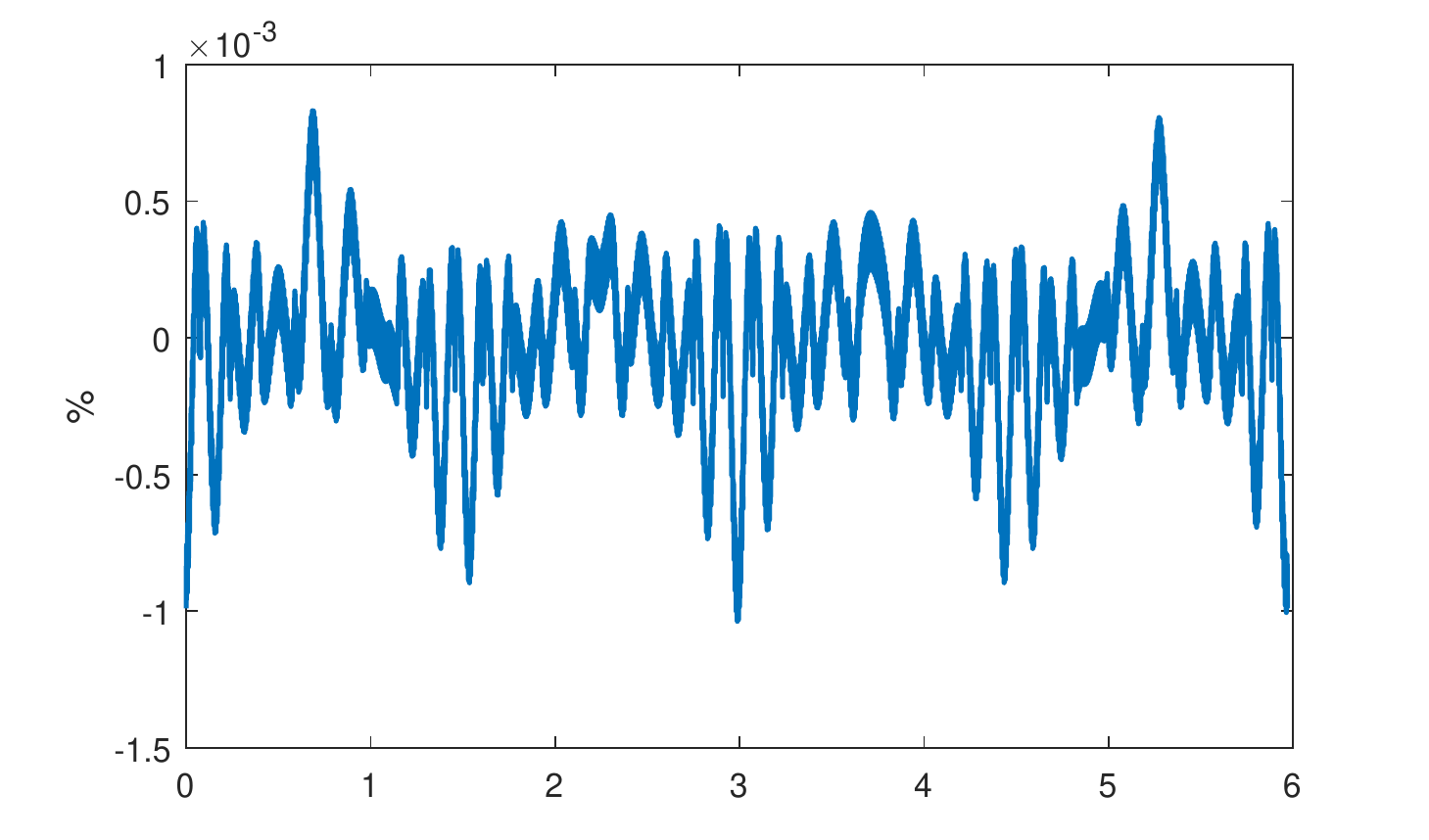}
\caption{Ring under centrifugal load. Top: Von Mises stresses for $p=\{1,2,3\}$ elements. Bottom: Variation in midline radial displacement relative to its mean.}
\label{fig:thin-ring}
\end{figure}

\subsection{Compound Bodies}
Compound bodies consist of several subdomains where the discrete solution on each subdomain is described on a separate mesh. Such a construction gives a number of powerful modeling possibilities, for example
\begin{itemize}
\item adding or changing small geometric details, for example fillets, without needing to reassemble all matrices
\item in each subdomain using different material properties, mesh resolution or elements
\item constructing objects from a number of predefined and preassembled building blocks
\end{itemize}
The weak enforcement of Dirichlet boundary conditions by Nitsche's method \cite{Nit70} in CutFEM makes it easy to construct compound bodies by a simple adaptation of Nitsche's method to interface conditions.

\begin{figure}
\centering
\includegraphics{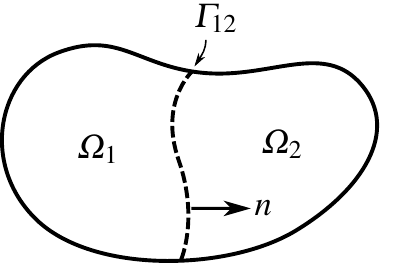}
\caption{Illustration of compound body consisting of two subdomains $\Omega_1$ and $\Omega_2$ joined at the interface $\Gamma_{12} = \partial\Omega_1 \cap \partial\Omega_2$ with prescribed normal $n = n_{\partial\Omega_1}$.}
\label{fig:compound-illustration}
\end{figure}
Consider the schematic compound body illustrated in Figure~\ref{fig:compound-illustration} for which we have the following interface conditions on $\Gamma_{12}$
\begin{subequations}
\label{eq:interface-condition}
\begin{alignat}{2}
[\bfu] &= 0 \quad &&\text{on $\Gamma_{12}$}
\label{eq:interface-u}
\\
[\bfsig(\bfu)\cdot \bfn] &= 0 \quad &&\text{on $\Gamma_{12}$}
\label{eq:interface-stress}
\end{alignat}
\end{subequations}
Summing the standard CutFEM formulations for all subdomains, with discrete functions on each subdomain described on its own background grid (and possibly its own type of elements), we for this example have the following interface term
\begin{align}
I_{12} = -(\bfsig(&\bfu)\cdot \bfn_{\partial\Omega_1},\bfv )_{\partial\Omega_1 \cap\Gamma_{12}}
-
(\bfsig(\bfu)\cdot \bfn_{\partial\Omega_2},\bfv )_{\partial\Omega_2 \cap\Gamma_{12}}
\end{align}
From this term and the interface conditions \eqref{eq:interface-condition} we derive Nitsche's method for the interface conditions via the calculation
\begin{align}
I_{12}
&=
-(\bfsig(\bfu)\cdot \bfn,\bfv )_{\partial\Omega_1 \cap\Gamma_{12}}
+
(\bfsig(\bfu)\cdot \bfn,\bfv )_{\partial\Omega_2 \cap\Gamma_{12}}
\\&=
-(\bfsig(\bfu)\cdot \bfn - \underbrace{[\bfsig(\bfu)\cdot \bfn]/2}_{=0 \text{ by \eqref{eq:interface-stress}}},\bfv )_{\partial\Omega_1 \cap\Gamma_{12}}
%\\ \nonumber &\quad
+
(\bfsig(\bfu)\cdot \bfn - \underbrace{[\bfsig(\bfu)\cdot \bfn]/2}_{=0 \text{ by \eqref{eq:interface-stress}}},\bfv )_{\partial\Omega_2 \cap\Gamma_{12}}
\\&=
-(\underbrace{\bfsig(\bfu)\cdot \bfn - [\bfsig(\bfu)\cdot \bfn]/2}_{=\langle\bfsig(\bfu)\cdot \bfn\rangle},[\bfv] )_{\Gamma_{12}}
\\&=
-(\langle\bfsig(\bfu)\cdot \bfn\rangle,[\bfv] )_{\Gamma_{12}}
-
\underbrace{([\bfu],\langle\bfsig(\bfv)\cdot \bfn\rangle)_{\Gamma_{12}}
+
\gamma_D h^{-1}
([\bfu],[\bfv])_{\Gamma_{12}}
}_{=0 \text{ by \eqref{eq:interface-u}}}
\end{align}

As an illustration of compound bodies we in Figures~\ref{fig:L-shape1} and \ref{fig:L-shape3} consider an L-shape domain where the inside corner has been drilled to avoid a stress singularity in the solution. We have constructed this domain out of three subdomains described on three different background grids and we use $p=2$ elements to represent the solution on each mesh. Note that there is no need for the various meshes to match at the interface as the interface conditions are enforced weakly.

\begin{figure}
\centering
\includegraphics[width=0.7\linewidth]{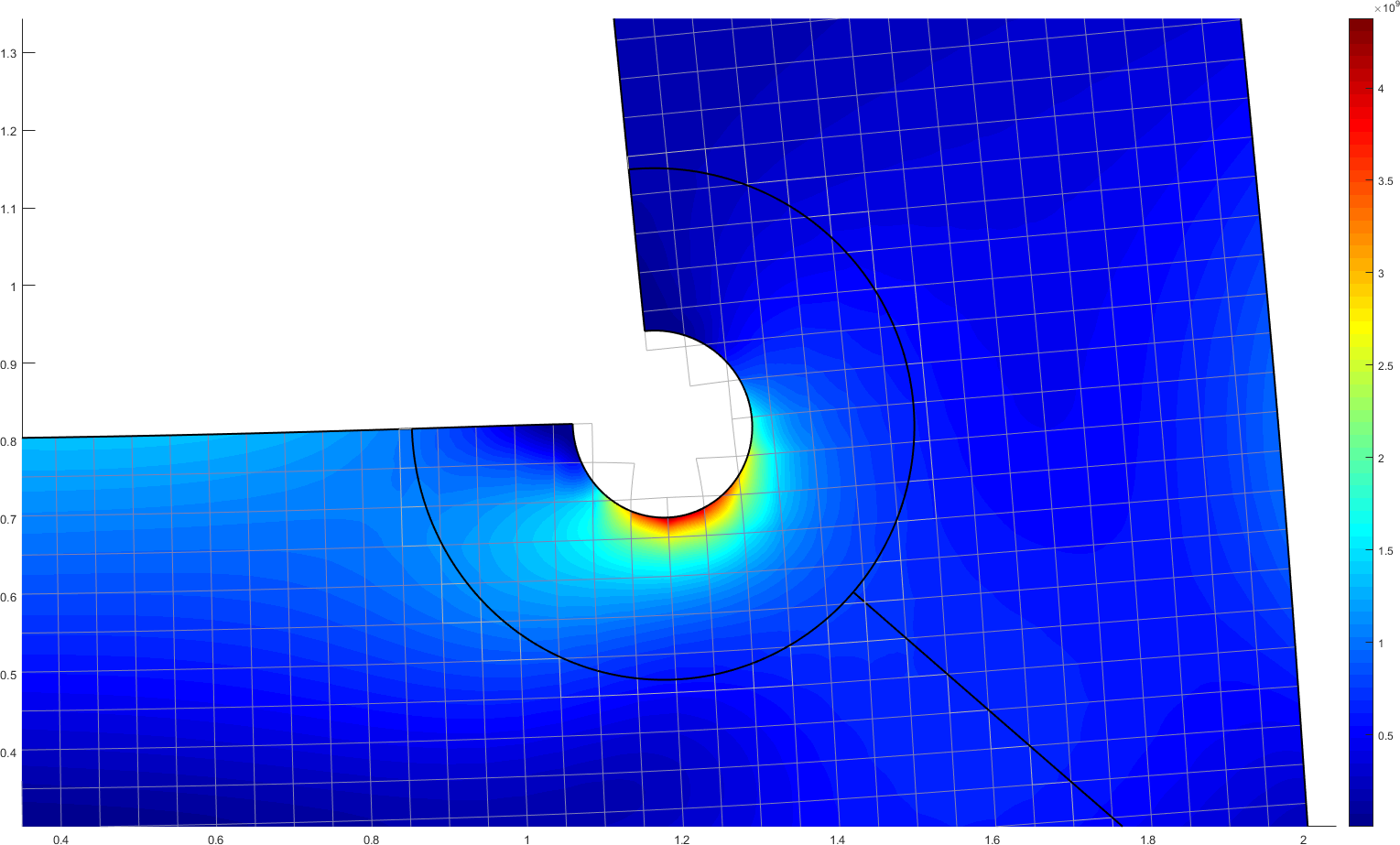}
\caption{Detail of drilled aluminium L-shape constructed as 3 separate bodies. Note the perfect flow of the von-Mises stress over the interfaces.}
\label{fig:L-shape1}
\end{figure}

\begin{figure}
\centering %[t]
\includegraphics[width=0.7\linewidth]{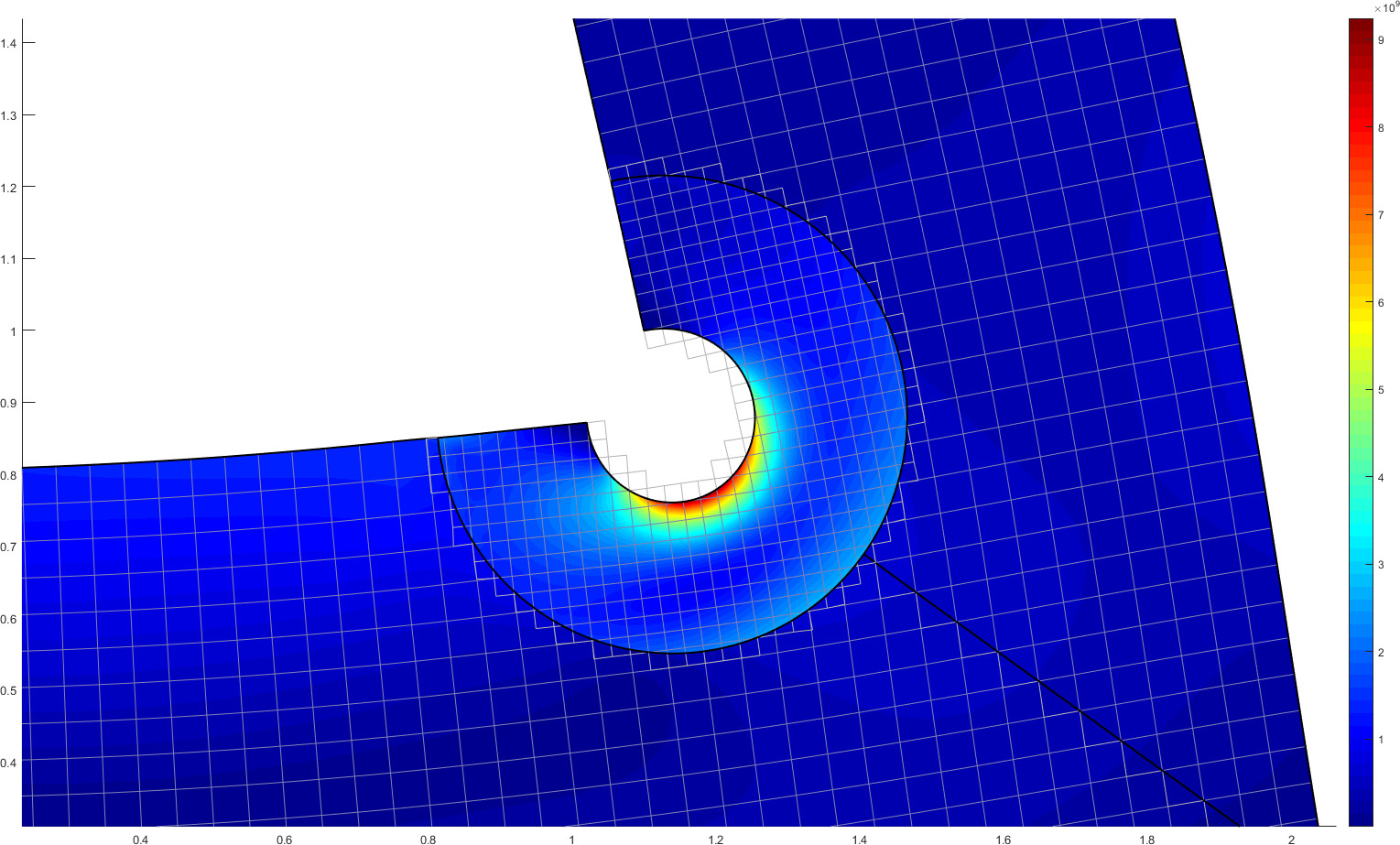}
\caption{Detail of drilled L-shape constructed as 3 separate bodies. Inner ring is steel, outer part of domain has 1/10 the stiffness of steel. Inner ring mesh also has higher resolution. Note the distinct discontinuity of the von-Mises stress over the interface which is due to the discontinuity of material stiffness.}
\label{fig:L-shape3}
\end{figure}

\subsection{Two-Grid Eigenvalue Estimation}

The use of structured background grids in CutFEM is very convenient when employing multi-grid techniques.
We here illustrate this with a two-grid method for estimation of eigenvalues \cite{XuZh01}. Consider two background grids, one coarse grid with mesh size $H$ and one fine with mesh size $h$, constructed such that the fine grid is a refinement of the coarse grid. The two-grid algorithm reads
\begin{enumerate}
\item Solve the eigenvalue problem on the coarse background grid: Find $(\bfu_H,\lambda_H) \in V_H\times\mathbb{R}$ such that
\begin{equation}
A_H(\bfu_H,\bfv) - \lambda_H m_H(\bfu_H,\bfv) = 0 \, , \quad\forall \bfv\in V_H
\end{equation}

\item Solve a single linear problem on the fine background grid: Find $\bfu_h\in V_h$
\begin{equation}
A_h(\bfu_h,\bfv) = \lambda_H m_h(\bfu_H,\bfv) \, , \quad\forall \bfv\in V_h
\end{equation}

\item Compute the Rayleigh quotient
\begin{equation}
\lambda_h = \frac{\| \bfu_h \|_{a_h}}{\| \bfu_h \|_{m_h}}
\end{equation}

\end{enumerate}

The analysis in \cite{XuZh01} for a conforming two-grid method implies that the optimal choice of mesh size $H$ for the coarse grid is given by the relationship
\begin{align}
H \propto \sqrt{h}
\end{align}
However in a cut situation, especially in light of the results in Section~\ref{sec:thin-geom}, it is very reasonable to assume that the coarse grid mesh size must be small enough to give an acceptable approximation of the eigenvalues and eigenmodes for this procedure to be efficient.

As a simple numerical example we consider the eigenvalue problem of a free ring with coarse and fine grids illustrated in Figure~\ref{fig:two-grid-mesh}. As this is a free eigenvalue problem we in step 2 of the algorithm must seek the solution $\bfu_h\in \bfV_h/\bm{RM}$. The resulting approximations of the 10:th eigenvalue and the corresponding eigenmode are presented in Figure~\ref{fig:two-grid-example} together with a reference solution computed using high order parametric conforming FEM.

\begin{figure}
\centering
\includegraphics[width=0.45\linewidth]{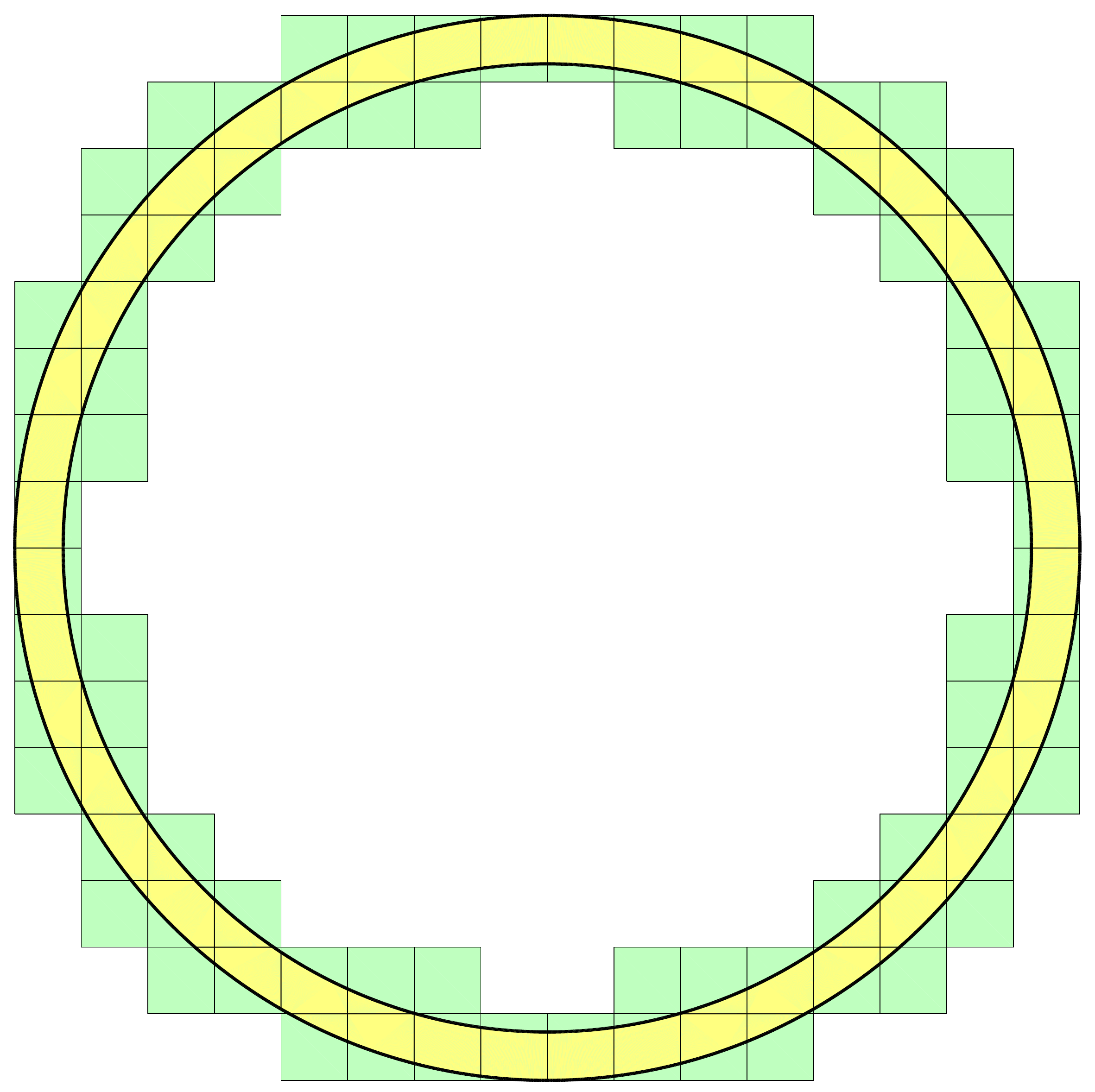}\quad
\includegraphics[width=0.45\linewidth]{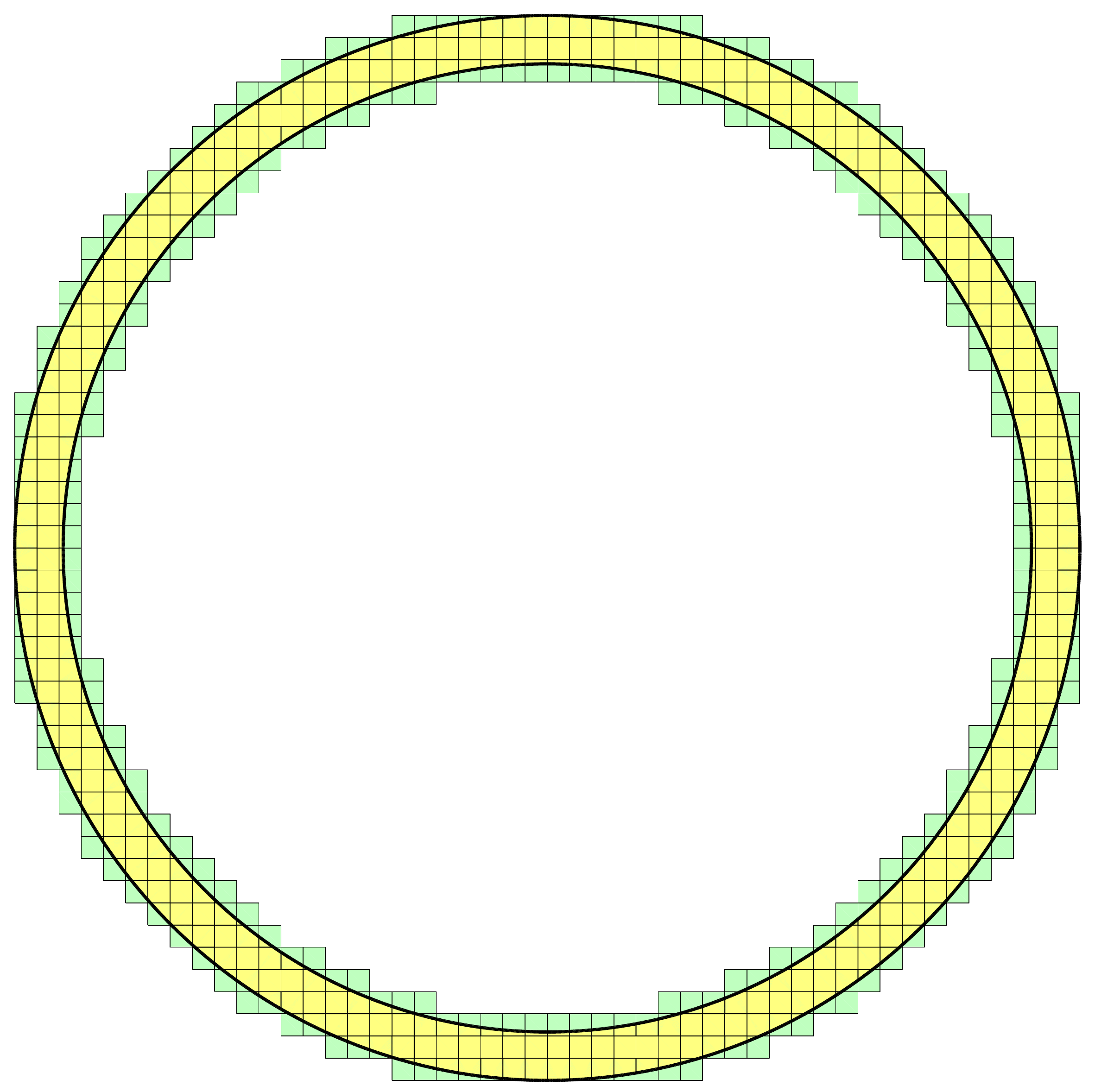}
\caption{Meshes used in two-grid estimation of the 10:th eigenvalue in the free eigenvalue problem for a steel ring. Here mesh sizes $H=0.1$ and $h = H/3$ are used for the coarse and fine scale meshes, respectively.}
\label{fig:two-grid-mesh}
\end{figure}

\begin{figure}
\centering
\includegraphics[width=0.32\linewidth]{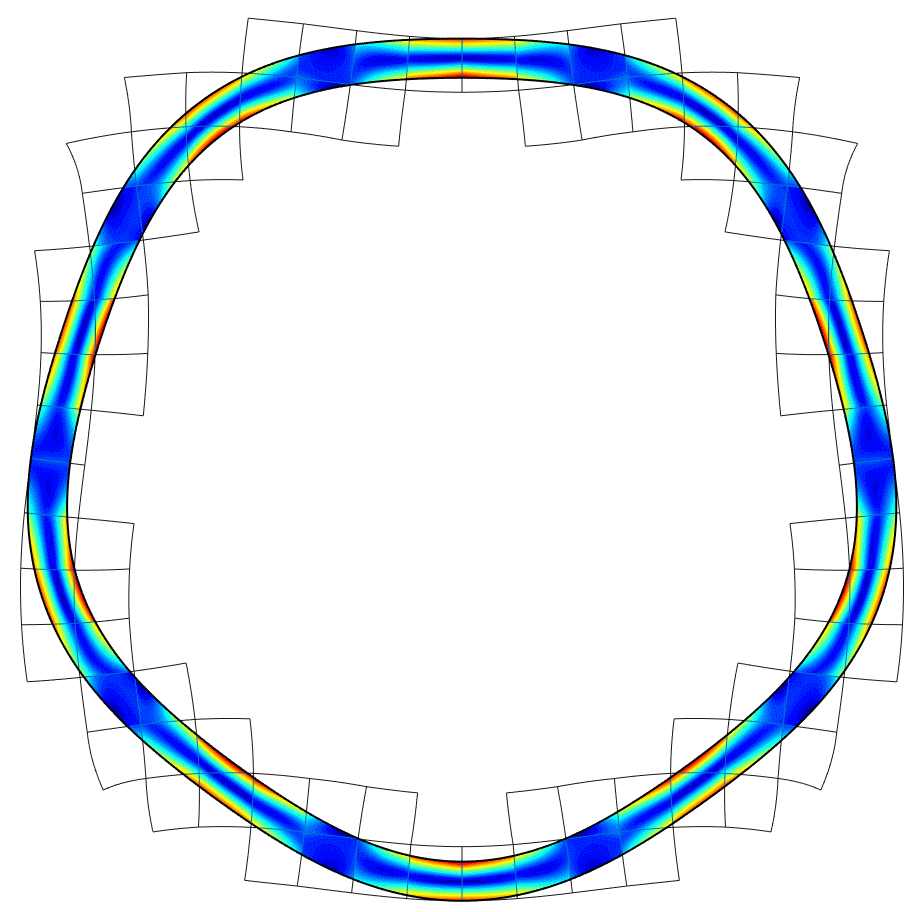}\
\includegraphics[width=0.32\linewidth]{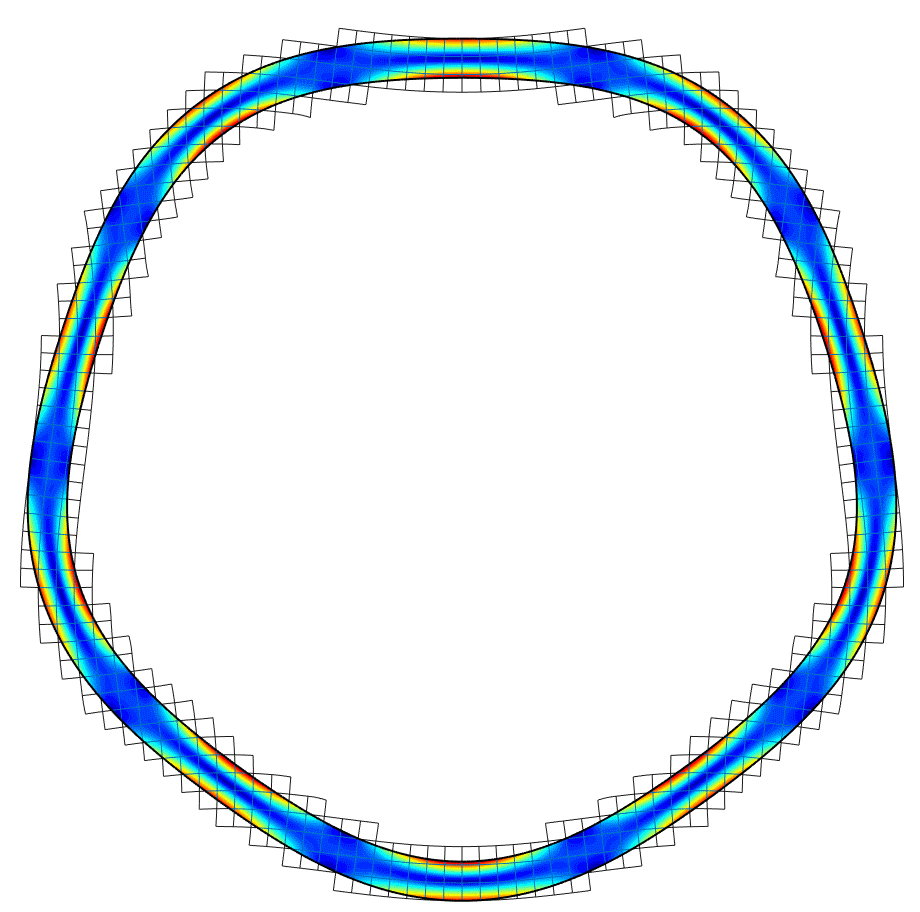}\
\includegraphics[width=0.32\linewidth]{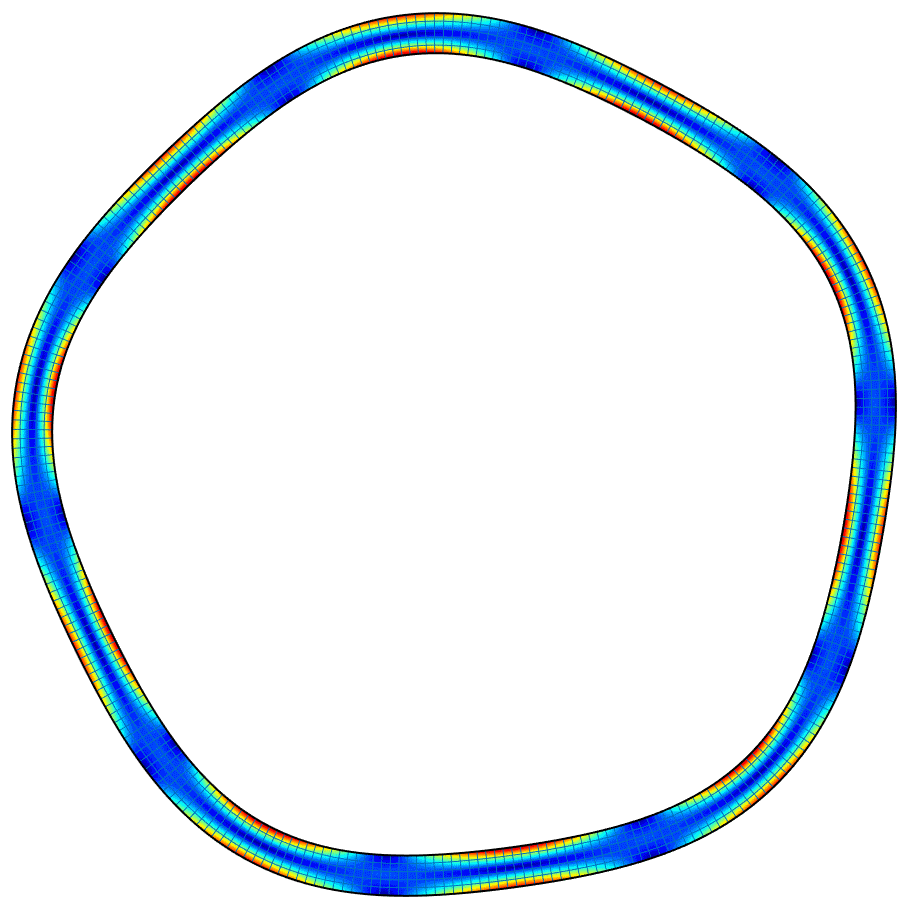}
\caption{Eigenmodes associated with the 10:th eigenvalue computed using CutFEM on a coarse grid (left), CutFEM two-grid estimation on a subgrid (middle), and parametric conforming FEM (right). CutFEM solutions are computed using $p=2$ elements while the reference FEM solution is computed using $p=4$ elements. Corresponding eigenvalues are $\lambda_H = 9.99836561\cdot 10^6$, $\lambda_h = 9.89445550\cdot 10^6$, and $\lambda_\mathrm{ref}=9.891710308\cdot 10^6$. Note that due to the rotational symmetry of the problem, the rotation of the eigenmode is undetermined.
}
\label{fig:two-grid-example}
\end{figure}

\section{Modeling of Embedded Structures with Application to Fibre Reinforced Materials}

In this Section we consider the embedding of membranes, thin plates, 
truss and beam elements into a higher dimensional elastic body. 
Simulations of this type are suitable for reinforced structures,
e.g., reinforced concrete, fiber reinforced polymers, and composite 
materials in general.

The approach is based on:
\begin{itemize}
\item Given a continuous finite element space, based on at least second 
second order polynomials, we define the finite element space for the 
thin structure as the restriction of the bulk finite element space 
to the thin structure which is geometrically modeled by an embedded 
curve or surface. 
\item  
To formulate a finite element method on the restricted or trace finite 
element space we employ continuous/discontinuous Galerkin approximations 
of plate and beam models. 
\end{itemize}
The thin structures are then modelled using the CutFEM paradigm and the 
stiffness of the embedded structure is in the most basic version, which we 
consider here, simply added to the bulk stiffness. This in turn means that  
the bulk stabilizes the CutFEM contributions and no additional stabilization 
is needed. In this initial report, we limit ourselves to 
two-dimensional bulk problems and thus the embedded structures will be 
limited to trusses and beams. In future work, the handling of Kirchhoff plates 
and beams and trusses of codimension 2 will be considered. The bulk problem may also be 
viewed as an interface problem on the case of an embedded surface in three 
dimensions or curve in two dimensions in order to more accurately approximate 
the bulk problem in the vicinity of the thin structure.

The work presented here is and extension of earlier work \cite{CeHaLa16}
where only membrane structures were considered, in which case a linear 
approximation in the bulk suffices.

\subsection{Modeling of Trusses and Beams}

Consider the elasticity equations (\ref{eq:standard}) in 
$\Omega\subset \mathbb{R}^2$. Embedded in $\Omega$ we have 
straight truss and beam elements. Curved elements can be 
modeled using piecewise affine approximations of the curve, 
following \cite{HaLa17}.

We consider modeling the truss and beam using tangential 
differential calculus. To this end we follow the exposition 
in \cite{HaLaLa14} and then make some simplifying assumptions. 
Let then $\Sigma$ denote line embedded in $\mathbb{R}^2$, with 
tangent vector $\bft$.  We let 
$\bfp:\mathbb{R}^2\rightarrow \Sigma$ be the closest point 
mapping, i.e. $\bfp (\bfx ) = \bfy$ 
where $\bfy \in \Sigma$ minimizes the Euclidean norm $|  \bfx - \bfy |_{\IR^3}$.
We define $\zeta$ as the signed distance function $\zeta(\bfx) := \pm\vert\bfx - \bfp(\bfx)\vert$, positive on one side of $\Sigma$ and negative 
on the other. The line $\Sigma$ is assumed to be the center line 
of a truss element with thickness $t$, which we for simplicity 
assume is constant. 

The  linear projector
$\bfP_\Sigma = \bfP_\Sigma(\bfx)$, onto the tangent line of 
$\Sigma$ at $\bfx\in\Sigma$, is given by
\begin{equation}
\bfP_\Sigma = \bft\otimes\bft . 
\end{equation}

Based on the assumption that planar cross sections orthogonal 
to the midline remain plane
after deformation we assume that the displacement takes the form 
\begin{equation} \label{displacementfield}
\bfu = \bfu_{0} + \theta  \zeta \bft  
\end{equation}
where $\bfu_{0}:\Sigma \rightarrow \mathbb{R}^2$ is the deformation of the midline, decomposed as 
\begin{equation}
\bfu_0 = u_n\bfn+u_t\bft , \quad u_n :=\bfu_0 \cdot\bfn , \;  u_t := \bfu_0\cdot\bft,
\end{equation}
where $\bfn \perp \bft$,
and $\theta:\Sigma \rightarrow \mathbb{R}$ is an angle representing 
an infinitesimal rotation. Both $\bfu_0$ and $\theta$ are assumed 
constant in the normal plane.

In Euler--Bernoulli beam theory the beam cross-section is assumed 
plane and orthogonal to the beam midline after deformation and no 
shear deformations occur. This means that we have
\begin{align}
\theta = \bft\cdot\nabla u_n := \partial_tu_n
\end{align}
This definition for $\theta$ in combination with \eqref{displacementfield} constitutes the
Euler--Bernoulli kinematic assumption.

We assume the usual Hooke's law for one dimensional structural members
\begin{align} \label{constitutive}
\bfsig_\Sigma(\bfu) = E\bfeps_\Sigma(\bfu)
\end{align}
where $\bfeps_\Sigma(\bfu):=\bfP_\Sigma\bfeps(\bfu) \bfP_\Sigma $.
We then notice that the strain energy density can be written
\begin{equation}
E\bfeps_\Sigma(\bfu) : \bfeps_\Sigma(\bfu) = E(\bft\cdot\bfeps(\bfu)\cdot\bft)^2 = E\left((\partial_t u_t)^2+ \zeta^2(\partial_t\theta)^2\right) .
\end{equation}
Assuming computations are being done per unit length in the third 
dimension, the second moment of inertia $I ={t^3}/{12}$ and the area
$A  = t$ and the total energy of the truss and beam structures are 
postulated as
\begin{equation}
{\cal E} = {\cal E}_\text{T}+{\cal E}_\text{B}
\end{equation}
where the truss energy
\begin{equation}
{\cal E}_\text{T}:= \frac12\int_\Sigma
E A \left(\partial_t  u_t\right)^2   \, d\Sigma
- \int_\Sigma A\bff\cdot\bft \, u_t\, d\Sigma 
\end{equation}
and the beam energy
\begin{equation}
{\cal E}_\text{B}:= \frac12\int_\Sigma
E I  \,\left(\partial^2_{tt} u_n \right)^2   \, d\Sigma
 -\int_\Sigma A \bff\cdot \bfn\, u_n\, d\Sigma
\end{equation}
where $\partial^2_{tt}: =\partial_{t}\partial_{t}$. Assuming zero 
displacements and rotations at the end points of $\Sigma$, we 
thus seek $(u_t,u_n)\in H^1_0(\Sigma)\times H^2_0(\Sigma)$ such that
\begin{equation}
\int_\Sigma
E A \partial_t  u_t \partial_t  v_t \, d\Sigma + \int_\Sigma
E I  \,\partial^2_{tt} u_n \partial^2_{tt} v_n  \, d\Sigma
= \int_\Sigma A(f_t  v_t+f_n  v_n ) \, d\Sigma 
\end{equation}
for all $(v_t,v_n)\in H^1_0(\Sigma)\times H^2_0(\Sigma)$.

\subsection{Finite Element Discretization}

In analogy with the preceding discussion, we now
let $\mcK_{h}$ be a  quasiuniform 
partition, with mesh parameter $h$, of the bulk domain $\Omega$ 
into shape regular elements $K$, and denote the set of faces 
in $\mcK_{h}$ by $\mcF_{h}$. The elements of the bulk mesh cut 
by $\Sigma$, and the corresponding element faces, are denoted 
\begin{align}
\mcK_h(\Sigma) &=
\{ K \in \mcK_h \, : \, \overline{K} \cap \Gamma \neq \emptyset \}
\\
\mcF_h(\Sigma) &=
\{ F \in \mcF_h \, : \, F \cap \partial K \neq \emptyset \, , \ K \in \mcK_h(\Sigma) \}
\end{align}
The intersection points between $\Sigma$ and element faces in 
$\mcK_h(\Sigma)$ is denoted
\begin{equation}
\mcP_h(\Sigma)= \{ \bfx: \, \bfx = F \cap\Sigma, \,F \in \mcF_h \} ,
\end{equation}
and we assume that this is a discrete set of points (thus excluding 
the case where any $F\in\mcF_h$ coincides with a part of $\Sigma$ 
of nonzero measure). 
Setting as above
\begin{equation}
a(\bfv,\bfw) := 2\mu (\bfeps(\bfv),\bfeps(\bfw))_\Omega 
+ \lambda (\text{tr}(\bfeps(\bfv)),\text{tr}(\bfeps(\bfw)))_\Omega
\end{equation}
and introducing
\begin{equation}
b(\bfv,\bfw) := (E A \partial_t  (\bfv\cdot\bft) ,\partial_t (\bfw\cdot\bft))_\Sigma
\end{equation}
and
\begin{equation}
c(\bfv,\bfw) :=(E I  \,\partial^2_{tt} (\bfv\cdot\bfn), \partial^2_{tt} (\bfw\cdot\bfn))_\Sigma
\end{equation}
and defining $a_\text{tot}(\bfv,\bfw) :=a(\bfv,\bfw)+b(\bfv,\bfw)+c(\bfv,\bfw)$
we propose the following continuous/discontinuous Galerkin method. With
\begin{equation}
V_h := \{\bfv:\; \bfv\vert_K\in [P^2(K)]^2\; \forall K \in \mcK_h , \; \bfv \in [C^0(\Omega)]^2, \; \bfv = \bf0\; \text{on $\partial\Omega$}\}
\end{equation}
we seek $\bfu\in V_h$ such that
\begin{align}\nonumber
(A\bff,\bfv)_\Sigma + (\bff, \bfv)_{\Omega} = {}& a_\text{tot}(\bfu,\bfv) 
-\sum_{\bfx\in\mcP_h(\Sigma)}\langle EI \partial^2_{tt} (\bfu (\bfx)\cdot\bfn)\rangle [\partial_{t} (\bfv (\bfx)\cdot\bfn)]\\ \nonumber
 & -\sum_{\bfx\in\mcP_h(\Sigma)}\langle EI \partial^2_{tt} (\bfv (\bfx)\cdot\bfn)\rangle [\partial_{t} (\bfu (\bfx)\cdot\bfn)]\\ 
& + \sum_{\bfx\in\mcP_h(\Sigma)}\frac{\beta EI}{h}[\partial_{t} (\bfu (\bfx)\cdot\bfn)][\partial_{t} (\bfv (\bfx)\cdot\bfn)]
\end{align}
for all $\bfv\in V_h$. Here
\begin{equation}
\langle f (\bfx)\rangle := \frac{1}{2}\left(f(\bfx+\epsilon \bft)+f(\bfx-\epsilon \bft)\right)\vert_{\epsilon\downarrow 0}, \quad [f(\bfx)] :=  \left(f(\bfx+\epsilon \bft)-f(\bfx-\epsilon \bft)\right)\vert_{\epsilon\downarrow 0}.
\end{equation}
The terms on the discrete set $\mcP_h(\Sigma)$ are associated with the work of the end moments on the end rotation which occur due to the lack of $C^1(\Omega)$ continuity of the approximation, cf. \cite{HaLa17}.  In this formulation, stabilization terms of the kind discussed above can be added to ensure coercivity. Due to the presence of the bulk equations, the need for these terms is however mitigiated.

\subsection{Numerical Examples}
We consider a bulk material on the domain $\Omega=(0,4)\times(0,1)$ with $E=300$, $\nu=1/3$ in plane strain.
The domain is fixed at the left end and unloaded on the rest of the boundary.
It is first loaded with a volume load $\bff = (0,-1)$ and the displacements using only the bulk material are shown in Fig. \ref{bulkdef} with corresponding
(Frobenius) norm of the stresses shown in Fig. \ref{bulkmises}. We next introduce two truss elements as shown in Fig \ref{trusses} at $y=0.249$ and at $y=0.751$.
The thickness of the trusses was set to $t=0.1$ and Young's modulus $E=10^4$. In Figs. \ref{trussdef} and \ref{trussmises} we show the corresponding deformations and norm of the stresses.
Finally we remove the trusses and add a beam at $y=0.501$, with thickness $t=0.1$ and $E=10^6$. In Figs. \ref{beamdef} and \ref{beammises} we show the corresponding deformations and norm of the stresses.

In our second example we use the same boundary conditions but pull the material using $\bff=(100,0)$. The corresponding displacements using only the bulk material are shown in Fig. \ref{pulledbulk} with associated norm of the stresses shown in Fig. \ref{pulledbulkmises}. We then insert the truss elements from the previous example and apply the load again to obtain the deformation shown in Fig. \ref{pulledtruss} with associated norm of stresses in Fig. \ref{pulledtrussmises}.

In all these examples we see that the added lower dimensional structural elements have a profound impact on the deformation and stress distribution in the bulk material.
The possibility of modeling lower dimensional structural elements without meshing is clearly beneficial and can be used for example in optimization algorithms where the 
structural members are required to move around to find their optimal shape and place.

\begin{figure}
\centering
\includegraphics[width=0.7\linewidth]{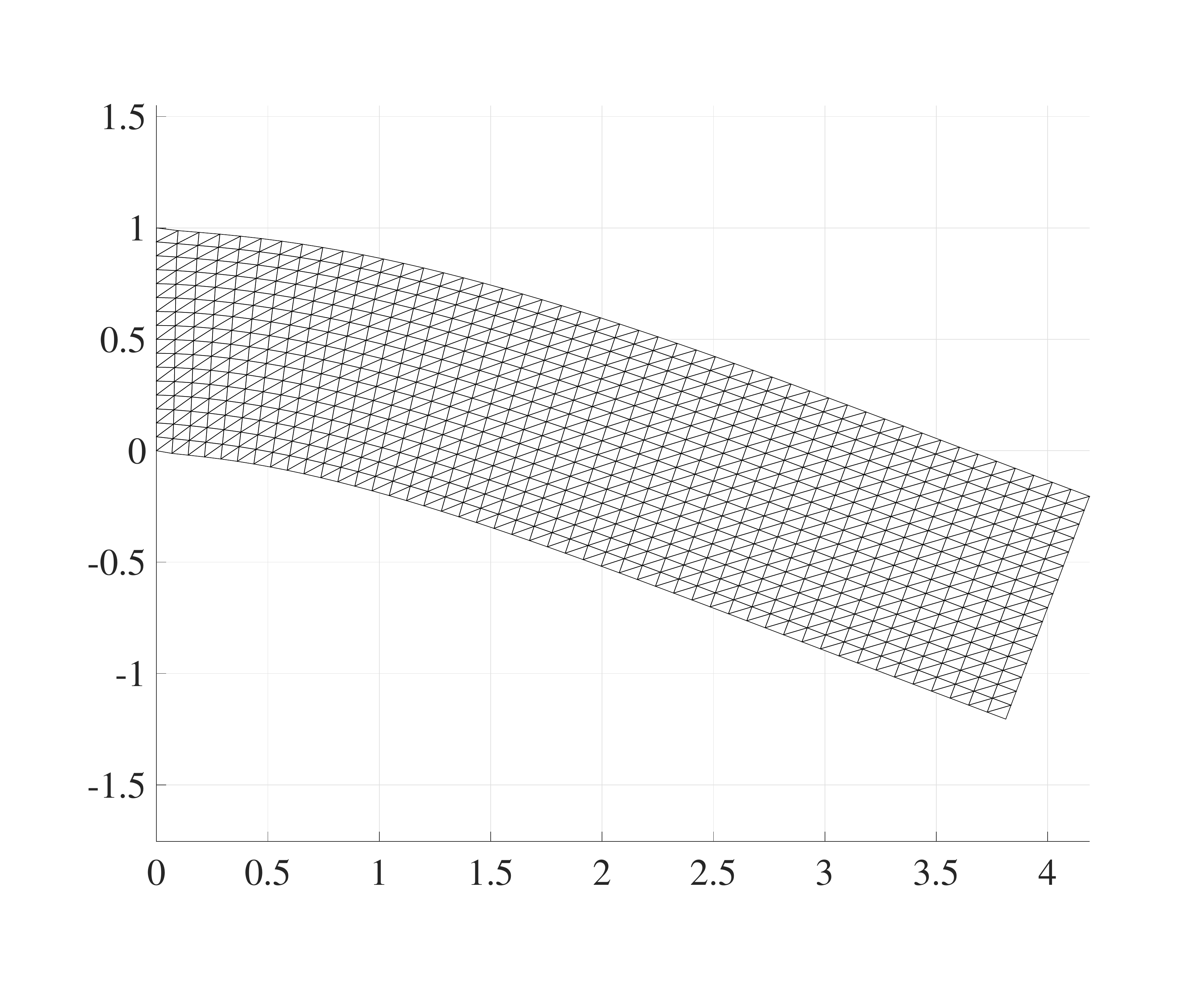}
\caption{Deformation of the bulk material.\label{bulkdef}}
\end{figure}
\begin{figure}
\centering
\includegraphics[width=0.7\linewidth]{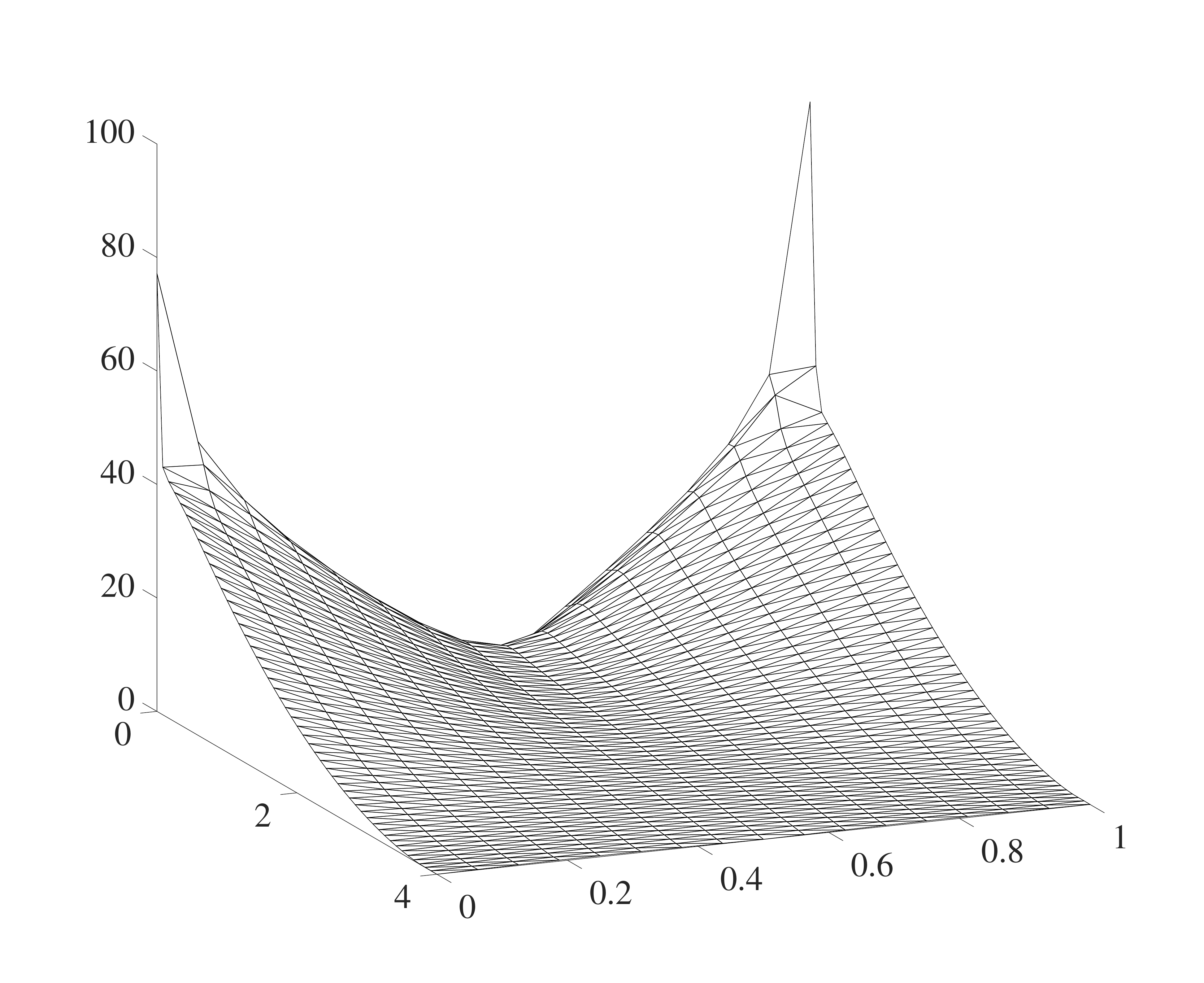}
\caption{Norm of the stresses in the bulk.\label{bulkmises}}
\end{figure}
\begin{figure}
\centering
\includegraphics[width=0.7\linewidth]{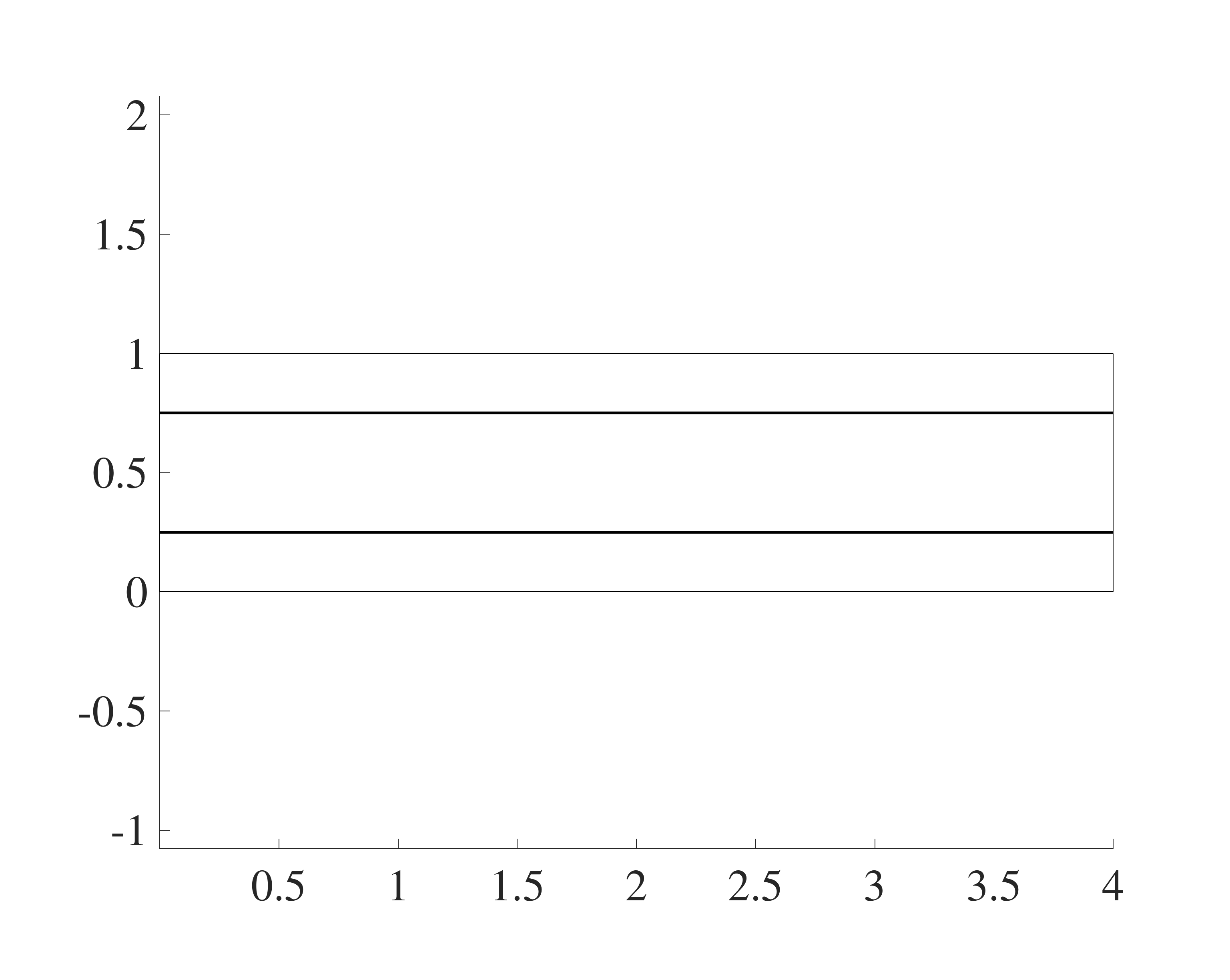}
\caption{Trusses in the bulk.\label{trusses}}
\end{figure}
\begin{figure}
\centering
\includegraphics[width=0.7\linewidth]{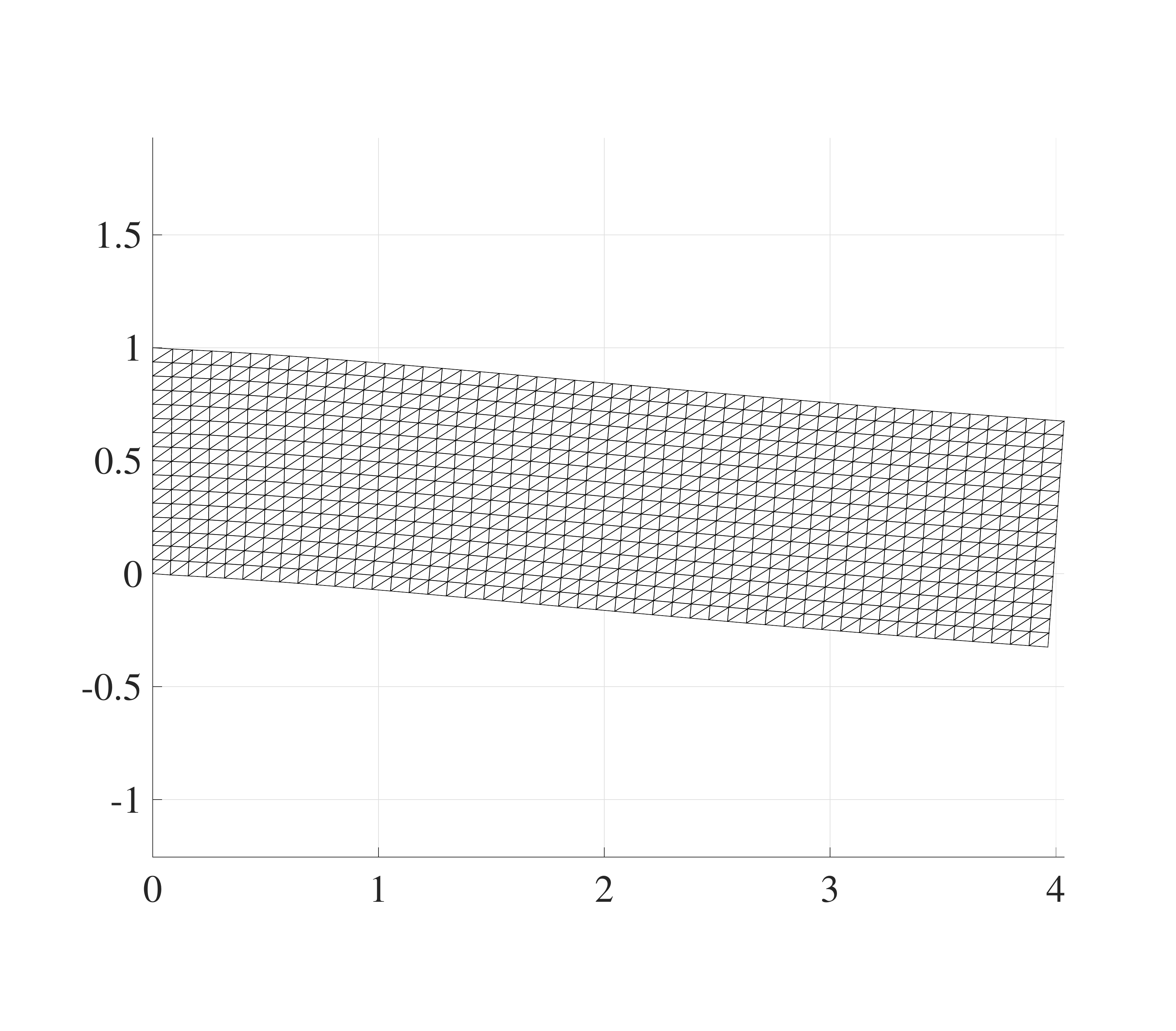}
\caption{Deformation with trusses.\label{trussdef}}
\end{figure}
\begin{figure}
\centering
\includegraphics[width=0.7\linewidth]{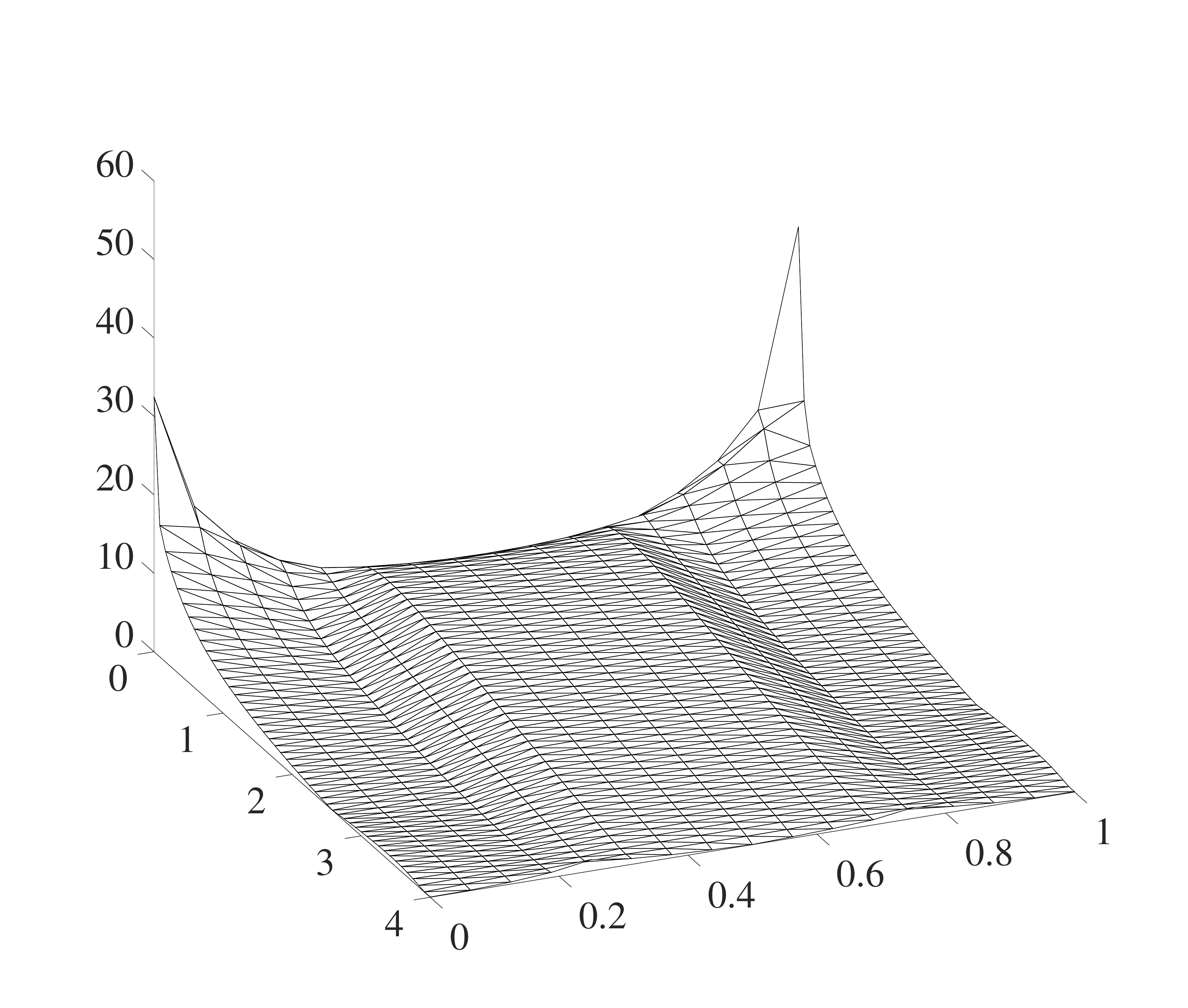}
\caption{Norm of the stresses with trusses.\label{trussmises}}
\end{figure}
\begin{figure}
\centering
\includegraphics[width=0.7\linewidth]{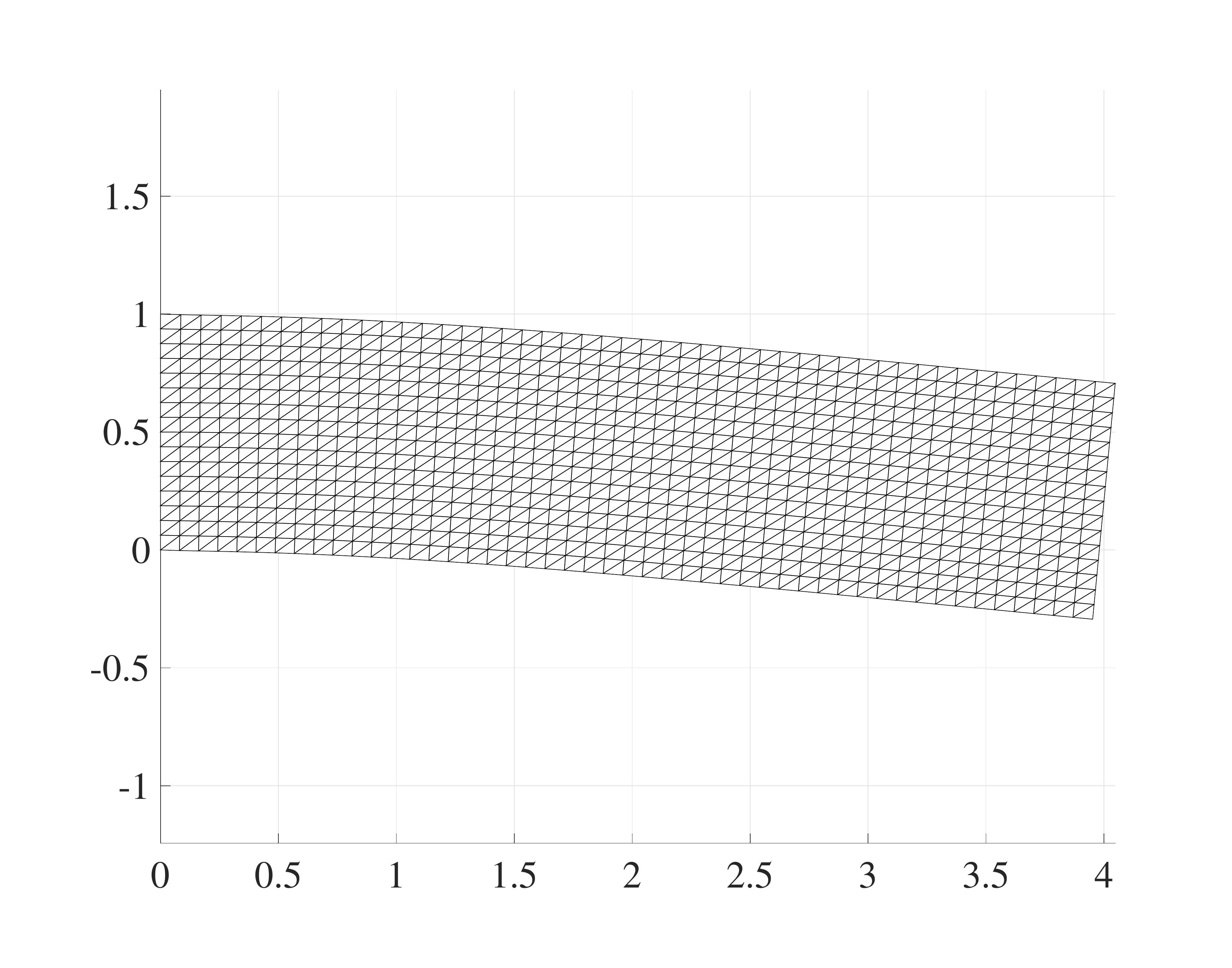}
\caption{Deformation with beam.\label{beamdef}}
\end{figure}
\begin{figure}
\centering
\includegraphics[width=0.7\linewidth]{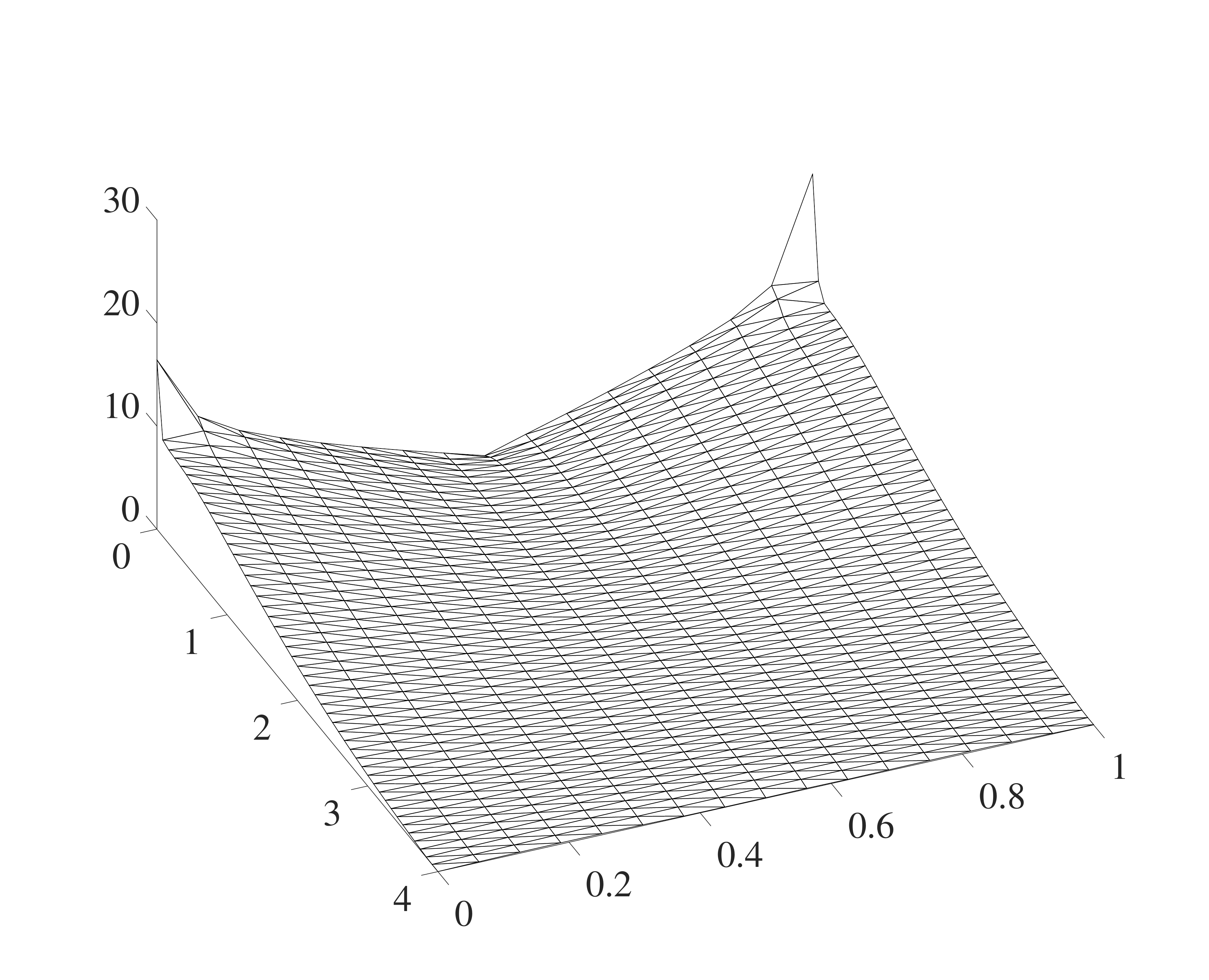}
\caption{Norm of the stresses with beam.\label{beammises}}
\end{figure}
\begin{figure}
\centering
\includegraphics[width=0.7\linewidth]{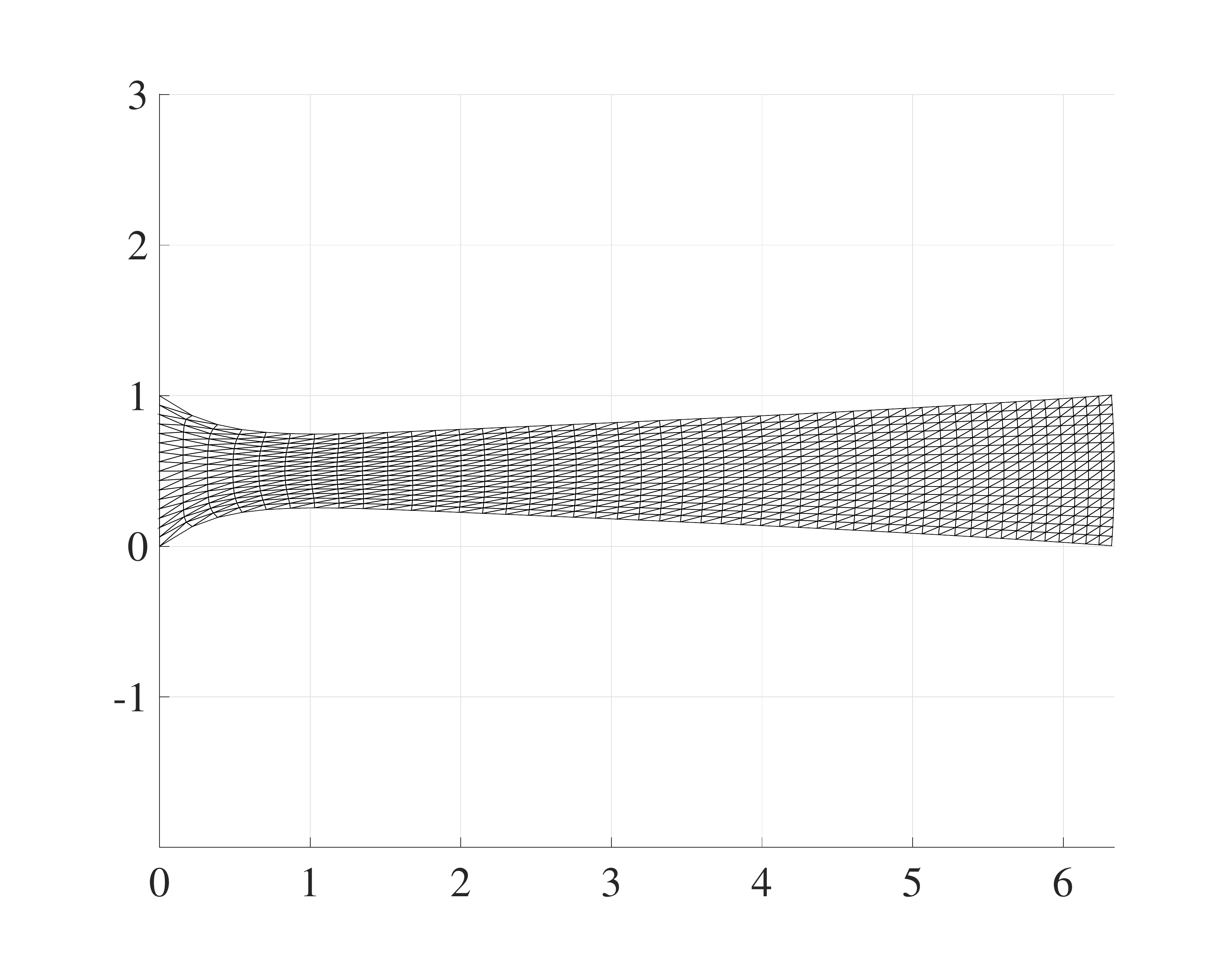}
\caption{Deformation of pulled bulk material.\label{pulledbulk}}
\end{figure}
\begin{figure}
\centering
\includegraphics[width=0.7\linewidth]{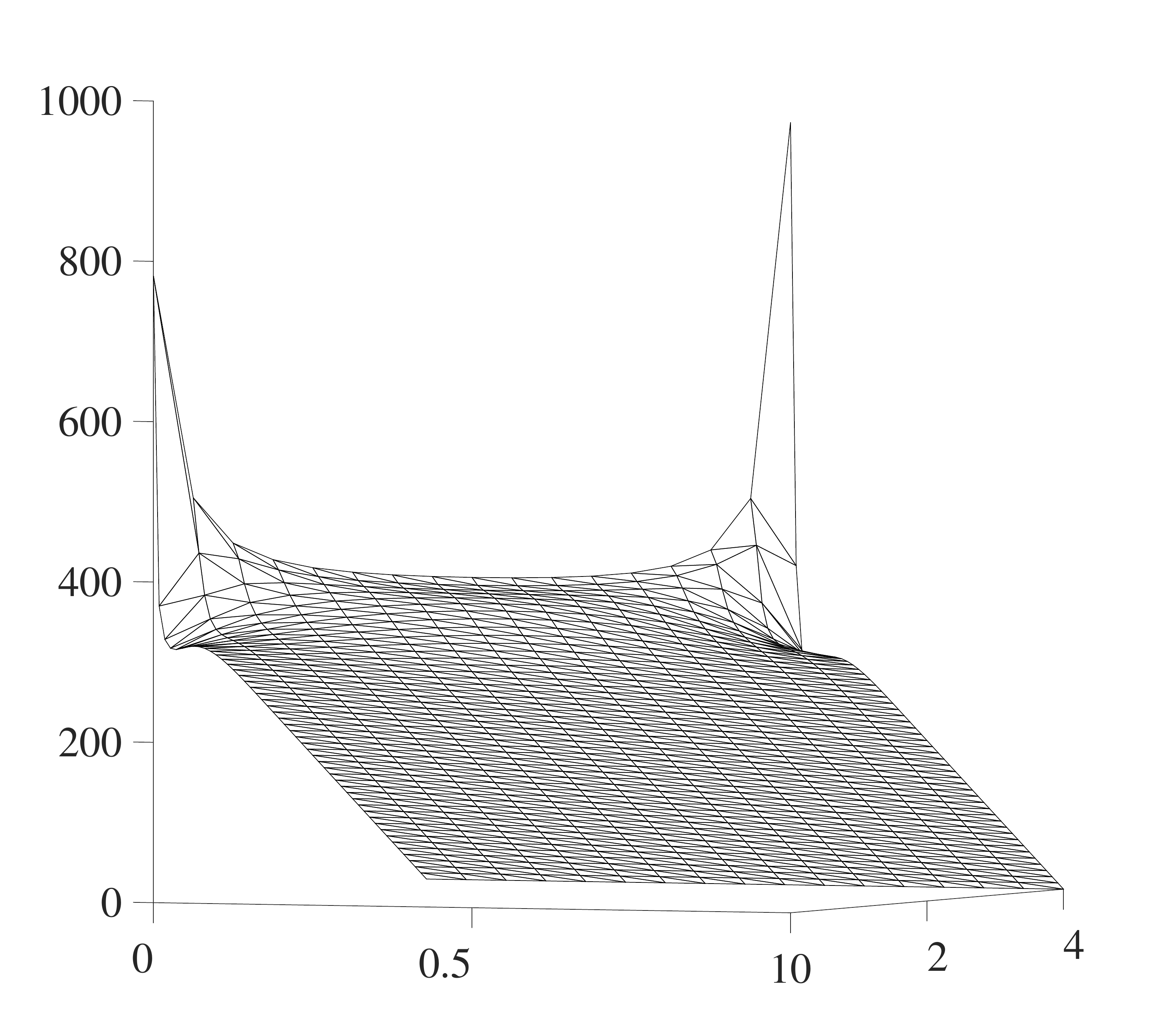}
\caption{Norm of the stresses; pulled bulk material.\label{pulledbulkmises}}
\end{figure}
\begin{figure}
\centering
\includegraphics[width=0.7\linewidth]{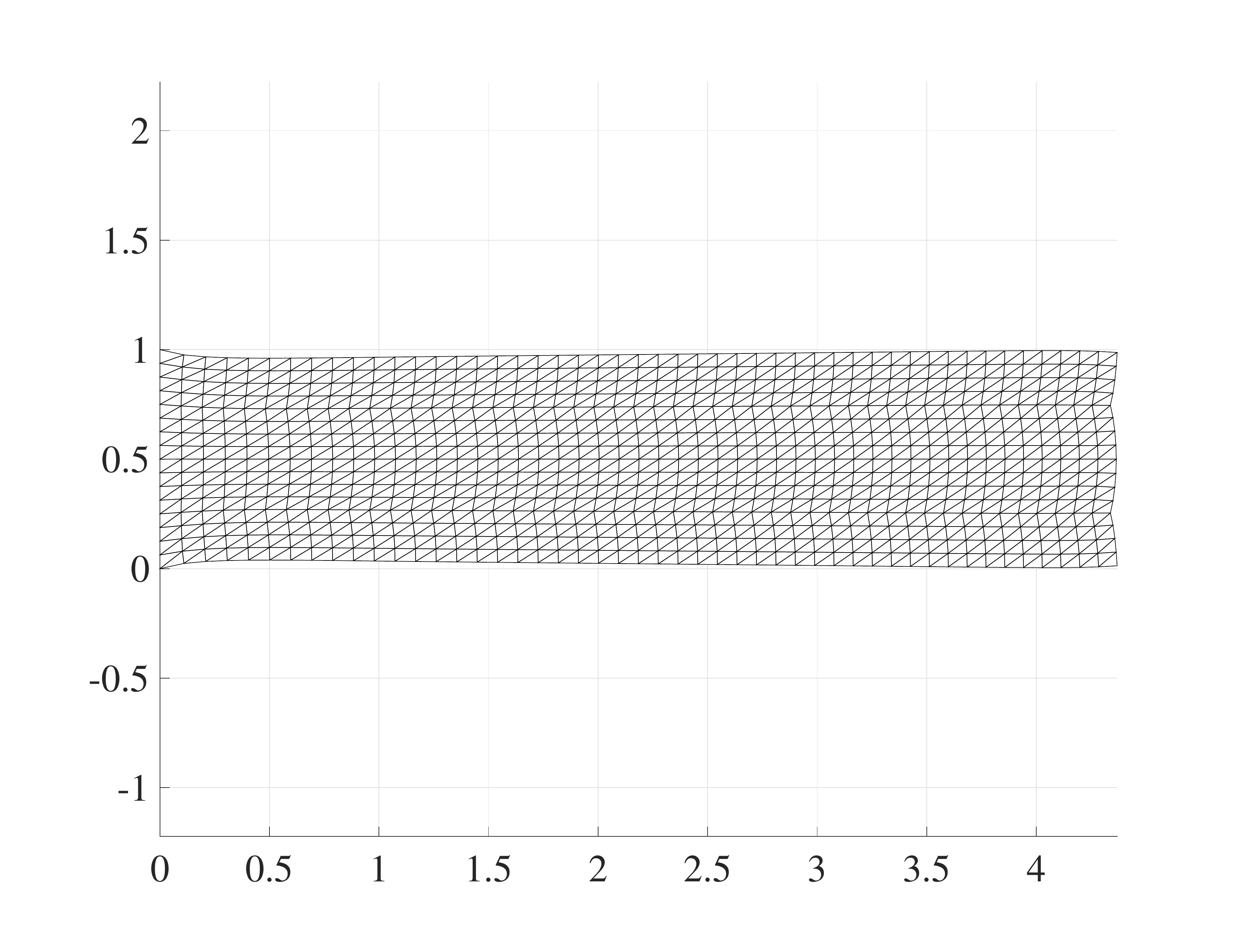}
\caption{Deformation of pulled bulk material with trusses.\label{pulledtruss}}
\end{figure}
\begin{figure}
\centering
\includegraphics[width=0.7\linewidth]{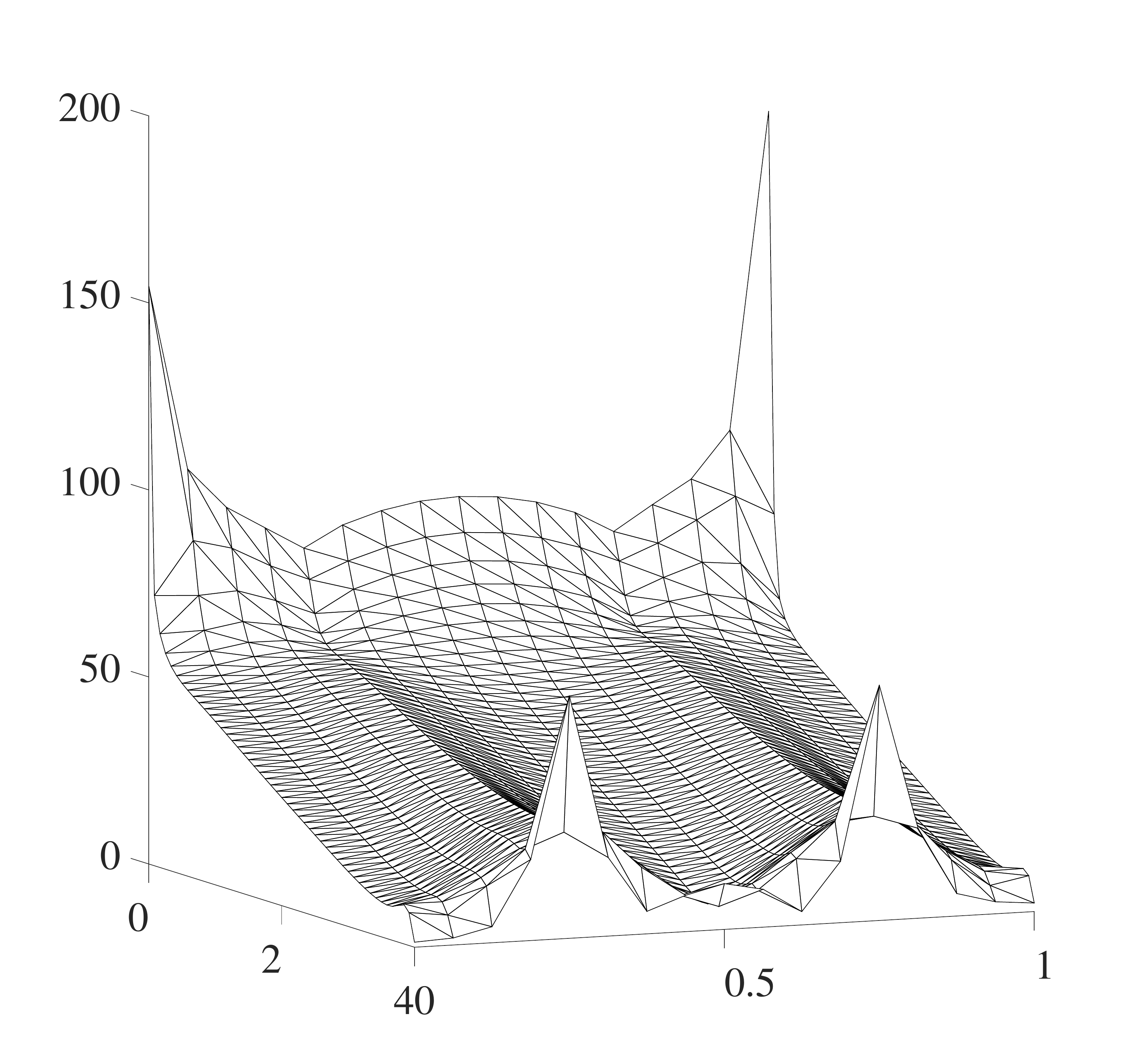}
\caption{Norm of the stresses; pulled bulk material with trusses.\label{pulledtrussmises}}
\end{figure}

\clearpage
\bibliographystyle{abbrv}
\bibliography{CutElasticity}

\end{document}